\documentclass[english]{amsart}

\usepackage{amsmath}

\usepackage{amssymb}

\usepackage{amscd} \usepackage{tabularx}
\usepackage[all,cmtip,line]{xy}
\usepackage{enumerate}

\newcommand{\ra}{\rightarrow}
\newcommand{\lra}{\longrightarrow}

\newcommand{\into}{\hookrightarrow}

\newlength{\ownl}

\newcommand{\ndiv}{\nmid}

\newcommand{\Art}{{\operatorname{Art}\,}}

\newcommand{\Fil}{{\operatorname{Fil}\,}}

\newcommand{\Frob}{{\operatorname{Frob}}}
\newcommand{\Gal}{{\operatorname{Gal}\,}}

\newcommand{\Hom}{{\operatorname{Hom}\,}}

\newcommand{\Ind}{{\operatorname{Ind}\,}}

\newcommand{\Res}{{\operatorname{Res}}}

\newcommand{\WD}{{\operatorname{WD}}}

\newcommand{\Spec}{{\operatorname{Spec}\,}}

\newcommand{\ad}{{\operatorname{ad}\,}}

\newcommand{\gr}{{\operatorname{gr}\,}}

\newcommand{\rec}{{\operatorname{rec}}}

\newcommand{\diag}{{\operatorname{diag}}}

\newcommand{\Wdiag}{W^{\diag}}
\newcommand{\Wcris}{W^{\cris}}
\newcommand{\WBDJ}{W^{\operatorname{BDJ}}}
\newcommand{\WSch}{W^{\operatorname{Sch}}}
\newcommand{\WGHS}{W^{\operatorname{GHS}}}
\newcommand{\Wexplicit}{W^{\operatorname{explicit}}}

\newcommand{\cris}{{\operatorname{cris}}}

\newcommand{\ab}{{\operatorname{ab}}}

\newcommand{\univ}{{\operatorname{univ}}}

\newcommand{\A}{{\mathbb{A}}}

\newcommand{\C}{{\mathbb{C}}}

\newcommand{\F}{{\mathbb{F}}}

\newcommand{\Q}{{\mathbb{Q}}}
\newcommand{\R}{{\mathbb{R}}}

\newcommand{\Z}{{\mathbb{Z}}}

\newcommand{\CC}{{\mathcal{C}}}
\newcommand{\calD}{{\mathcal{D}}}

\newcommand{\CG}{{\mathcal{G}}}

\newcommand{\CO}{{\mathcal{O}}}

\newcommand{\CS}{{\mathcal{S}}}

\newcommand{\gothsl}{{\mathfrak{sl}}}
\newcommand{\gothgl}{{\mathfrak{gl}}}

\newcommand{\gm}{{\mathfrak{m}}}

\newcommand{\barF}{\overline{{F}}}

\newcommand{\barK}{\overline{{K}}}

\newcommand{\barFF}{\overline{{\F}}}

\newcommand{\barQQ}{\overline{{\Q}}}

\newcommand{\barr}{\overline{{r}}}

\newcommand{\tI}{\widetilde{{I}}}

\newcommand{\tS}{\widetilde{{S}}}

\newcommand{\tu}{\widetilde{{u}}}
\newcommand{\tv}{{\widetilde{{v}}}}

\newcommand{\epsilonbar    }{\overline{\epsilon}}

 \newcommand{\bartheta    }{\overline{\theta}}

 \newcommand{\barmu    }{\overline{\mu}}

 \newcommand{\barrho   }{{\overline{\rho}}}

 \newcommand{\tmu    }{\widetilde{\mu}}

\newcommand{\Qbar}{{\overline{\Q}}}
\newcommand{\Fbar}{{\overline{\F}}}

\def\RCS$#1: #2 ${\expandafter\def\csname RCS#1\endcsname{#2}}
\RCS$Revision: 137 $
\RCS$Date: 2011-06-12 11:12:31 -0500 (Sun, 12 Jun 2011) $

\newcommand{\notdiv}{\nmid}

\newcommand{\To}{\longrightarrow}\newcommand{\m}{\mathfrak{m}}

\newcommand{\onto}{\twoheadrightarrow}
\newcommand{\isoto}{\stackrel{\sim}{\To}}

\newcommand{\bigO}{\mathcal{O}}

\newcommand{\Favoid}{F^{(\mathrm{avoid})}}

\newcommand{\thetabar}{\bar{\theta}}
\newcommand{\taubar}{\overline{\tau}}

\newcommand{\bb}{\mathbb} 
\newcommand{\mc}{\mathcal}
\newcommand{\mf}{\mathfrak}

\newcommand{\cA}{\mathcal{A}}

\newcommand{\cG}{\mathcal{G}}

\newcommand{\cO}{\mathcal{O}}

\DeclareMathOperator{\FL}{FL}

\newcommand{\rhobar}{\overline{\rho}} 
\newcommand{\chibar}{\overline{\chi}} 
\newcommand{\rbar}{\bar{r}}

\newcommand{\mubar}{\overline{\mu}}

\newcommand{\GL}{\operatorname{GL}}
\newcommand{\SO}{\operatorname{SO}}
\newcommand{\SU}{\operatorname{SU}}

\newcommand{\PGL}{\operatorname{PGL}}
\newcommand{\HT}{\operatorname{HT}}

 \newcommand{\Qp}{\Q_p}
\newcommand{\Ql}{\Q_l} 

\newcommand{\Qlbar}{{\overline{\Q}_{l}}}

\newcommand{\Fl}{{\F_l}}
\newcommand{\Flbar}{{\overline{\F}_l}}

\newcommand{\SL}{\operatorname{SL}}
\newcommand{\PSL}{\operatorname{PSL}}

\newcommand{\Sym}{\operatorname{Sym}}

\usepackage{amsthm}
\usepackage{hyperref}

 \newtheorem{ithm}{Theorem}
\newtheorem{thm}{Theorem}[subsection]

\newtheorem{cor}[thm]{Corollary}
 \newtheorem{lemma}[thm]{Lemma}
\newtheorem{lem}[thm]{Lemma} \newtheorem{prop}[thm]{Proposition}
\newtheorem{conj}[thm]{Conjecture} \theoremstyle{definition}
 \theoremstyle{definition}
\newtheorem{defn}[thm]{Definition} \theoremstyle{remark}
\newtheorem{rem}[thm]{Remark} 
\newtheorem{remark}[thm]{Remark} 
 
\numberwithin{equation}{subsection}

\theoremstyle{definition}

\newtheorem*{sublem}{Sublemma}

\begin{document}
\title{Serre weights for rank two unitary groups.}

\author{Thomas Barnet-Lamb}\email{tbl@brandeis.edu}\address{Department of Mathematics, Brandeis University}
\author{Toby Gee} \email{gee@math.northwestern.edu} \address{Department of
  Mathematics, Northwestern University}\author{David Geraghty}
\email{geraghty@math.princeton.edu}\address{Princeton University and
  Institute for Advanced Study}  \thanks{The second author was partially supported
  by NSF grant DMS-0841491, and the third author was partially supported
  by NSF grant DMS-0635607.}  
\begin{abstract}We study the weight part of (a generalisation of)
  Serre's conjecture for mod $l$ Galois representations associated to
  automorphic representations on rank two unitary groups for odd
  primes $l$. We propose a conjectural set of Serre weights, agreeing
  with all conjectures in the literature, and under a mild assumption
  on the image of the mod $l$ Galois representation we are able to
  show that any modular representation is modular of each conjectured
  weight. We make no assumptions on the ramification or inertial
  degrees of $l$. Our main innovation is to make use of the lifting
  techniques introduced in \cite{blgg}, \cite{blggord} and
  \cite{BLGGT}.
\end{abstract}
\maketitle
\tableofcontents

\section{Introduction.}\label{sec:intro}

\subsection{} In recent years there has been considerable progress on
proving generalisations of the weight part of Serre's conjecture for
mod $l$ representations corresponding to automorphic representations
of $\GL_2$. Such a generalisation was initially formulated in
\cite{bdj}, for Hilbert modular forms over a totally real field $F^+$
in which $l$ is unramified, and was largely proved in \cite{geebdj}. A
generalisation of the conjecture of \cite{bdj} for tamely ramified
Galois representations was proposed in \cite{MR2430440}, and in the
case that $l$ is totally ramified in $F^+$ this conjecture was mostly
proved in \cite{geesavitttotallyramified}. In his forthcoming
University of Arizona PhD thesis, Ryan Smith uses essentially the same
argument to prove some cases when the inertial and ramification
indexes are both two.

While these results represent a considerable advance on our
understanding of 2-dimensional mod $l$ Galois representations, they
are limited in several respects. Firstly, it seems to be hopeless to
expect to be able to push the methods of proof to work over a general
totally real
field. This is not merely aesthetically unsatisfactory; it also limits
the applicability of the results, for example limiting the options of
combining them with base change techniques, or of applying them to
generalisations of the arguments of Khare and Wintenberger which
proved Serre's conjecture over $\Q$. Secondly, the techniques of
\cite{geebdj} do not allow one to prove results for all weights, but
only for weights which are sufficiently regular; in applications, for
example to modularity lifting theorems and the Breuil-Mezard
conjecture (cf. \cite{kisinICM}), one often needs a result for all
weights. Finally, the methods employed in these earlier papers entail
some exceedingly unpleasant combinatorial and $p$-adic Hodge theoretic
calculations.

In the present paper we resolve most of these difficulties, proving a
very general theorem about the weight part of Serre's conjecture for
rank two unitary groups. These groups are outer forms of $\GL_2$ over
totally real fields, as opposed to the inner forms studied in the
papers discussed above. We choose to use these groups for two
reasons. Firstly, we have developed a considerable body of material on
automorphy lifting theorems for these groups in our recent work
(\cite{blgg}, \cite{blggord}, \cite{BLGGT}). Secondly, the
relationship between the weights of mod $l$ Galois representations and
$l$-adic Galois representations is simpler than for the inner forms,
because there is no obstruction coming from the units in the
totally real field (this can already be seen for $\GL_1$: one has
considerably more flexibility to choose the weights of an algebraic
character over an imaginary CM field than over a totally real field).

Our main theorem is as follows (see Theorem \ref{thm: explicit local lifts implies Serre
    weight}). Given a modular representation $\rbar$, we define a set
  of Serre weights $\Wexplicit(\rbar)$, which is the set of predicted
  weights for $\rbar$ from the papers \cite{bdj}, \cite{MR2430440} and \cite{GHS}.
\begin{ithm}Let $F$ be an imaginary CM field with maximal totally real subfield
  $F^+$. Assume that $\zeta_l\notin F$, that $F/F^+$ is unramified at all finite places,
  that every place of $F^+$ dividing $l$ splits completely in $F$,
  and that $[F^+:\Q]$ is even. Suppose that $l>2$, and that
  $\rbar:G_F\to\GL_2(\Flbar)$ is an irreducible modular
  representation with split ramification. Assume that $\rbar(G_{F(\zeta_l)})$ is adequate.
 
 Let $a$ be a
  Serre weight. Assume that $a\in \Wexplicit(\rbar)$. Then $\rbar$ is
  modular of weight $a$. 
  \end{ithm}
(See Sections \ref{sec:Definitions} and \ref{sec:Conjectures} for any
unfamiliar terminology.) Note in particular that if $l\ge 7$, the
hypothesis that  $\rbar(G_{F(\zeta_l)})$ is adequate may be replaced
by the usual Taylor-Wiles assumption that $\rbar|_{G_{F(\zeta_l)}}$ is irreducible.

Our approach is related to that of \cite{geebdj}, in that we prove
that a mod $l$ Galois representation is modular of a given weight by
producing $l$-adic lifts with certain properties. In \cite{geebdj} we
were forced to work with potentially Barsotti-Tate lifts, due to our
dependence on the modularity lifting theorems proved in \cite{kis04}
and \cite{MR2280776}. This led to much of the combinatorial
difficulties mentioned above, which in turn limited us to working over
a totally real field in which $l$ is unramified. Thanks to the
techniques developed in our previous papers, and in particular the
lifting theorems proved in \cite{BLGGT}, in the present paper we are
able to produce lifts of arbitrary weight. This completely removes the
combinatorial difficulties, as we now explain.

Let $F$ be an imaginary CM field with maximal totally real subfield
$F^+$. Assume that $F/F^+$ is unramified at all finite places and
split at all places lying over $l$, and that $[F^+:\Q]$ is even. In
section \ref{sec:Definitions} below we define a certain rank two
unitary group $G$ over $F^+$, which is compact at all infinite places and
quasisplit at all finite places, and split over $F$. It is thus split
at all places dividing $l$, so there is a natural notion of a Serre
weight $a$, which is an irreducible representation of the product of
the $\GL_2(k_v)$, where $v$ runs over the places of $F$ dividing
$l$. We have a notion of an irreducible mod $l$ Galois representation
$\rbar:G_F\to\GL_2(\Flbar)$ being modular of some Serre weight, in
terms of algebraic modular forms on $G$. An elementary, but extremely
useful, fact is that any Serre weight $a$ can be lifted to a
characteristic $0$ weight $\lambda$ (that is, to an irreducible
algebraic representation of $\GL_2(\cO_{F^+,l})$). Since $G$ is
compact, it is easy to check that $\rbar$ being modular of weight $a$
is equivalent to $\rbar$ having a lift which corresponds to an
automorphic representation of weight $\lambda$ and level prime to $l$,
and by the theory of base change this is equivalent to $\rbar$ having
a lift which corresponds to a conjugate-self dual automorphic
representation of $\GL_2(\A_F)$ of weight $\lambda$ and level prime to
$l$.

The weight part of Serre's conjecture thus reduces to a question about
the existence of automorphic lifts of $\rbar$ with specific local
properties; the condition that the corresponding automorphic
representation has weight $\lambda$ and level prime to $l$ translates
to the condition that the Galois representation be crystalline with
Hodge-Tate weights determined by $\lambda$. This gives an obvious
necessary condition for $\rbar$ to be modular of weight $a$: for each
place $v|l$ of $F$, $\rbar|_{G_{F_v}}$ must have a crystalline lift of
the appropriate Hodge-Tate weights. Following \cite{gee061}, we
conjecture that this condition is also sufficient.

Our main result in this direction is that, subject to mild hypotheses
on the image $\rbar(G_F)$, if $\rbar$ is assumed to be modular and if for each place $v|l$ of $F$,
$\rbar|_{G_{F_v}}$ has a potentially diagonalizable crystalline lift
of the appropriate Hodge-Tate weights, then $\rbar$ is modular of
weight $a$. We refer the reader to section \ref{sec:A lifting theorem}
for the definition of the term ``potentially diagonalizable'', which
was introduced in \cite{BLGGT}. This result is a straightforward
consequence of the above discussion and the results of
\cite{BLGGT}, together with the results of \cite{kis04} and
\cite{MR2280776} (which show that $\rbar$ necessarily has some
automorphic lift which is potentially diagonalizable).

Since we do not know if every crystalline representation is
potentially diagonalizable, it is not immediately clear how useful the
above result is. Accordingly, we examine the explicit conjectures made
in \cite{bdj}, \cite{MR2430440} and \cite{GHS}, and note that in
(almost) every case, whenever
the conjectures made in those papers suggest that $\rbar$ should be
modular of weight $a$, we can find potentially diagonalizable
crystalline lifts of the correct Hodge-Tate weights. Indeed, we can
find potentially diagonalizable lifts of a particularly simple kind:
they are either an extension of two characters, or are induced from a
character.

Accordingly, we have reduced the weight part of Serre's conjecture in
this setting to a purely local question, of determining whether if a
mod $l$ Galois representation has a crystalline lift with specified
Hodge-Tate weights (constrained to lie in a particular range), it has
one which is furthermore potentially diagonalizable. We strongly
suspect that this question has an affirmative answer. In the
2-dimensional cases at hand, this is presumably accessible via a brute
force calculation in integral $p$-adic Hodge theory. We have not
attempted such a calculation, as we expect that it would be lengthy
and unenlightening. We do, however, completely determine the list of
weights when the absolute ramification index of each prime $v$ of $F$
dividing $l$ is at least $l$, and for each such $v$ the representation
$\rbar|_{G_{F_v}}$ is semisimple. Note that one can always reduce to
this case by base change, which may make this result particularly
valuable in applications.
We remark that some of the above discussion carries over to rank $n$
unitary groups for arbitrary $n$. However, there are several
difficulties with obtaining results as strong as those obtained
here. Firstly, the correspondence between weights in characteristic 0
and characteristic $l$ is less simple: there are irreducible
$\Flbar$-representations of $\GL_n(\F_{l})$ which do not lift to
irreducible $\Qlbar$-representations. Secondly, we do not know that
every modular $\rbar$ has an automorphic lift which is potentially
diagonalizable. Nonetheless, our methods give non-trivial results for
general $n$, which we will explain in a subsequent paper.

We now explain the structure of this paper. In section
\ref{sec:Definitions}, we define the unitary groups that we use, and
recall some basic facts about the automorphic representations and Galois
representations that we use. In section \ref{sec:A lifting theorem} we
deduce the main lifting theorem that we need from the results of
\cite{BLGGT}. In section \ref{sec:Conjectures} we explain the explicit
Serre weight
conjectures in the literature, and write down various explicit
potentially diagonalizable representations. In Section \ref{sec:main
  theorem} we deduce our main explicit theorems. Finally, in Appendix
\ref{app:adequacy} we discuss the adequate subgroups of
$\GL_2(\Flbar)$ for $l=3$ and $l=5$, and we improve on a result of
\cite{BLGGT}; this section allows us to treat the cases $l=3$, $5$ in
this paper, whereas a direct appeal to the results of \cite{BLGGT}
would force us to assume that $l\ge 7$.

We would like to thank Florian Herzig for his helpful comments on an
earlier draft of this paper.

\subsection{Notation and conventions}If $M$ is a field, we let $G_M$ denote its absolute Galois group.  We
write all matrix transposes on the left; so ${}^tA$ is the transpose
of $A$. Let $\epsilon_l$ denote the $l$-adic cyclotomic character, and
$\bar{\epsilon}_l$ or $\omega_l$ the mod $l$ cyclotomic character. If $M$ is
a finite extension of $\bb{Q}_p$ for some prime $p$, we write $I_M$ for the
inertia subgroup of $G_M$. If $M$ and $K$ are algebraic extensions of
$\Q_p$, then all homomorphisms $M\to K$ are assumed to be continuous
for the $p$-adic topology. If $R$ is a local ring we write
$\mf{m}_{R}$ for the maximal ideal of $R$.

If $K$ is a finite extension of $\Qp$, we will let $\rec_K$ be the local Langlands correspondence of
\cite{ht}, so that if $\pi$ is an irreducible
admissible complex representation of $\GL_n(K)$, then $\rec_K(\pi)$ is a
Weil-Deligne representation of the Weil group $W_K$. We will write $\rec$ for $\rec_K$
when the choice of $K$ is clear. We write $\Art_K:K^\times\to W_K$ for
the isomorphism of local class field theory, normalised so that
uniformisers correspond to geometric Frobenius elements. If $(V,r,N)$
is a Weil-Deligne representation of $W_K$ over some algebraically
closed field of characteristic zero, then we define its Frobenius
semisimplification $(V,r,N)^{F-ss}$ (resp. its semisimplification
$(V,r,N)^{ss}$) as in section 1 of \cite{ty}.

Let $W$ be a
continuous finite-dimensional representation of $G_K$ over
$\Qlbar$ for some prime $l$. If $p=l$, assume that $W$ is de Rham. Then we denote by
$\WD(W)$ the Weil-Deligne representation associated to $W$. Assume
now that $p=l$. If $\tau:K \into \barQQ_l$ is a continuous embedding, then by definition the multiset
$\HT_\tau(W)$ of Hodge-Tate weights of $W$ with respect to $\tau$ contains $i$ with multiplicity $\dim_{\barQQ_l} (W
\otimes_{\tau,K} \widehat{\barK}(i))^{G_K} $. Thus for example
$\HT_\tau(\epsilon_l)=\{ -1\}$.

\section{Definitions}\label{sec:Definitions}\subsection{}Let $l>2$ be a prime, and let $F$ be an imaginary CM field with
maximal totally real field subfield $F^+$. We assume throughout this
paper that:
\begin{itemize}
\item $F/F^+$ is unramified at all finite places.
\item Every place $v|l$ of $F^+$ splits in $F$.
\item $[F^+:\Q]$ is even.
\end{itemize}
Under these hypotheses, there is a reductive algebraic group $G/F^+$
with the following properties:
\begin{itemize}
\item $G$ is an outer form of $\GL_2$, with $G_{/F}\cong\GL_{2/F}$.
\item If $v$ is a finite place of $F^+$, $G$ is quasi-split at $v$.
\item If $v$ is an infinite place of $F^+$, then $G(F^+_v)\cong U_2(\R)$.
\end{itemize}
To see that such a group exists, one may argue as follows. Let $B$
denote the matrix algebra $M_2(F)$. An involution $\ddag$ of the
second kind on $B$ gives a reductive group $G_\ddag$ over $F^+$ by
setting \[G_\ddag(R)=\{g\in B\otimes_{F^+} R:g^\ddag g=1\}\] for any
$F^+$-algebra $R$. Any such $G_\ddag$is an outer form of $\GL_2$,
with $G_{\ddag/F}\cong\GL_{2/F}$. One can choose $\ddag$ such that
\begin{itemize}
 \item If $v$ is a finite place of $F^+$, $G_\ddag$ is quasi-split at $v$.
\item If $v$ is an infinite place of $F^+$, then $G_\ddag(F^+_v)\cong U_2(\R)$.
\end{itemize}
To see this, one uses the argument of Lemma I.7.1 of \cite{ht}; it is
here that we require the hypotheses that $F/F^+$ is unramified at all
finite places, and that $[F^+:\Q]$ is even. We then fix some choice of
$\ddag$ as above, and take $G=G_\ddag$.

As in section 3.3 of \cite{cht} we define a model for $G$ over
$\cO_{F^+}$ in the following way. We choose an order $\cO_B$ in $B$
such that $\cO_B^\ddag=\cO_B$, and $\cO_{B,w}$ is a maximal order in
$B_w$ for all places $w$ of $F$ which are split over $F^+$ (see
section 3.3 of \cite{cht} for a proof that such an order exists). Then
we can define $G$ over $\cO_{F^+}$ by setting \[G(R)=\{g\in
\cO_B\otimes_{\cO_{F^+}} R:g^\ddag g=1\}\] for any $\cO_{F^+}$-algebra $R$.

If $v$ is a place of $F^+$ which splits as $ww^c$ over $F$, then we
choose an isomorphism \[\iota_v:\cO_{B,v}\isoto
M_2(\cO_{F,v})=M_2(\cO_{F_w})\oplus M_2(\cO_{F_{w^c}})\] such that
$\iota_v(x^\ddag)={}^t\iota_v(x)^c$. This gives rise to an
isomorphism \[\iota_w:G(\cO_{F_v^+})\isoto \GL_2(\cO_{F_w})\] sending
$\iota_v^{-1}(x,{}^tx^{-c})$ to $x$.

Let $K$ be an algebraic extension of $\Ql$ in $\Qlbar$ which contains
the image of every embedding $F\into\Qlbar$, let $\cO$ denote the ring
of integers of $K$, and let $k$ denote the residue field of $K$. Let
$S_l$ denote the set of places of $F^+$ lying over $l$, and for each
$v\in S_l$ fix a place $\tv$ of $F$ lying over $v$. Let $\tilde{S}_l$
denote the set of places $\tilde{v}$ for $v\in S_l$.

Let $W$ be an $\cO$-module with an action of $G(\cO_{F^+,l})$, and let
$U\subset G(\A_{F^+}^\infty)$ be a compact open subgroup with the
property that for each $u\in U$, if $u_l$ denotes the projection of
$u$ to $G(F_l^+)$, then $u_l\in G(\cO_{F^+_l})$. Let $S(U,W)$ denote
the space of algebraic modular forms on $G$ of level $U$ and weight
$W$, i.e. the space of functions \[f:G(F^+)\backslash
G(\A_{F^+}^\infty)\to W\] with $f(gu)=u_l^{-1}f(g)$ for all $u\in U$.

Let $\tI_l$ denote the set of embeddings $F\into K$ giving rise to a
place in $\tS_l$. For any $\tv\in\tS_l$, let $\tI_\tv$ denote the set
of elements of $\tI_l$ lying over $\tv$. We can naturally identify
$\tI_\tv$ with $\Hom(F_\tv,\Qlbar)$. Let $\Z^2_+$ denote the set of pairs $(\lambda_1,\lambda_2)$ of
integers with $\lambda_1\ge \lambda_2$. If $\Omega$ is an algebraically closed field of
characteristic 0 we write $(\Z^2_+)_0^{\Hom(F,\Omega)}$ for the subset of elements
    $\lambda\in(\Z^2_+)^{\Hom(F,\Omega)}$ such
    that \[\lambda_{\tau,1}+\lambda_{\tau\circ c,2}=0\] for all $\tau$.
Note that we can identify $(\Z^2_+)_0^{\Hom(F,\Qlbar)}$ with $(\Z^2_+)^{\tI_l}$ in a natural fashion.
If $\lambda$ is an element of $(\Z^2_+)^{\tI_l}$ (resp.\
$(\Z^2_+)^{\Hom(F,\Qlbar)}$) and $w \in \tS_l$
(resp.\ $w|l$) is a place of $F$, we define $\lambda_w\in
(\Z^2_+)^{\Hom(F_w,K)}$ to be $(\lambda_\sigma)_\sigma$ with
$\sigma$ running over all embeddings $F\into K$ inducing $w$.

If $w|l$ is a place of $F$ and $\lambda\in(\Z^2_+)^{\Hom(F_w,\Qlbar)}$, let $W_\lambda$
be the free $\cO$-module with an action of $\GL_2(\cO_{F_w})$ given
by \[W_\lambda:=\otimes_{\tau\in\Hom(F_w,\Qlbar)}\det{}^{\lambda_{\tau,2}}\otimes\Sym^{\lambda_{\tau,1}-\lambda_{\tau,2}}\cO_{F_w}^2\otimes_{\cO_{F_w},\tau}\cO.\]
If $v=w|_{F^+}$, we give this an action of $G(\cO_{F^+,v})$ via $\iota_w$.
If $\lambda\in(\Z^2_+)^{\tI_l}$, we let $W_\lambda$
be the free $\cO$-module with an action of $G(\cO_{F^+,l})$ given
by \[W_\lambda:=\otimes_{\tv\in\tS_l} W_{\lambda_\tv}.\]

If $A$ is an $\cO$-module we let \[S_\lambda(U,A):=S(U,W_\lambda\otimes_\cO A).\]

For any compact open subgroup $U$ as above of $G(\A_{F^+}^\infty)$ we may write
$G(\A_{F^+}^\infty)=\coprod_i G(F^+)t_i U$ for some finite set
$\{t_i\}$. Then there is an isomorphism \[S(U,W)\to\oplus_i W^{U\cap
  t_i^{-1}G(F^+)t_i}\]given by $f\mapsto (f(t_i))_i$. 
We say that
$U$ is \emph{sufficiently small} if for some finite place $v$ of $F^+$
the projection of $U$ to $G(F^+_v)$ contains no element of finite
order other than the identity. Suppose that $U$ is sufficiently
small. Then for each $i$ as above we have
$U\cap t_i^{-1}G(F^+)t_i=\{1\}$, so taking $W=W_\lambda\otimes_\cO A$ we see
that for any $\cO$-algebra $A$, we
have \[S_\lambda(U,A)\cong S_\lambda(U,\cO)\otimes_\cO A.\]
We note when $U$ is not sufficiently small, we still have $S_\lambda(U,A)\cong
S_\lambda(U,\cO)\otimes_\cO A$ whenever $A$ is $\cO$-flat.

We now recall the relationship between our spaces of algebraic
modular forms and the space of automorphic forms on $G$. Write
$S_\lambda(\Qlbar)$ for the direct limit of the spaces $S_\lambda(U,\Qlbar)$ over compact
open subgroups $U$ as above (with the transition maps being the obvious
inclusions $S_\lambda(U,\Qlbar)\subset S_\lambda(V,\Qlbar)$ whenever $V\subset
U$). Concretely, $S_\lambda(\Qlbar)$ is the set of
functions \[f:G(F^+)\backslash G(\A_{F^+}^\infty)\to W_\lambda\otimes_\cO\Qlbar\]
such that there is a compact open subgroup $U$ of
$G(\A_{F^+}^{\infty,l})\times G(\cO_{F^+,l})$
with \[f(gu)=u_l^{-1}f(g)\] for all $u\in U$, $g\in G(\A_{F^+}^\infty)$. This
space has
a natural left action of $G(\A_{F^+}^\infty)$ via \[(g\cdot
f)(h):=g_lf(hg).\]

Fix an isomorphism $\imath:\Qlbar\isoto\C$. For each embedding
$\tau:F^+\into\R$, there is a unique embedding $\tilde{\tau}:F\into\C$
extending $\tau$ such that $\imath^{-1}\tilde{\tau}\in\tI_l$. Let $\sigma_{\imath\lambda}$ denote the representation of $G(F_\infty^+)$ given by
$W_\lambda\otimes_\cO\Qlbar\otimes_{\Qlbar,\imath}\C$, with an element $g\in
G(F^+_\infty)$ acting via
$\otimes_\tau\tilde{\tau}(\iota_{\tilde{\tau}}(g))$. Let $\cA$ denote
the space of automorphic forms on $G(F^+)\backslash G(\A_{F^+})$. From
the proof of Proposition 3.3.2 of \cite{cht}, one easily obtains the following.

\begin{lem}
  \label{lem: relationship of algebraic automorphic forms to classical
    automorphic forms} There is an isomorphism of
  $G(\A_{F^+}^\infty)$-modules \[\imath S_\lambda(\Qlbar)\isoto\Hom_{G(F^+_\infty)}(\sigma_{\imath\lambda}^\vee,\cA).\]
\end{lem}

In particular, we note that $S_\lambda(\Qlbar)$ is a semisimple admissible $G(\A_{F^+}^\infty)$-module.

We now recall from \cite{cht} the notion of a RACSDC automorphic
representation. We say that an automorphic
representation $\pi$ of $\GL_2(\A_F)$ is
\begin{itemize}
\item {\it regular algebraic} if $\pi_\infty$ has the same infinitesimal character as some
  irreducible algebraic representation of $\Res_{F/\Q}\GL_2$;
\item {\it conjugate self dual} if $\pi^c\cong \pi^\vee$.
\end{itemize}
If $\pi$ satisfies both of these properties and is also cuspidal, we
well say that $\pi$ is {\it RACSDC} (regular, algebraic, conjugate self dual and cuspidal).
We say that $\pi$ {\it has level prime to} $l$ if $\pi_v$ is unramified for
all $v|l$. 

If $\lambda\in  (\Z^2_+)^{\Hom(F,\C)}$ we write $\Sigma_\lambda$ for the
irreducible algebraic representation of $\GL_2^{\Hom(F,\C)}\cong \Res_{F/\Q}\GL_2\times_\Q\C$ given by
the tensor product over $\tau$ of the irreducible representations with
highest weights $\lambda_\tau$; i.e. of the
representations \[\det{}^{\lambda_{\tau,2}}\otimes\Sym^{\lambda_{\tau,1}-\lambda_{\tau,2}}\C^2.\]
We say that a RACSDC automorphic representation $\pi$ of
$\GL_2(\A_F)$ {\it has weight} $\lambda\in(\Z^2_+)^{\Hom(F,\C)}$ if $\pi_\infty$ has the same
infinitesimal character as $\Sigma_\lambda^\vee$. If this is the case then
necessarily $\lambda\in(\Z^2_+)_0^{\Hom(F,\C)}$.

\begin{thm}
  \label{thm: existence of Galois reps attached to RACSDC}

If $\pi$ is a RACSDC automorphic representation of $\GL_2(\A_F)$ of
weight $\lambda$, then there is
a continuous irreducible
representation \[r_{l,\imath}(\pi):G_F\to\GL_2(\Qlbar)\] such that
\begin{enumerate}
\item $r_{l,\imath}(\pi)^c\cong
  r_{l,\imath}(\pi)^\vee\otimes\epsilon_l^{-1}$.
\item The representation $r_{l,\iota}(\pi)$ is de Rham, and is
  crystalline if $\pi$ has level prime to $l$. If $\tau:F\into\Qlbar$
  then \[\HT_\tau(r_{l,\imath}(\pi))=\{\lambda_{\imath\tau,1}+1,\lambda_{\imath\tau,2}\}.\]
\item If $v\nmid l$ then \[\imath\WD(r_{l,\imath}(\pi)|_{G_{F_v}})^{F-ss}\cong
  \rec(\pi_v^\vee\otimes|\det|^{-1/2}).\]
\item If $v|l$ then \[\imath\WD(r_{l,\imath}(\pi)|_{G_{F_v}})^{ss}\cong
  \rec(\pi_v^\vee\otimes|\det|^{-1/2})^{ss}.\]
\end{enumerate}

\end{thm}
\begin{proof}
  This follows immediately from the main results of
  \cite{chenevierharris}, \cite{ana} and \cite{blggtlocalglobalII}.
\end{proof}

After conjugating, we may assume that $r_{l,\imath}(\pi)$ takes values
in $\GL_2(\CO_\Qlbar)$. Composing with the map
$\GL_2(\CO_\Qlbar)\to\GL_2(\Flbar)$ and semisimplifying, we obtain a
representation $\rbar_{l,\imath}(\pi) : G_F \to \GL_2(\Flbar)$ which is
independent of any choices made.

We say that a continuous irreducible representation
$r:G_F\to\GL_2(\Qlbar)$ (respectively $\rbar:G_F\to\GL_2(\Flbar)$) is
{\it automorphic} if $r\cong r_{l,\imath}(\pi)$ (respectively $\rbar\cong\rbar_{l,\imath}(\pi)$) for some RACSDC
representation $\pi$ of $\GL_2(\A_F)$. We say that a continuous irreducible representation
$r:G_F\to\GL_2(\Qlbar)$ is {\it automorphic of weight} $\lambda\in
(\Z^2_+)^{\Hom(F,\Qlbar)}$ if $r\cong r_{l,\imath}(\pi)$ for some RACSDC
representation $\pi$ of $\GL_2(\A_F)$ of weight $\imath
\lambda$. By Th\'eor\`eme 3.13 of
\cite{MR1044819}, these notions do not depend on the choice of $\imath$.

The theory of base change gives a close relationship between
automorphic representations of $G(\A_{F^+})$ and automorphic representations of
$\GL_2(\A_F)$. For example, one has the
following consequences of Corollaire 5.3 and Th\'{e}or\`eme 5.4 of
\cite{labesse}.
\begin{thm}\label{thm: base change from GL to G}
  Suppose that $\pi$ is a RACSDC representation of $\GL_2(\A_F)$ of
  weight $\lambda\in(\Z_+^2)_0^{\Hom(F,\C)}$. Then there is an automorphic
  representation $\Pi$ of $G(\A_{F^+})$ such that
  \begin{enumerate}
  \item For each embedding $\tau:F^+\into\R$ and each
    $\hat{\tau}\into\C$ extending $\tau$, we have $\Pi_\tau\cong
    \Sigma_{\lambda_{\hat{\tau}}}^\vee\circ\iota_{\hat{\tau}}$.
  \item If $v$ is a finite place of $F^+$ which splits as $ww^c$ in
    $F$, then $\Pi_v\cong \pi_w\circ\iota_w$.
  \item If $v$ is a finite place of $F^+$ which is inert in $F$, and
    $\pi_v$ is unramified, then $\Pi_v$ has a fixed vector for some
    hyperspecial maximal compact subgroup of $G(F^+_v)$.
  \end{enumerate}

\end{thm}

\begin{thm}
  \label{thm: base change from G to GL} Suppose that $\Pi$ is an
  automorphic representation of $G(\A_{F^+})$.
  Then there is a regular algebraic, conjugate self dual automorphic
  representation $\pi$ of $\GL_2(\A_F)$ of some weight $\lambda\in(\Z^2_+)_0^{\Hom(F,\C)}$ such that either
  \begin{enumerate}[\ \ \ \ \ (a)]
  \item $\pi$ is cuspidal, or
  \item $\pi = \chi_1\boxplus \chi_2$ is the isobaric direct sum of
  characters $\chi_1,\chi_2: F^\times\backslash \A_F^\times\to\C^\times$ 
  \end{enumerate} and in either case we have:
    \begin{enumerate}
    \item For each embedding $\tau:F^+\into\R$ and each
    $\hat{\tau}\into\C$ extending $\tau$, we have $\Pi_\tau\cong
    \Sigma_{\lambda_{\hat{\tau}}}^\vee\circ\iota_{\hat{\tau}}$.
  \item If $v$ is a finite place of $F^+$ which splits as $ww^c$ in
    $F$, then $\Pi_v\cong \pi_w\circ\iota_w$.
  \item If $v$ is a finite place of $F^+$ which is inert in $F$, and
    $\Pi_v$ has a fixed vector for some
    hyperspecial maximal compact subgroup of $G(F^+_v)$, then $\pi_v$ is unramified.
    \end{enumerate}
\end{thm}

We now wish to define what it means for an irreducible representation
$\rbar:G_F\to\GL_2(\Flbar)$ to be modular of some weight. In order to
do so, we return to the spaces of algebraic modular forms considered
before. For each place $w|l$ of $F$, let $k_w$ denote the residue
field of $F_w$. If $w$ lies over a place $v$ of $F^+$, write
$v=ww^c$. Let $(\Z^2_+)_0^{\coprod_{w|l}\Hom(k_w,\Flbar)}$ denote the
subset of $(\Z^2_+)^{\coprod_{w|l}\Hom(k_w,\Flbar)}$ consisting of
elements $a$ such that for each $w|l$, if $\sigma\in\Hom(k_w,\Flbar)$
then \[a_{\sigma,1}+a_{\sigma c,2}=0.\] 
If $a \in (\Z^2_+)^{\coprod_{w|l}\Hom(k_w,\Flbar)}$ and $w|l$ is a
place of $F$, then we
denote by $a_w$ the element $(a_\sigma)_{\sigma \in
  \Hom(k_w,\Flbar)}\in (\Z^2_+)^{\Hom(k_w,\Flbar)}$.

If $\F$ is
a finite extension of $\Fl$, we say that an element $a\in(\Z^2_+)^{\Hom(\F,\Flbar)}$ is a \emph{Serre
  weight} if for each $\sigma\in\Hom(\F,\Flbar)$ we
have \[l-1\ge a_{\sigma,1}-a_{\sigma,2}.\]
If $a\in(\Z^2_+)^{\Hom(\F,\Flbar)}$ is a Serre weight then we define
an irreducible $\Flbar$-representation $F_a$ of $\GL_2(\F)$
by \[F_a:=\otimes_{\tau\in\Hom(\F,\Flbar)}\det{}^{a_{\tau,2}}\otimes\Sym^{a_{\tau,1}-a_{\tau,2}}\F^2\otimes_{\F,\tau}\Flbar.\]
We say that two Serre weights $a$ and $b$ are \emph{equivalent} if and only if
$F_a\cong F_b$ as representations of $\GL_2(\F)$. This is equivalent
to demanding that for each $\sigma\in\Hom(\F,\Flbar)$, we
have \[a_{\sigma,1}-a_{\sigma,2}=b_{\sigma,1}-b_{\sigma,2},\] and the
character $\F^\times\to\Flbar^\times$ given
by \[x\mapsto\prod_{\sigma\in\Hom(\F,\Flbar)}\sigma(x)^{a_{\sigma_2}-b_{\sigma_2}}\]is
trivial. If $L$ is
a finite extension of $\Ql$ with residue field $\F$, we say that an
element $\lambda\in(\Z^2_+)^{\Hom(L,\Qlbar)}$ is a \emph{lift} of an
element $a\in(\Z^2_+)^{\Hom(\F,\Flbar)}$ if for each $\sigma\in\Hom(\F,\Flbar)$ there is an element
$\tau\in\Hom(L,\Qlbar)$ lifting $\sigma$ such
that $\lambda_{\tau}=a_\sigma$, and for all other
$\tau'\in\Hom(L,\Qlbar)$ lifting $\sigma$ we have
$\lambda_{\tau'}=0$.

We say that an
element  $a\in(\Z^2_+)_0^{\coprod_{w|l}\Hom(k_w,\Flbar)}$ is a \emph{Serre
  weight} if $a_w$ is a Serre weight for each $w|l$.
 If $a\in(\Z^2_+)_0^{\coprod_{w|l}\Hom(k_w,\Flbar)}$ is a Serre weight,
we define an irreducible $\Flbar$-representation $F_a$ of
$G(\cO_{F^+,l})$ as follows: we define \[F_a=\otimes_{\tv\in\tS_l}
F_{a_\tv},\] an irreducible representation of
$\prod_{\tv\in\tilde{S}_l}\GL_2(k_\tv)$, and we let $G(\cO_{F^+,l})$ act on
$F_{a_\tv}$ by the composition of $\iota_\tv$ and the map $\GL_2(\cO_{F_\tv})\to\GL_2(k_\tv)$. Again, we say that two Serre weights $a$ and $b$ are equivalent
if and only if $F_a\cong F_b$ as representations of
$G(\cO_{F^+,l})$. This is equivalent to demanding that for each place
$w|l$ and each $\sigma\in\Hom(k_w,\Flbar)$ we have  \[a_{\sigma,1}-a_{\sigma,2}=b_{\sigma,1}-b_{\sigma,2},\] and the
character $k_w^\times\to\Flbar^\times$ given
by \[x\mapsto\prod_{\sigma\in\Hom(k_w,\Flbar)}\sigma(x)^{a_{\sigma_2}-b_{\sigma_2}}\]is
trivial. We say that a weight $\lambda\in (\Z^2_+)_0^{\Hom(F,\Qlbar)}$ is a
\emph{lift} of a Serre weight
$a\in(\Z^2_+)_0^{\coprod_{w|l}\Hom(k_w,\Flbar)}$ if for each $w|l$,
$\lambda_w$ is a lift of $a_w$.

For the rest of this section, fix $K=\Qlbar$.
\begin{defn}\label{defn: good subgroup} We say that a compact open subgroup
of $G(\A_{F^+}^\infty)$ is \emph{good} if $U=\prod_vU_v$ with $U_v$ a
compact open subgroup of $G(F^+_v)$ such that:
\begin{itemize}
\item $U_v\subset G(\bigO_{F^+_v})$ for all $v$ which split in $F$;
 \item $U_v=G(\bigO_{F^+_v})$ if $v|l$;
  \item $U_v$ is a hyperspecial maximal compact subgroup of $G(F_v^+)$
    if $v$ is inert in $F$.
\end{itemize}
\end{defn}

Let $U$ be a good compact open subgroup of $G(\A_{F^+}^\infty)$. Let $T$ be a finite set of finite places of $F^+$ which split in $F$,
containing $S_l$ and all the places $v$ which split in $F$ for which
$U_v\neq G(\bigO_{F^+_v})$. We let $\mathbb{T}^{T,univ}$ be the
commutative $\bigO$-polynomial algebra generated by formal variables
$T_w^{(j)}$ for all $1\le j\le 2$, $w$ a place of $F$ lying over a
place $v$ of $F^+$ which splits in $F$ and is not contained in $T$.
For any $\lambda\in (\Z^n_+)^{\tI_l}$ (resp. any Serre weight $a\in(\Z^2_+)_0^{\coprod_{v|l}\Hom(k_v,\Flbar)}$), the algebra
$\mathbb{T}^{T,univ}$ acts on $S_\lambda(U,\cO)$ (resp.\ $S(U,F_a)$) via the
 Hecke operators
  \[ T_{w}^{(j)}:=  \iota_{w}^{-1} \left[ GL_2(\mc{O}_{F_w}) \left( \begin{matrix}
      \varpi_{w}1_j & 0 \cr 0 & 1_{2-j} \end{matrix} \right)
GL_2(\mc{O}_{F_w}) \right] 
\] for $w\not \in T$ and $\varpi_w$ a uniformiser in $\mc{O}_{F_w}$.
Suppose that $\mf{m}$ is a maximal ideal of
$\mathbb{T}^{T,univ}$ with residue field $\Flbar$ such that
$S_\lambda(U,\Qlbar)_{\mf{m}}\neq 0$. Then (cf. Proposition 3.4.2 of \cite{cht}) by Lemma \ref{lem: relationship of algebraic automorphic forms to classical
    automorphic forms}, Theorem \ref{thm: base change from G to GL},
  and Theorem \ref{thm: existence of Galois reps attached to RACSDC}, 
there is a continuous semisimple
representation \[\rbar_{\mf{m}}:G_{F}\to\GL_2(\Flbar)\]associated
to $\mf{m}$, which is uniquely determined by the properties that:
\begin{itemize}
\item $\rbar_{\mf{m}}^c\cong\rbar_{\mf{m}}^\vee\epsilonbar_l^{-1}$,
\item for all finite
places $w$ of $F$ not lying over $T$, $\rbar_{\mf{\m}}|_{G_{F_w}}$ is
unramified, and\item if $w$ is a finite place of $F$ which doesn't lie over $T$ and which splits over $F^+$,
  then the characteristic polynomial of  $\rbar_{\mf{\m}}(\Frob_w)$
  is \[X^2-T_w^{(1)}X+(\mathbf{N}w)T_w^{(2)}. \]
\end{itemize}

\begin{lem}\label{Lemma: equivalence of modular in char 0 and l of
    some weight, unitary group version}
  Suppose that $U$ is sufficiently small, and let $\mf{m}$ be a maximal ideal of
$\mathbb{T}_\lambda^{T,univ}$ with residue field $\Flbar$. Suppose
that $a\in(\Z^2_+)_0^{\coprod_{v|l}\Hom(k_v,\Flbar)}$ is a Serre weight,
and that $\lambda\in(\Z^2_+)^{\tI_l}$ is a lift of $a$. Then
\[S_\lambda(U,\Qlbar)_{\mf{m}}\neq 0\] if and only if \[S(U,F_a)_{\mf{m}}\neq 0.\]
\end{lem}
\begin{proof}
Since $\Qlbar$ is $l$-torsion free, we have $S_\lambda(U,\Qlbar)_{\mf{m}}=S_\lambda(U,\cO_\Qlbar)_{\mf{m}}\otimes\Qlbar$, so $S_\lambda(U,\Qlbar)_{\mf{m}}\neq
  0$ if and only if  $S_\lambda(U,\cO_\Qlbar)_{\mf{m}}\neq
  0$. Since $U$ is sufficiently small, $S_\lambda(U,\Flbar)_{\mf{m}}\neq
  0$ if and only if  $S_\lambda(U,\cO_\Qlbar)_{\mf{m}}\neq
  0$, so that $S_\lambda(U,\Qlbar)_{\mf{m}}\neq
  0$ if and only if  $S_\lambda(U,\Flbar)_{\mf{m}}\neq
  0$. 

It then suffices to note that there is a natural isomorphism of
$G(\cO_{F^+,l})$-representations
$W_\lambda\otimes_{\cO_\Qlbar}\Flbar\isoto F_a$, so that we obtain a
$\mathbb{T}^{T,univ}$-equivariant isomorphism  $S_\lambda(U,\Flbar)\isoto S(U,F_a)$.
\end{proof}
We have the following definitions.

\begin{defn}\label{defn: galois split ramification}
  If $R$
  is a commutative ring and
  $r:G_F\to\GL_2(R)$ is a representation, we say that $r$ has
  \emph{split ramification} if $r|_{G_{F_w}}$ is unramified for any
  finite place $w\in F$ which does not split over $F^+$.
\end{defn}

\begin{defn}
  If $\pi$ is a RACSDC automorphic representation of $\GL_2(\A_F)$, we
  say that $\pi$ has \emph{split ramification} if $\pi_w$ is unramified for any
  finite place $w\in F$ which does not split over $F^+$.
\end{defn}

\begin{defn}\label{defn: modular of some Serre weight}
  Suppose that $\rbar:G_F\to\GL_2(\Flbar)$ is a continuous
  irreducible representation. Then we say that $\rbar$ \emph{is modular
  of weight} $a\in(\Z^2_+)_0^{\coprod_{w|l}\Hom(k_w,\Flbar)}$ if $a$
is a Serre weight and there is a sufficiently small, good level $U$, a set of places $T$ as
  above and a maximal ideal $\mathfrak{m}$ of
  $\mathbb{T}^{T,univ}$ with residue field $\Flbar$ such
  that
  \begin{itemize}
  \item $S(U,F_a)_{\mf{m}}\neq 0$, and
  \item $\rbar\cong \rbar_{\mathfrak{m}}$.
  \end{itemize} (Note that $\rbar_{\mf{m}}$ exists by Lemma \ref{Lemma: equivalence of modular in char 0 and l of
    some weight, unitary group version}.) We say that $\rbar$ is \emph{modular} if it is modular of
  some weight.
\end{defn}
\begin{remark}
  Note that if $\rbar:G_F\to\GL_2(\Flbar)$ is modular then $\rbar$
  must have split ramification and
  $\rbar^c\cong\rbar^\vee\epsilonbar_l^{-1}$. Note also that this
  definition is independent of the choice of $\tilde{S}_l$ (because  $F_{a_{\tv}}\circ
    \imath_{\tv}\cong F_{a_{c\tv}}\circ \imath_{c\tv}$, we see that $F_a$ itself
    is independent of the choice of $\tilde{S}_l$). 
\end{remark}

\begin{lem}\label{lem: equivalence of modular of Serre weight and
    RACSDC lift}Suppose that $\rbar:G_F\to\GL_2(\Flbar)$ is a continuous
  irreducible representation with split ramification. Let
  $a\in(\Z^2_+)_0^{\coprod_{w|l}\Hom(k_w,\Flbar)}$ be a Serre weight, and
  let $\lambda\in(\Z^2_+)_0^{\Hom(F,\Qlbar)}$ be a lift of $a$. Then $\rbar$ is modular
  of weight $a\in(\Z^2_+)_0^{\coprod_{w|l}\Hom(k_w,\Flbar)}$ if and only
  if there is a RACSDC automorphic representation $\pi$ of
  $\GL_2(\A_F)$ of weight $\imath\lambda$ and level prime to
  $l$ which has split ramification, and which satisfies $\rbar_{l,\imath}(\pi)\cong\rbar$.
\end{lem}
\begin{proof}Suppose firstly that $\rbar$ is modular of weight
  $a$. Then by definition there is a good $U$ and a $T$ as
  above with $U$ sufficiently small, and a maximal ideal $\mathfrak{m}$ of
  $\mathbb{T}^{T,univ}$ with residue field $\Flbar$ such
  that
  \begin{itemize}
  \item $S(U,F_a)_{\mf{m}}\neq 0$, and
  \item $\rbar\cong \rbar_{\mathfrak{m}}$.
  \end{itemize}By Lemma \ref{Lemma: equivalence of modular in char 0 and l of
    some weight, unitary group version}, the first condition is
  equivalent to $S_\lambda(U,\Qlbar)_{\mf{m}}\neq 0$. Define a compact
  open subgroup $U'=\prod_w U_w'$ of $\GL_2(\A_F^\infty)$ by
  \begin{itemize}
  \item $U_w'=\GL_2(\cO_{F_w})$ if $w$ is not split over $F^+$.
  \item $U_w'=\iota_w(U_{w|_{F^+}})$ if $w$ splits over $F^+$.
  \end{itemize}By Lemma \ref{lem: relationship of algebraic automorphic forms to classical
    automorphic forms}, Theorem \ref{thm: base change from G to GL},
  and Theorem \ref{thm: existence of Galois reps attached to RACSDC},
  there is a RACSDC automorphic representation $\pi$ of weight
  $\imath\lambda$  which satisfies $\rbar_{l,\iota}(\pi)\cong\rbar$, and
  $\pi_w^{U'_w}\ne 0$ for all finite places $w$ of $F$. Since $U$ is
  good, we see that $\pi$ has level prime to $l$, and it has split
  ramification, as required.

Conversely, suppose that there is a RACSDC automorphic representation $\pi$ of
  $\GL_2(\A_F)$ of weight $\imath\lambda$ which has split ramification and level prime to $l$ with
  $\rbar_{l,\iota}(\pi)\cong\rbar$.  Let $U=\prod_v U_v$ be a compact
  open subgroup of
  $G(\A_{F^+}^\infty)$ such that:
  \begin{itemize}
  \item $U_v=G(\cO_{F^+_v})$ if $v$ is inert in $F$;
  \item if $v$ splits as $v=ww^c$ in $F$, then $\pi_w^{\iota_w(U_v)}\neq (0)$;
  \item there is a finite place $v'$ of $F$ which splits as
    $w'w'^c$ in $F$ and is such
    that
    \begin{itemize}
    \item $v'$ lies above a rational prime $p$ with $[F(\zeta_p):F]>2$, and
    \item $\iota_{w'}(U_{v'})=\ker(\GL_2(\cO_{w'})\to\GL_2(k_{w'}))$.
    \end{itemize}
  \end{itemize}
The third bullet point implies that $U$ is
sufficiently small. Then by Lemma \ref{lem: relationship of algebraic automorphic forms to classical
    automorphic forms} and Theorem \ref{thm: base
    change from GL to G} we have $S_\lambda(U,\Qlbar)_{\mf{m}}\ne
  0$. The result follows from Lemma \ref{Lemma: equivalence of modular in char 0 and l of
    some weight, unitary group version}.
\end{proof}

\section{A lifting theorem}\label{sec:A lifting
  theorem}\subsection{}We recall some terminology from \cite{BLGGT},
specialized to the crystalline (as opposed to potentially crystalline)
case. Fix a prime
$l$. Let $K$ be a finite extension of $\Ql$, and let $\cO$ be the ring of
integers in a finite extension of $\Ql$ in $\Qlbar$, with residue field
$k$. Assume that for each continuous embedding $K \into \barQQ_l$, the
image is contained in the field of fractions of $\cO$.

Let $\rhobar:G_K\to\GL_n(k)$ be a continuous representation, and let
$R_{\cO,\rhobar}^\Box$ be the universal $\cO$-lifting ring. Let $\{
H_\tau \}$ be a collection of $n$-element multisets of integers
parametrized by $\tau \in \Hom(K,\barQQ_l)$. Then
$R_{\CO,\barrho}^\Box$ has a unique quotient $R_{\CO,\barrho,\{
  H_\tau\}, \cris}^\Box$ which is reduced and without $l$-torsion and
such that a $\barQQ_l$-point of $R_{\CO,\barrho}^\Box$ factors through
$R_{\CO,\barrho,\{ H_\tau\}, \cris}^\Box$ if and only if it
corresponds to a representation $\rho:G_K \ra \GL_n(\barQQ_l)$ which is
crystalline and has $\HT_\tau(\rho)=H_\tau$ for all $\tau:K \into
\barQQ_l$. We will write $R_{\barrho,\{ H_\tau\}, \cris}^\Box \otimes
\barQQ_l$
for  $R_{\CO,\barrho,\{ H_\tau\}, \cris}^\Box \otimes_{\CO} \barQQ_l$. 
This definition is independent of the choice of $\CO$. The scheme
$\Spec (R_{\barrho,\{ H_\tau\}, \cris}^\Box \otimes \barQQ_l)$ is
formally smooth over $\Qlbar$. (See \cite{kisindefrings}.)

If $\rho_i : G_K \to \GL_n(\CO_\Qlbar)$ are continuous representations
for $i=1,2$,
we say that $\rho_1$ {\em connects to} $\rho_2$, which we denote $\rho_1
\sim \rho_2$, if and only if 
\begin{itemize}
\item the reduction $\barrho_1:=\rho_1 \bmod \gm_{\barQQ_l}$ is equivalent to the reduction $\barrho_2:= \rho_2 \bmod \gm_{\barQQ_l}$;
\item $\rho_1$ and $\rho_2$ are both  crystalline;
\item for each $\tau:K \into \barQQ_l$ we have $\HT_\tau(\rho_1)=\HT_\tau(\rho_2)$;
\item and $\rho_1$ and $\rho_2$ define points on the same irreducible component of the scheme $\Spec (R_{\barrho_1,\{ \HT_\tau(\rho_1)\}, \cris}^\Box \otimes \barQQ_l)$.
\end{itemize}
(In this last bullet point, we mean that $\rho_1$ and $A\rho_2A^{-1}$
define points on the same irreducible component of $\Spec
(R_{\barrho_1,\{ \HT_\tau(\rho_1)\}, \cris}^\Box \otimes \barQQ_l)$
where $A\in \GL_n(\CO_{\Qlbar})$ is such that $A\barrho_2
A^{-1}=\barrho_1$. This condition is independent of the choice of $A$
by Lemma 1.2.1 of \cite{BLGGT}.) 
As in section 1.4 of \cite{BLGGT}, we have the following:
\begin{enumerate}
\item The relation $\rho_1 \sim \rho_2$ does not on the $\GL_n(\CO_{\barQQ_l})$-conjugacy class of $\rho_1$ or $\rho_2$.
\item $\sim$ is symmetric and transitive. 
\item If $K'/K$ is a finite extension and $\rho_1 \sim \rho_2$ then $\rho_1|_{G_{K'}} \sim \rho_2|_{G_{K'}}$. 
\item If $\rho_1 \sim \rho_2$ and $\rho_1' \sim \rho_2'$ then
  $\rho_1\oplus \rho_1'  \sim \rho_2 \oplus \rho_2'$ and
  $\rho_1\otimes \rho_1'  \sim \rho_2 \otimes \rho_2'$ and $\rho_1^\vee\sim\rho_2^\vee$.
\item If $\mu:G_K \ra \barQQ_l^\times$ is a continuous unramified
  character with $\mubar=1$ and $\rho_1$ is crystalline then $\rho_1 \sim \rho_1 \otimes \mu$.
\item \label{semisimp} Suppose that $\rho_1$ crystalline and that $\barrho_1$ is
  semisimple. Let $\Fil^i$ be an invariant filtration on $\rho_1$ by
  $\CO_{\barQQ_l}$ direct summands. Then $\rho_1 \sim \bigoplus_i
  \gr^i \rho_1$.
\end{enumerate}

We will call a crystalline representation $\rho:G_K \ra \GL_n(\CO_{\barQQ_l})$ {\em diagonal} if 
it is of the form $\chi_1 \oplus \dots \oplus \chi_n$ with $\chi_i:G_K \ra \CO_{\barQQ_l}^\times$.
We will call a crystalline representation $\rho:G_K \ra \GL_n(\CO_{\barQQ_l})$ {\em diagonalizable} if 
it connects to some diagonal representation. We will call a representation $\rho:G_K \ra \GL_n(\CO_{\barQQ_l})$ {\em potentially diagonalizable} if there is a finite extension $K'/K$ such that $\rho|_{G_{K'}}$ is diagonalizable.
Note that if $K''/K$ is a finite extension and $\rho$ is diagonalizable (resp. potentially diagonalizable)
then $\rho|_{G_{K''}}$ is diagonalizable (resp. potentially
diagonalizable).

As in \cite{BLGGT}, we make the following convention: Suppose that $F$
is a global field and that $r:G_F \ra GL_n(\barQQ_l)$ is a continuous
representation with irreducible reduction $\barr$. In this case there
is model $r^\circ:G_F \ra \GL_n(\CO_{\barQQ_l})$ of $r$, which is
unique up to $\GL_n(\CO_{\barQQ_l})$-conjugation. If $v|l$ is a place
of $F$ we write $r|_{G_{F_v}} \sim \rho_2$ to mean $r^\circ|_{G_{F_v}}
\sim \rho_2$. We will also say that $r|_{G_{F_v}}$ is (potentially)
diagonalizable to mean that $r^\circ|_{G_{F_v}}$ is.

Fix an isomorphism $\iota:\Qlbar\isoto\C$. Let $F$ be an imaginary CM field with maximal totally real subfield
$F^+$. We now demonstrate that any irreducible modular representation
$\rbar:G_F\to\GL_2(\Flbar)$ is, after a solvable base change,
congruent to an automorphic Galois representation which is diagonalizable at all places dividing $l$. The argument is similar to
that of the proof of Lemma 6.1.1 of \cite{blggord}, which proves an
analogous result over totally real fields.
\begin{lem}\label{lem: in 2 dimensions, exists a pot diag lift}
  Suppose that $\pi$ is a RACSDC automorphic representation of $\GL_2(\A_F)$ and
  that $\rbar_{l,\imath}(\pi)$ is irreducible. Let $\Favoid/F$ be a
  finite extension. Then there is a finite solvable extension $F'/F$
  and a RACSDC automorphic representation $\pi'$ of $\GL_2(\A_{F'})$
  such that
  \begin{itemize}
  \item $F'$ is linearly disjoint from $\Favoid$ over $F$.
\item $\pi'$ has weight 0.
  \item $\rbar_{l,\imath}(\pi')\cong \rbar_{l,\imath}(\pi)|_{G_{F'}}$.
  \item For each place $w|l$ of $F'$, $r_{l,\imath}(\pi')|_{G_{F'_w}}$ is diagonalizable.
  \end{itemize}

\end{lem}
\begin{proof}
  We first show that after a solvable base change,
  $\rbar_{l,\imath}(\pi)$ has a lift which is automorphic of weight 0. (This is presumably true
  without making a base change but the weaker statement will suffice
  for our purposes and allows us to transfer to a definite unitary group.) Choose a finite solvable extension of CM fields $F_1/F$ such that
\begin{itemize}

\item $F_1$ is linearly disjoint from $\Favoid\bar{F}^{\ker\rbar_{l,\imath}(\pi)}$ over $F$.
\item $F_1/F_1^+$ is unramified at all finite places.
\item $[F_1^+:\Q]$ is even.
\item Every place $v|l$ of $F_1^+$ splits in $F_1$.
\item If $\pi_1$ denotes the base change of $\pi$ to $F_1$, then
  $\pi_1$ is unramified at all finite places of $F_1$ lying over an
  inert place of $F^+_1$.
\item If $v|l$ is a place of $F_1$ such that $\pi_v$ is ramified, then
  $\pi_v$ is an unramified twist of the Steinberg representation.
\end{itemize}
As in section \ref{sec:Definitions}, we can choose a rank two unitary group
$G/F_1^+$ which is quasisplit at all finite places, compact at all
infinite places, and is split over $F_1$. Fix a model for $G$ over
$\cO_{F^+_1}$ as in section \ref{sec:Definitions}. We will freely use
the notation introduced in section \ref{sec:Definitions} to describe
spaces of algebraic modular forms on $G$. 

Suppose that $\pi_1$ has weight $a\in(\Z^2_+)^{\Hom(F_1,\C)}_0$. By Theorem \ref{thm: base change from GL to G} there is an automorphic
representation $\Pi$ of $G$ such that\begin{itemize}

\item If $v$ is a finite place of $F_1^+$ which is inert in $F_1$,
  then $\Pi_v$ has a fixed vector for some hyperspecial maximal
  compact subgroup of $G(F_{1,v}^+)$.
\item If $v$ is a finite place of $F_1^+$ which splits as $ww^c$ in $F_1$,
  then $\Pi_v\cong\pi_{1,w}\circ\iota_w$.
\item For each embedding $\tau:F_1^+\into\R$ and each $\tilde{\tau}$
  extending $\tau$, we have $\Pi_\tau\cong \Sigma^\vee_{a_\tau}\circ\iota_{\tilde{\tau}}$.
\end{itemize}Define a representation $W$ of $G(\cO_{F^+_1,l})$ on a
finite-dimensional $\Qlbar$-vector space as follows. Let $S_l$ denote
the set of places of $F^+_1$ lying over $l$, and let $\tS_l$ denote a
subset of the places of $F_1$ lying over $l$ consisting of exactly one
place $\tv$ lying over each place $v\in S_l$. Let $\tI_l$ denote the
set of embeddings $F_1\into\Qlbar$ giving rise to a place in $\tS_l$,
and for each $\tv\in\tS_l$ let $\tI_\tv$ denote the subset of $\tI_l$
of elements lying over $\tv$. Let $V_a$ be the  $\Qlbar$-vector space
with an action of $G(\cO_{F^+_{1,l}})$ given by
$W_a\otimes_{\cO_\Qlbar}\Qlbar$, where $W_a$ is defined as in section
\ref{sec:Definitions}. Let $V_l$ be the $\Qlbar$-vector space
with an action of $G(\cO_{F^+_{1,l}})$ given by \[V_l:=\otimes_{v\in
  S_l}V_v\]where $V_v$ is the (self-dual)
representation $\Ind_{I(\tv)}^{\GL_2(\cO_{F_{1,\tv}})}1_\Qlbar$, where $1_\Qlbar$ is the trivial
$\Qlbar$-representation of the standard Iwahori subgroup $I(\tv)$ of
$\GL_2(\cO_{F_{1,\tv}})$, and $G(\cO_{F^+_{1,l}})$ acts on $V_v$ via
$\iota_\tv$. Finally, let $W:=V_a\otimes V_l$, and let $W^\circ$ be a
$G(\cO_{F^+_{1,l}})$-stable $\cO_\Qlbar$-lattice in $W$.

Lemma \ref{lem: relationship of algebraic automorphic forms to
  classical automorphic forms} and the existence of $\Pi$ imply that
there is a compact open subgroup $U\in G(\A_{F_1^+}^\infty)$ which is
good in the sense of Definition \ref{defn: good subgroup} and is
sufficiently small, together with a finite set of places $T$ of
$F^+_1$ as in section \ref{sec:Definitions}, such that there is a
maximal ideal $\mf{m}$ of $\mathbb{T}^{T,univ}$ with:
\begin{itemize}
\item $S(U,W)_{\mf{m}}\ne 0$.
\item $\rbar_{\mf{m}}\cong \rbar_{l,\iota}(\pi_1)$.
\end{itemize}
Since $U$ is sufficiently small, we see (as in the proof of Lemma \ref{Lemma: equivalence of modular in char 0 and l of
    some weight, unitary group version}) that
  $S(U,W^\circ\otimes_{\cO_\Qlbar}\Flbar)_{\mf{m}}\ne 0$. Thus there is a Jordan-H\"older factor $F$ of the
$G(\cO_{F^+_{1,l}})$-representation
$W^\circ\otimes_{\cO_\Qlbar}\Flbar$ such that
$S(U,F)_{\mf{m}}\ne 0$. There is a smooth irreducible $\Qlbar$-representation $W_{sm}$ of
$G(\cO_{F^+_{1,l}})$ containing a stable $\cO_{\Qlbar}$-lattice $W_{sm}^\circ$
such that $F$ is a Jordan-H\"older factor of
$W_{sm}^\circ \otimes_{\cO_\Qlbar}\Flbar$ (this follows from the fact that
$F$ is a subquotient of $\otimes_{\tv \in
  \tS_l}\Ind_{1}^{\GL_2(k_\tv)}1_\Flbar$, so we may take $W_{sm}$ to
be a subquotient of the representation $\otimes_{\tv \in
  \tS_l}\Ind_{K_{\tv,1}}^{\GL_2(\mc{O}_{F_{1,\tv}})}1_{\Qlbar}$ where
$K_{\tv,1}=\ker(\GL_2(\mc{O}_{F_{1,\tv}})\to \GL_2(k_\tv))$). Since $U$ is
sufficiently small, we see that $S(U,W_{sm}^\circ)_\mf{m}\ne 0$, so
$S(U,W_{sm})_{\mf m}\ne 0$. Again, by Lemma \ref{lem: relationship of algebraic automorphic forms to classical
    automorphic forms} and Theorem \ref{thm: base change from
  G to GL} we see that there is a RACSDC automorphic representation
$\pi'_1$ of $\GL_2(\A_{F_1})$ of weight $0$ such that
$\rbar_{l,\imath}(\pi'_1)\cong \rbar_{\mf{m}}$ (the fact that
$\rbar_{\mf{m}}$ is irreducible allows us to deduce that $\pi_1'$ is cuspidal). After possibly making a
further solvable base change, we can assume that in addition to the
properties of $F_1$ listed above,
\begin{itemize}
  \item if $v|l$ is a place of $F_1$ such that $\pi'_{1,v}$ is ramified, then
  $\pi'_{1,v}$ is an unramified twist of the Steinberg representation.
\end{itemize}

We now repeat the argument above with $\pi_1$ replaced by $\pi'_1$
and hence with $a$ replaced by $0$.
By Lemma 3.1.5 of \cite{kis04}, we can choose $W_{sm}^\circ$ to be of the
form $W_{sm}=\otimes_{v\in S_l} W_{sm,v}^\circ \circ \imath_{\tv}$ where each $W_{sm,v}^\circ$ is a
cuspidal $F_{\tv}$-type (in the sense of \emph{loc.\ cit.}).
We see that there is a RACSDC automorphic representation
$\pi_1''$ of $\GL_2(\A_{F_1})$ of weight 0 such that
$\rbar_{l,\imath}(\pi_1'')\cong \rbar_{\mf{m}}$, and for each place
$v|l$, $\pi_{1,v}''$ is supercuspidal; so, after making another solvable
base change, we may assume that
\begin{itemize}
\item $\pi_{1,v}''$ is unramified for all
$v|l$.
\end{itemize} 
Summarising, we have obtained a solvable extension $F_1/F$ of CM
fields, and a RACSDC automorphic representation $\pi_1''$ of $\GL_2(\A_{F_1})$ such that
\begin{itemize}
\item $F_1$ is linearly disjoint from $\Favoid$ over $F$.
\item $\pi_1''$ has weight 0.
  \item $\rbar_{l,\imath}(\pi_1'')\cong \rbar_{l,\imath}(\pi)|_{G_{F_1}}$. 
\end{itemize}By Theorem \ref{thm: existence of Galois reps attached to
  RACSDC}, we see that for each place $v|l$ we have
\begin{itemize}
\item $r_{l,\imath}(\pi_1'')|_{G_{F_{1,v}}}$ is crystalline, and for
  each embedding $\tau:F_{1,v}\into\Qlbar$, we have
  $\HT_\tau(r_{l,\imath}(\pi_1'')|_{G_{F_{1,v}}})=\{0,1\}$.
\end{itemize}

Making a further base change, we may assume in addition that
\begin{itemize}

\item For each place $v|l$ of $F_1$,  $\rbar_{l,\imath}(\pi_1'')|_{G_{F_{1,v}}}$ is
  trivial, and there are crystalline representations $\rho_1$,
  $\rho_2:G_{F_{1,v}}\to\GL_2(\cO_\Qlbar)$ such that
  \begin{itemize}
  \item $\rhobar_1=\rhobar_2$ is the trivial representation.
  \item $\rho_1$ and $\rho_2$ are both diagonal.
  \item $\rho_1$ is ordinary, and $\rho_2$ is non-ordinary.
  \item For each $\tau: F_{1,v}\into\Qlbar$, $\HT_\tau(\rho_1)=\HT_\tau(\rho_2)=\{0,1\}$.
  \end{itemize}
\end{itemize}
From the existence of $\rho_1$ and $\rho_2$, Proposition 2.3 of
\cite{MR2280776}, and Corollary 2.5.16 of \cite{kis04}, it follows that
\begin{itemize}
\item For each place $v|l$,  $r_{l,\imath}(\pi_1'')|_{G_{F_{1,v}}}$ is diagonalizable.
\end{itemize}
The result follows, taking $F'=F_1$ and $\pi'=\pi_1''$.
\end{proof}
The following Theorem will allow us to ``change the weight'' of a
modular Galois representation. For the notion of an \emph{adequate}
subgroup of $\GL_2(\Flbar)$, which was originally defined in
\cite{jack}, we refer the reader to Appendix \ref{app:adequacy}, where a
detailed discussion of this condition is given. In particular, we
remind the reader that if $l\ge 7$, any irreducible subgroup of
$\GL_2(\Flbar)$ is adequate.

\begin{thm}\label{thm: existence of lifts, pot diag components}
  Let $l>2$ be prime and let $F$ be a CM field with maximal
  totally real subfield $F^+$. Assume that $\zeta_l \not \in F$ and
  that the extension $F/F^+$ is split at all places dividing $l$. Let
  $S$ be a finite set of finite places of $F^+$ which split in $F$ and
  assume that $S$ contains all the places dividing $l$. For each $v\in
  S$ choose a place $\tv$ of $F$ above $v$.

Suppose that \[\rbar:G_F\to\GL_2(\Flbar)\] is an irreducible
representation which is unramified at all places not lying over $S$ and which satisfies the following properties.
\begin{enumerate}
\item $\rbar$ is automorphic.\item $\rbar(G_{F(\zeta_l)})$ is adequate.
\end{enumerate}
For each $v\in S$, let $\rho_{\tv}:G_{F_\tv}\to\GL_2(\CO_\Qlbar)$ be a
lift of $\rbar|_{G_{F_\tv}}$. Assume that
\begin{itemize}
\item if $v|l$, then $\rho_\tv$ is crystalline and potentially diagonalizable, and if
  $\tau:F_\tv\into\Qlbar$ is any embedding, then $\HT_\tau(\rho_\tv)$
  consists of two distinct integers.\end{itemize}

Then there is a RACSDC automorphic representation $\pi$ of
$\GL_2(\A_F)$ of level prime to $l$ such that
\begin{itemize}
\item $\rbar\cong\rbar_{l,\iota}(\pi)$.
\item $\pi_v$ is unramified for all $v$ not lying over a place of $S$,
  so that $r_{l,\iota}(\pi_v)$ is unramified at all such $v$.
\item $r_{l,\iota}(\pi)|_{G_{F_\tv}}\sim\rho_\tv$ for all $v\in S$. In
  particular, for each place $v|l$,  $r_{l,\iota}(\pi)|_{G_{F_\tv}}$ is
  crystalline and for each embedding $\tau:F_\tv\into\Qlbar$, $\HT_\tau(r_{l,\iota}(\pi)|_{G_{F_\tv}})=\HT_\tau(\rho_\tv)$.
\end{itemize}

\end{thm}
\begin{proof}Let $\cG_2$ be the group scheme over $\Z$ defined in
  section 2.1 of \cite{cht}. Then by the main result of
  \cite{belchen}, $\rbar$ extends to a representation
  $\rhobar:G_{F^+}\to\cG_2(\Flbar)$ with multiplier
  $\epsilonbar_l^{-1}$.

By Lemma \ref{lem: in 2 dimensions, exists a pot diag lift}, we may
find a finite solvable extension $F'/F$ of CM fields and a RACSDC
automorphic representation $\pi'$ of $\GL_2(\A_{F'})$ such that
\begin{itemize}
\item $\rbar_{l,\iota}(\pi')\cong\rbar|_{G_{F'}}$.

\item $F'$ is linearly disjoint from $\bar{F}^{\ker \rbar}(\zeta_l)$
  over $F$.
\item $\pi'$ is unramified at all finite places.
\item For each place $w|l$ of $F'$, $r_{l,\imath}(\pi')|_{G_{F_w}}$ is crystalline and
  potentially diagonalizable.
\end{itemize}

We now apply Theorem \ref{diaglift} below, with
\begin{itemize}
\item $F$, $F'$, $S$ and $l$ as in the present setting,
\item $n=2$,
\item $\rbar$ our present $\rhobar$,
\item $\rho_v$ our $\rho_\tv$,
\item $\mu=\epsilon_l^{-1}$.
\end{itemize}

We conclude that there is a lift $r:G_F\to\GL_2(\CO_\Qlbar)$ (the
restriction to $G_F$ of the representation $r$ produced by Theorem
\ref{diaglift}) of $\rbar$ such that
\begin{itemize}
\item $r^c\cong r^\vee\epsilon_l^{-1}$,
\item if $v\in S$ then $r|_{G_{F_\tv}}\sim \rho_\tv$,
\item $r$ is unramified outside $S$.
\item $r|_{G_{F'}}$ is automorphic of level potentially prime to $l$.
\end{itemize}
Since the extension $F'/F$ is solvable,  we deduce that $r$ is
automorphic. Let $\pi$ be the RACSDC automorphic representation of
$\GL_2(\A_F)$ with $r_{l,\iota}(\pi)\cong r$.
By Theorem \ref{thm: existence of
  Galois reps attached to RACSDC}, we see that (since $r|_{G_{F_w}}$
is crystalline for all $w|l$, and unramified at all places $w$ not
lying over a place in $S$)
$\pi_w$ is unramified for all $w|l$ and all $w$ not lying over a place
in $S$, as required.\end{proof}

\section{Serre weight
  conjectures}\label{sec:Conjectures}\subsection{}We now recall
various formulations of Serre weight conjectures for $\GL_2$,
following \cite{bdj}, \cite{MR2430440}, \cite{gee061}, and
\cite{GHS}. These conjectures were formulated for various inner forms
of $\GL_2$ (indefinite and definite quaternion algebras), but it is
widely believed that they should also apply to outer forms of $\GL_2$,
such as the groups considered in the present paper. These conjectures
all consist of purely local descriptions of sets of weights, in a
sense which we will now explain (as in the rest of the paper, we work
with unitary groups, but the local formulations are the same as for inner
forms of $\GL_2$ which are split at all places lying over $l$).

Let $K$ be a finite extension of $\Ql$, with ring of integers $\cO_K$
and residue field $k$. Let $\rhobar:G_K\to\GL_2(\Flbar)$ be a
continuous representation. Then it is a folklore conjecture that there
is a set $W(\rhobar)$ of Serre weights of $\GL_2(k)$ with the property
that if $F$ is a CM field and $\rbar:G_F\to\GL_2(\Flbar)$ is an
irreducible modular representation (so in particular it is conjugate
self-dual), and $v|l$ is a place of $F$, then $\rbar$ is modular of
some Serre weight $\sigma_w\otimes_{\Flbar}\sigma^w$ (where $\sigma_w$
is a representation of $\GL_2(k_w)$) for some $\sigma^w$ if and only
if $\sigma_w\in W(\rbar|_{G_{F_w}})$.

It is natural to believe that there is a description of $W(\rhobar)$ in terms of the
existence of crystalline lifts with particular Hodge-Tate weights, as
we now explain. This is one of the motivations for the general Serre
weight conjectures explained in \cite{GHS}.

\begin{defn}
  \label{defn: Galois representation of Hodge type some weight}Let
  $K/\Ql$ be a finite extension, let
  $\lambda\in(\Z^2_+)^{\Hom(K,\Qlbar)}$, and let
  $\rho:G_K\to\GL_2(\Qlbar)$ be a de Rham representation. Then we say
  that $\rho$ has \emph{Hodge type} $\lambda$ if for each
  $\tau\in\Hom(K,\Qlbar)$, we have $\HT_\tau(\rho)=\{\lambda_{\tau,1}+1,\lambda_{\tau,2}\}$.
\end{defn}
\begin{rem}\label{rem: hodge type of auto Galois rep}
  As an immediate consequence of the definition and of Theorem
  \ref{thm: existence of Galois reps attached to RACSDC}, we see that
  if $\pi$ is a RACSDC automorphic representation of weight
  $\lambda\in(\Z^2_+)_0^{\Hom(F,\C)}$, then for each place $w|l$,
  $r_{l,\imath}(\pi)|_{G_{F_w}}$ has Hodge type $(\imath^{-1}\lambda)_w$.
\end{rem}

\begin{lemma}\label{lem:modularofaweightimpliescrystallinelifts -
    before the definition of the set of weights}
  Let $F$ be an imaginary CM field with maximal totally real subfield
  $F^+$, and suppose that $F/F^+$ is unramified at all finite places,
  that every place of $F^+$ dividing $l$ splits completely in $F$,
  and that $[F^+:\Q]$ is even. Suppose that $l>2$, and that
  $\rbar:G_F\to\GL_2(\Flbar)$ is an irreducible modular
  representation with split ramification. Let $a\in(\Z^2_+)_0^{\coprod_{w|l}\Hom(k_w,\Flbar)}$ be a
  Serre weight, and let $\lambda\in(\Z^2_+)_0^{\Hom(F,\Qlbar)}$ be a
  lift of $a$. If $\rbar$ is modular of weight $a$, then for each
  place $w|l$ there is a continuous lift $r_w:G_{F_w}\to\GL_2(\Qlbar)$
  of $\rbar|_{G_{F_w}}$
  such that
  \begin{itemize}

  \item $r_w$ is crystalline.
  \item $r_w$ has Hodge type $\lambda_w$.  

  \end{itemize}

\end{lemma}
\begin{proof}
  By Lemma \ref{lem: equivalence of modular of Serre weight and
    RACSDC lift} there is a RACSDC automorphic representation $\pi$ of
  $\GL_2(\A_F)$, which has level prime to $l$ and weight
  $\imath\lambda$, such that $\rbar_{l,\imath}(\pi)\cong\rbar$. Then
  we may take $r_w:=r_{l,\imath}(\pi)|_{G_{F_w}}$, which satisfies the
  above conditions by Remark \ref{rem: hodge type of auto Galois rep}.\end{proof}

This suggests the following definition, first made in \cite{gee061}.
\begin{defn}Let $K$ be a finite extension of $\Ql$, with ring of
  integers $\cO_K$ and residue field $k$. Let
  $\rhobar:G_K\to\GL_2(\Flbar)$ be a continuous representation. Then
  we let $\Wcris(\rhobar)$ be the set of Serre weights
  $a\in(\Z_+^2)^{\Hom(k,\Flbar)}$ with the property that there is a
  crystalline representation
  $\rho:G_K\to\GL_2(\Qlbar)$ lifting $\rhobar$, such that
  \begin{itemize}
  \item $\rho$ has Hodge type $\lambda$ for some lift
    $\lambda\in(\Z^2_+)^{\Hom(K,\Qlbar)}$ of $a$.

  \end{itemize}

\end{defn}

The results of section \ref{sec:A lifting theorem} inspire the
following definition.

\begin{defn}Let $K$ be a finite extension of $\Ql$, with ring of
  integers $\cO_K$ and residue field $k$. Let
  $\rhobar:G_K\to\GL_2(\Flbar)$ be a continuous representation. Then
  we let $\Wdiag(\rhobar)$ be the set of Serre weights
  $a\in(\Z_+^2)^{\Hom(k,\Flbar)}$ with the property that there is a
  continuous potentially diagonalizable crystalline representation
  $\rho:G_K\to\GL_2(\Qlbar)$ lifting $\rhobar$, such that
  \begin{itemize}
  \item  $\rho$ has Hodge type $\lambda$ for some lift
    $\lambda\in(\Z^2_+)^{\Hom(K,\Qlbar)}$ of $a$.

  \end{itemize}

\end{defn}
\begin{remark} 
Note that if a lift $\rho$ exists for one such $\lambda$, then
  composition of this lift with automorphisms of $\Qlbar$ provides a
  lift for any other choice of $\lambda$. If $a$ and $b$ are
  equivalent Serre weights, then $a\in\Wcris(\rhobar)$ (respectively
  $\Wdiag(\rhobar)$) if and only if $b\in\Wcris(\rhobar)$
  (respectively $\Wdiag(\rhobar)$). This is an easy consequence of
  Lemma \ref{lem:existence of crystalline chars} below, which provides
  a crystalline character with trivial reduction by which one can twist
  the crystalline Galois representations of Hodge type some lift of
  $a$ to obtain crystalline representations of Hodge type some lift of
  $b$. The same remarks apply to the set $\Wexplicit(\rhobar)$ defined
  below.
\end{remark}

Thus by definition we have $\Wdiag(\rhobar)\subset \Wcris(\rhobar)$. We
``globalise'' these definitions in the obvious way:
\begin{defn}\label{defn: serre weights global}Let $\rbar:G_F\to\GL_2(\Flbar)$ be a continuous
  representation with $\rbar^c\cong\rbar^\vee\epsilonbar_l^{-1}$. Then we let $\Wcris(\rbar)$ (respectively
  $\Wdiag(\rbar)$) be the set of Serre weights
  $a\in(\Z^2_+)_0^{\coprod_{w|l}\Hom(k_w,\Flbar)}$ such that for each
  place $w|l$, the corresponding Serre weight
  $a_w\in(\Z^2_+)^{\Hom(k_w,\Flbar)}$ is an element of
  $\Wcris(\rbar|_{G_{F_w}})$ (respectively $\Wdiag(\rbar|_{G_{F_w}})$).
  
\end{defn}
The
point of these definitions is the following corollary and theorem.

\begin{cor}
  \label{cor: modular of some weight implies crystalline lifts exist}Let $F$ be an imaginary CM field with maximal totally real subfield
  $F^+$, and suppose that $F/F^+$ is unramified at all finite places,
  that every place of $F^+$ dividing $l$ splits completely in $F$,
  and that $[F^+:\Q]$ is even. Suppose that $l>2$, and that
  $\rbar:G_F\to\GL_2(\Flbar)$ is an irreducible modular
  representation with split ramification. Let $a\in(\Z^2_+)_0^{\coprod_{w|l}\Hom(k_v,\Flbar)}$ be a
  Serre weight. If $\rbar$ is modular of weight $a$, then $a\in \Wcris(\rbar)$.
\end{cor}
\begin{proof}
  This is an immediate consequence of Lemma \ref{lem:modularofaweightimpliescrystallinelifts -
    before the definition of the set of weights} and Definition \ref{defn: serre weights global}.
\end{proof}
\begin{thm}
  \label{thm: potentially diagonalizable local lifts implies Serre
    weight}Let $F$ be an imaginary CM field with maximal totally real subfield
  $F^+$. Assume that $\zeta_l\notin F$, that $F/F^+$ is unramified at all finite places,
  that every place of $F^+$ dividing $l$ splits completely in $F$,
  and that $[F^+:\Q]$ is even. Suppose that $l>2$, and that
  $\rbar:G_F\to\GL_2(\Flbar)$ is an irreducible modular
  representation with split ramification. Assume that $\rbar(G_{F(\zeta_l)})$ is adequate.

 Let $a\in(\Z^2_+)_0^{\coprod_{w|l}\Hom(k_w,\Flbar)}$ be a
  Serre weight. Assume that $a\in \Wdiag(\rbar)$. Then $\rbar$ is
  modular of weight $a$. 
\end{thm}
\begin{proof}By the assumption that $a\in \Wdiag(\rbar)$, there is a
  lift $\lambda$ of $a$ such that for each
  $w|l$ there is a potentially
  diagonalizable crystalline lift $\rho_w:G_{F_w}\to\GL_2(\Qlbar)$ of
  $\rbar|_{G_{F_w}}$ of Hodge type $\lambda_w$.

By Theorem \ref{thm: existence of lifts, pot diag components} 
(applied with the set $S$ of that theorem being the set of places dividing $l$
together with the places at which $\rbar$ is ramified, and taking the
lifts $\rho_{\tv}$ to be those defined in the previous paragraph for $v|l$ and
arbitrary for $v$ not dividing $l$, noting that the fact that $\rbar$ is modular
guarantees the existence of lifts),
there is a RACSDC
automorphic representation $\pi$ of $\GL_2(\A_F)$ of weight
$\imath\lambda$, of level prime to $l$ and with split
ramification, such that $\rbar_{l,\iota}(\pi)\cong\rbar$. The result follows from Lemma \ref{lem: equivalence of modular of Serre weight and
    RACSDC lift}.\end{proof}
The majority of the rest of this paper will be devoted to making this
theorem more explicit. We believe that in fact
$\Wdiag(\rbar)=\Wcris(\rbar)$ in all cases, and we are able to show
strong results in this direction. In addition, we exhibit many
explicit weights in $\Wdiag(\rbar)$ (and again, conjecturally all
such weights). In view of Corollary \ref{cor: modular of some weight
  implies crystalline lifts exist} and Theorem \ref{thm: potentially diagonalizable local lifts implies Serre
    weight} (and the trivial inclusion $\Wdiag(\rbar)\subset
  \Wcris(\rbar)$), we are reduced to purely local questions, so we
  return to the setting of a finite extension $K/\Ql$ with residue
  field $k$ and absolute ramification index $e$, and we fix a continuous
  representation $\rhobar:G_K\to\GL_2(\Flbar)$. We then consider the
  following two questions:
  \begin{itemize}
    \item What is a good lower bound for the set $\Wdiag(\rhobar)$?
    \item What is a good upper bound for the set $\Wcris(\rhobar)$?
  \end{itemize}
If these two questions have the same answer, then the above work gives
a complete determination of the Serre weights of 2-dimensional mod $l$
Galois representations. In particular, we
conjecture (following \cite{GHS}) that the lower bound we provide for
$\Wdiag(\rhobar)$ is also an upper bound for $\Wcris(\rhobar)$.

The papers \cite{bdj}, \cite{MR2430440} and \cite{GHS} all give
explicit conjectural descriptions of $\Wcris(\rhobar)$ in increasing
orders of generality. [Strictly speaking, \cite{bdj} and
\cite{MR2430440} do not phrase their conjectures in the language of
crystalline lifts, but the results above make it reasonable to discuss
their descriptions in this optic; that is, we would like to see
whether their lists of weights can be proved to be lower bounds for
$\Wdiag(\rhobar)$ or upper bounds for $\Wcris(\rhobar)$. We will see
that the lower bound we provide for $\Wdiag(\rhobar)$ agrees with the
sets of weights predicted in \cite{bdj} and \cite{MR2430440} in most
cases, and conjecturally in all cases.] We now recall these
conjectures.

 We begin by defining the
fundamental characters of the inertia group of a finite extension $K$ of
$\Ql$. For each $\sigma\in \Hom(k,\Flbar)$ we
define the fundamental character $\omega_{\sigma}$ corresponding to
$\sigma$ to be the composite $$\xymatrix{I_{K^{\ab}/K}\ar[r]^{\Art_K^{-1}} &
  \bigO_{K}^{\times}\ar[r] & k^{\times}\ar[r]^{\sigma^{-1}} &
  \Flbar^{\times}.}$$ Let $K'$ denote the quadratic unramified
extension of $K$ inside $\Qlbar$, with residue field
$k'\subset\Flbar$.

We now recall a slight variant of the conjectures of
\cite{bdj}, who associate a set of weights to any continuous
representation $\rhobar:G_K\to\GL_2(\Flbar)$ in the case that $K/\Ql$
is unramified. We define a set of weights $\WBDJ(\rhobar)$ as follows:

\begin{defn}
  \label{defn: BDJ niveau 2} Assume that $K/\Ql$ is unramified. If $\rhobar$ is irreducible, then a Serre
  weight $a\in(\Z^2_+)^{\Hom(k,\Flbar)}$ is in $\WBDJ(\rhobar)$ if
  and only if there is a subset $J\subset\Hom(k',\Flbar)$ consisting of
  exactly one embedding extending each element of $\Hom(k,\Flbar)$, such that if we write
  $\Hom(k',\Flbar)=J\coprod J^c$, then (where here and below, if
  $\sigma\in\Hom(k',\Flbar)$ we write $a_{\sigma,i}$ for $a_{\sigma|_k,i}$)  \[\rhobar|_{I_K}\cong
  \begin{pmatrix}\prod_{\sigma\in
      J}\omega_{\sigma}^{a_{\sigma,1}+1}\prod_{\sigma\in
      J^c}\omega_\sigma^{a_{\sigma,2}}&0\\ 0& \prod_{\sigma\in
      J^c}\omega_\sigma^{a_{\sigma,1}+1}\prod_{\sigma\in
      J}\omega_\sigma^{a_{\sigma,2}}   
  \end{pmatrix}.\]
\end{defn}If $\tau\in \Hom(K,\Qlbar)$, we let
$\overline{\tau}$ be the induced element of $\Hom(k,\Flbar)$.

\begin{defn}
  \label{defn: BDJ niveau 1}Assume that $K/\Ql$ is unramified. If $\rhobar$ is reducible, then a Serre
  weight $a\in(\Z^2_+)^{\Hom(k,\Flbar)}$ is in $\WBDJ(\rhobar)$ if
  and only if there is a subset $J\subset\Hom(k,\Flbar)$ such that
  $\rhobar$ has a crystalline lift of the form \[  \begin{pmatrix}\psi_1&*\\ 0& \psi_2
  \end{pmatrix}\]where $\HT_\tau(\psi_1)=a_{\taubar,1}+1$ if $\tau\in
  J$ and $a_{\taubar,2}$ otherwise, and
  $\HT_\tau(\psi_2)=a_{\taubar,2}$ if $\tau\in J$, and
  $a_{\taubar,1}+1$ otherwise. In particular, if we write
  $\Hom(k,\Flbar)=J\coprod J^c$ and $a\in \WBDJ(\rhobar)$ then we necessarily have  \[\rhobar|_{I_K}\cong
  \begin{pmatrix}\prod_{\sigma\in
      J}\omega_\sigma^{a_{\sigma,1}+1}\prod_{\sigma\in
      J^c}\omega_\sigma^{a_{\sigma,2}}&*\\ 0& \prod_{\sigma\in
      J^c}\omega_\sigma^{a_{\sigma,1}+1}\prod_{\sigma\in
      J}\omega_\sigma^{a_{\sigma,2}}

  \end{pmatrix}.\]
\end{defn}

[The description of $\rhobar|_{I_K}$ in the reducible case is
immediate from Lemma \ref{lem:existence of crystalline chars} below
(see also Lemma \ref{lem:explicit semisimple crystalline lifts in niveau 1}). To see
the relationship of these definitions to those of \cite{bdj} is
straightforward. In the irreducible case, it follows at once from
equation 3.1(1) of \cite{bdj} that our description agrees with that of
\cite{bdj} (where the set that we denote $\WBDJ(\rhobar)$ is called $W_{\mf{p}}(\rho)$). 

In the reducible case, it is possible that our set
$\WBDJ(\rhobar)$ differs from the set proposed in \cite{bdj},
although it is conjectured in \cite{bdj} that this is not the case,
and in any case we shall see below that $\WBDJ(\rhobar)\subset
\Wdiag(\rhobar)$. Suppose firstly that $\rhobar$ is not a twist of
an extension of the trivial character by either the trivial character
or the cyclotomic character. Then the definition of $W_{\mf{p}}(\rho)$
in \cite{bdj} agrees with our $\WBDJ(\rhobar)$, except that
\cite{bdj} make an additional prescription on the character
$\psi_1\psi_2^{-1}$ (they demand that it takes a certain value on a fixed
Frobenius element). However, Remark 3.10 of \cite{bdj} explains that
in most cases these two formulations are equivalent, and conjectures
that they are always equivalent.

In the remaining cases, it is not immediately clear that our
definitions agree, although the authors of \cite{bdj} have indicated
to us that they conjecture that they agree, and that their definition
is intended as a refinement of the definition given here. The
definition given in \cite{bdj} is better suited to comparisons of the
sets $\WBDJ(\rhobar)$ as $\rhobar$ varies over representations with
the same semisimplification.]

We now turn to the formulation given in \cite{MR2430440}. We drop the
assumption that $K/\Ql$ is unramified, but assume instead that
$\rhobar|_{I_K}$ is semisimple. In this case a set $\WSch(\rhobar)$
of Serre weights is proposed in \cite{MR2430440} as follows.

\begin{defn}
  \label{defn: Schein  niveau 2} If $\rhobar$ is irreducible, then a Serre
  weight $a\in(\Z^2_+)^{\Hom(k,\Flbar)}$ is in $\WSch(\rhobar)$ if
  and only if there is a subset $J\in\Hom(k',\Flbar)$ consisting of
  exactly one embedding extending each element of $\Hom(k,\Flbar)$,
  and for each $\sigma\in\Hom(k,\Flbar)$ an integer $0\le
  \delta_\sigma\le e-1$ such that if we write
  $\Hom(k',\Flbar)=J\coprod J^c$, then
(where here and below we write
  $\delta_\sigma$ for $\delta_{\sigma|k}$)  \[\rhobar|_{I_K}\cong
  \begin{pmatrix}\prod_{\sigma\in
      J}\omega_{\sigma}^{a_{\sigma,1}+1+\delta_{\sigma}}\prod_{\sigma\in
      J^c}\omega_\sigma^{a_{\sigma,2}+e-1-\delta_{\sigma}}&0\\ 0& \prod_{\sigma\in
      J^c}\omega_\sigma^{a_{\sigma,1}+1+\delta_{\sigma}}\prod_{\sigma\in
      J}\omega_\sigma^{a_{\sigma,2}+e-1-\delta_{\sigma}}

  \end{pmatrix}.\]
\end{defn}

\begin{defn}
  \label{defn: Schein  niveau 1} If $\rhobar$ is reducible and
  $\rhobar|_{I_K}$ is semisimple, then a Serre
  weight $a\in(\Z^2_+)^{\Hom(k,\Flbar)}$ is in $\WSch(\rhobar)$ if
  and only if there is a subset $J\in\Hom(k,\Flbar)$,
  and for each $\sigma\in\Hom(k,\Flbar)$ an integer $0\le
  \delta_\sigma\le e-1$ such that if we write
  $\Hom(k,\Flbar)=J\coprod J^c$, then  \[\rhobar|_{I_K}\cong
  \begin{pmatrix}\prod_{\sigma\in
      J}\omega_\sigma^{a_{\sigma,1}+1+\delta_\sigma}\prod_{\sigma\in
      J^c}\omega_\sigma^{a_{\sigma,2}+\delta_\sigma}&0\\ 0& \prod_{\sigma\in
      J^c}\omega_\sigma^{a_{\sigma,1}+e-\delta_\sigma}\prod_{\sigma\in
      J}\omega_\sigma^{a_{\sigma,2}+e-1-\delta_\sigma}

  \end{pmatrix}.\]
\end{defn}
[That these agree with the definitions of \cite{MR2430440} is
immediate from the statements of Theorems 2.4 and 2.5 of
\cite{MR2430440} (after replacing $\delta_\sigma$ by
$e-1-\delta_\sigma$ in the case $\sigma\in J^c$).] Finally, following \cite{GHS}, we define an explicit
set of weights $\WGHS(\rhobar)$ in the case that $\rhobar$ is
reducible but not necessarily decomposable when restricted to $I_K$ (without assuming that
$K/\Ql$ is unramified). \begin{defn}
  \label{defn: GHS niveau 1}If $\rhobar$ is reducible, then a Serre
  weight $a\in(\Z^2_+)^{\Hom(k,\Flbar)}$ is in $\WGHS(\rhobar)$ if
  and only if $\rhobar$ has a crystalline lift of the
  form \[ \begin{pmatrix}\psi_1&*\\ 0& \psi_2
  \end{pmatrix}\] which has Hodge type $\lambda$ for some lift
  $\lambda\in(\Z^2_+)^{\Hom(K,\Qlbar)}$ of $a$. In particular, if $a\in \WGHS(\rhobar)$ then it is necessarily the case that there is a decomposition
  $\Hom(k,\Flbar)=J\coprod J^c$ and for each $\sigma\in \Hom(k,\Flbar)$ there is an integer
  $0\le \delta_\sigma\le e-1$ such that \[\rhobar|_{I_K}\cong
  \begin{pmatrix}\prod_{\sigma\in
      J}\omega_\sigma^{a_{\sigma,1}+1+\delta_\sigma}\prod_{\sigma\in
      J^c}\omega_\sigma^{a_{\sigma,2}+\delta_\sigma}&*\\ 0& \prod_{\sigma\in
      J^c}\omega_\sigma^{a_{\sigma,1}+e-\delta_\sigma}\prod_{\sigma\in
      J}\omega_\sigma^{a_{\sigma,2}+e-1-\delta_\sigma}

  \end{pmatrix}.\]
\end{defn}
[Again, the form of $\rhobar|_{I_K}$ is immediate from Lemma
\ref{lem:existence of crystalline chars} below.]
In order to see the relationship between these definitions, we now
study the question of when it is ``obvious'' that one can write down a
crystalline lift with specified Hodge-Tate weights of a given
$\rhobar$. If $\chi$ is a character of $G_{K}$ or $I_{K}$ valued in $\CO_\Qlbar^\times$, we denote
its reduction mod $l$ by $\overline{\chi}$. 

\begin{lem}\label{lem:existence of crystalline chars}Let
  $A=\{a_{\tau}\}_{\tau\in \Hom(K,\Qlbar)}$ be a set of integers. Then
  there is a crystalline character $\epsilon_{A}$ of $G_{K}$
  such that $\HT_{\tau}(\epsilon_{A})=a_{\tau}$ for all $\tau\in
  \Hom(K,\Qlbar)$, and $\epsilon_{A}$ is unique up to unramified
  twist. Furthermore,
  $\epsilonbar_{A}|_{I_{K}}=\prod_{\sigma\in
    \Hom(k,\Flbar)}\omega_{\sigma}^{b_{\sigma}},$ where $$b_{\sigma}=\sum_{\tau\in
    \Hom(K,\Qlbar):\overline{\tau}=\sigma}a_{\tau}.$$
\end{lem}
\begin{proof}
  This is Lemma 6.2 of \cite{geesavitttotallyramified}. [Note that
  the definitions of fundamental characters in this paper are the
  inverse of those defined in section 5
of \cite{geesavitttotallyramified}; this is because our conventions
for Hodge-Tate weights are the opposite of those of
\cite{geesavitttotallyramified}.]
\end{proof}

\begin{lem}\label{lem:explicit semisimple crystalline lifts in niveau 1}
  Suppose that $a\in(\Z^2_+)^{\Hom(k,\Flbar)}$ is a Serre weight, and that $\rhobar:G_K\to\GL_2(\Flbar)$ is
  a continuous 
representation which is a direct sum of two characters. Suppose that there is a
  decomposition $\Hom(k,\Flbar)=J\coprod J^c$ and for each $\sigma\in
  \Hom(k,\Flbar)$ there is an integer $0\le \delta_\sigma\le e-1$ with \[\rhobar|_{I_K}\cong
  \begin{pmatrix}\prod_{\sigma\in
      J}\omega_\sigma^{a_{\sigma,1}+1+\delta_\sigma}\prod_{\sigma\in
      J^c}\omega_\sigma^{a_{\sigma,2}+\delta_\sigma}&0\\ 0& \prod_{\sigma\in
      J^c}\omega_\sigma^{a_{\sigma,1}+e-\delta_\sigma}\prod_{\sigma\in
      J}\omega_\sigma^{a_{\sigma,2}+e-1-\delta_\sigma}

  \end{pmatrix}.\]

Then for any $\lambda\in(\Z^2_+)^{\Hom(K,\Qlbar)}$ lifting $a$,
$\rhobar$ has a diagonal crystalline lift of Hodge type
$\lambda$.
\end{lem}
\begin{proof}We define sets $B=\{b_{\tau}\}_{\tau\in \Hom(K,\Qlbar)}$ and
  $C=\{c_{\tau}\}_{\tau\in \Hom(K,\Qlbar)}$ of integers as follows. For each
  $\sigma\in \Hom(k,\Flbar)$, let $S_\sigma$ be the subset of $\Hom(K,\Qlbar)$ consisting of
  those $\tau$ with $\taubar=\sigma$. By definition, for each $\sigma$
  there is a distinguished element $\tilde{\sigma}$ of $S_\sigma$ with
  $\lambda_{\tilde{\sigma},i}=a_{\sigma,i}$, and for each element
  $\tau\ne\tilde{\sigma}$ of $S_\sigma$ we have
  $\lambda_{\tau,i}=0$. Choose a subset $K_\sigma$ of
  $S_\sigma\backslash\{\tilde{\sigma}\}$ of size $\delta_\sigma$. 

  Suppose $\sigma \in J$. We let $b_{\tilde{\sigma}}=a_{\sigma,1}+1$, we let
  $b_{\tau}=1$ if $\tau\in K_\sigma$, and $b_\tau=0$
  for all other $\tau \in S_\sigma$. Similarly, we let $c_{\tilde{\sigma}}=a_{\sigma,2}$, we let
  $c_{\tau}=1$ if $\tau \in S_\sigma\backslash K_\sigma\cup
  \{\tilde{\sigma}\}$ and $c_\tau=0$ for $\tau \in K_\sigma$.

  Suppose $\sigma \notin J$. We let $c_{\tilde{\sigma}}=a_{\sigma,1}+1$, we let
  $c_{\tau}=1$ if $\tau\in K_\sigma$, and $c_\tau=0$
  for all other elements of $S_\sigma$. Similarly, we let $b_{\tilde{\sigma}}=a_{\sigma,2}$, we let
  $b_{\tau}=1$ if $\tau \in S_\sigma\backslash K_\sigma\cup
  \{\tilde{\sigma}\}$ and $b_\tau=0$ for $\tau \in K_\sigma$.
  
Then by Lemma \ref{lem:existence of crystalline chars}, $\rhobar$ has
a lift given by the direct sum of unramified twist of $\epsilon_B$ and
an unramified twist of $\epsilon_C$. By definition, this is a
diagonal crystalline lift of Hodge type $\lambda$.
\end{proof}
\begin{cor}\label{cor: if e is big then all weights are explicit,
    niveau 1}
  Suppose that $e\ge l$, and $\rhobar:G_K\to\GL_2(\Flbar)$ is a
  continuous representation which is a direct sum of two characters. Suppose that
  $a\in(\Z^2_+)^{\Hom(k,\Flbar)}$ is a Serre weight such that
  \[\det\rhobar|_{I_K}=\prod_{\sigma\in
    \Hom(k,\Flbar)}\omega_\sigma^{a_{\sigma,1}+a_{\sigma,2}+e}.\]Then for any $\lambda\in(\Z^2_+)^{\Hom(K,\Qlbar)}$ lifting $a$,
$\rhobar$ has a diagonal crystalline lift of Hodge type
$\lambda$. 
\end{cor}
\begin{proof}Suppose that $\rhobar\cong\psi_1\oplus\psi_2$. Since any
  representation as in the statement of Lemma \ref{lem:explicit
    semisimple crystalline lifts in niveau 1} has  $\det\rhobar|_{I_K}=\prod_{\sigma\in
    \Hom(k,\Flbar)}\omega_\sigma^{a_{\sigma,1}+a_{\sigma,2}+e}$, it suffices to
  show that we can choose $J$ and $\delta_\sigma$ as in the statement
  of Lemma \ref{lem:explicit semisimple crystalline lifts in niveau 1}
  such that \[\psi_1|_{I_K}=\prod_{\sigma\in
      J}\omega_\sigma^{a_{\sigma,1}+1+\delta_\sigma}\prod_{\sigma\in
      J^c}\omega_\sigma^{a_{\sigma,2}+\delta_\sigma}.\] Take $J=\Hom(k,\Flbar)$,
    and
    write $\psi_1|_{I_K}\prod_{\sigma\in
      J}\omega_\sigma^{-(a_{\sigma,1}+1)}$ in the form
    $\prod_{\sigma\in J}\omega_\sigma^{c_\sigma}$ with $0\le
    c_\sigma\le l-1$. Then we may take $\delta_\sigma=c_\sigma$.
\end{proof}
\begin{remark} Contrary to the claim made in the introduction to
  \cite{MR2430440}, once $e=l-1$, it is no longer the case that for every $\rhobar$
with determinant $\prod_{\sigma\in
    \Hom(k,\Flbar)}\omega_\sigma^{a_{\sigma,1}+a_{\sigma,2}+e}$ can 
we can apply Lemma \ref{lem:explicit semisimple crystalline lifts in niveau 1}
to find a crystalline diagonal lift.
For a counterexample,
take $l=7$, $[k:\F_l]=2$, and label the two embeddings $k\into \Fbar_l$ as $\sigma_1$
and $\sigma_2$. Then take 
\begin{align*}
a_{\sigma_1,1} &= l-1,  & a_{\sigma_2,1} &= 1, \\
a_{\sigma_1,2}&=a_{\sigma_2,2} = 0, & \rhobar &= \psi_1 \oplus \psi_2,\\
\text{where }\hspace{1cm} \psi_1 &= \omega_{\sigma_1}^{l-1} \omega_{\sigma_2}^{4}
& \psi_2 &= \omega_{\sigma_1}^{l-1} \omega_{\sigma_2}^{l-4}.
\end{align*}
Then it is easy to see (by considering all
4 possible sets $J$) that we can never choose $\delta_\sigma$ to make $\rhobar$
equivalent to the representation in Lemma \ref{lem:explicit semisimple crystalline lifts in niveau 1}.
\end{remark}

We now consider the case of irreducible representations
$\rhobar:G_K\to\GL_2(\Flbar)$. Recall that $K'$ denotes the unique unramified
quadratic extension of $K$ and $k'$ denotes its residue field. Then $\rhobar$ is induced from a character
of $G_{K'}$, and $\rhobar|_{I_K}$ decomposes as a sum of characters.
\begin{lem}\label{lem:explicit semisimple crystalline lifts in niveau 2}
  Suppose that $a\in(\Z^2_+)_0^{\Hom(k,\Flbar)}$ is a Serre weight, and that $\rhobar:G_K\to\GL_2(\Flbar)$ is
  a continuous irreducible representation. Suppose that there is a
  decomposition  $\Hom(k',\Flbar)=J\coprod J^c$ such that $J$ contains
  exactly one embedding extending each element of $\Hom(k,\Flbar)$,
  and for each $\sigma\in \Hom(k,\Flbar)$ there is an integer $0\le \delta_\sigma\le e-1$ with \[\rhobar|_{I_K}\cong
  \begin{pmatrix}\prod_{\sigma\in
      J}\omega_{\sigma}^{a_{\sigma,1}+1+\delta_{\sigma}}\prod_{\sigma\in
      J^c}\omega_\sigma^{a_{\sigma,2}+e-1-\delta_{\sigma}}&0\\ 0& \prod_{\sigma\in
      J^c}\omega_\sigma^{a_{\sigma,1}+1+\delta_{\sigma}}\prod_{\sigma\in
      J}\omega_\sigma^{a_{\sigma,2}+e-1-\delta_{\sigma}}
  \end{pmatrix}.\]

Then for any  $\lambda\in(\Z^2_+)_0^{\Hom(K,\Qlbar)}$ lifting $a$,
$\rhobar$ has a potentially diagonalizable crystalline lift of Hodge type
$\lambda$ which becomes diagonal when restricted to $G_{K'}$.
\end{lem}
\begin{proof}We may write \[\rhobar\cong\Ind_{G_{K'}}^{G_K}\psi\]for
  some character $\psi:G_{K'}\to\Flbar^\times$ which satisfies \[\psi|_{I_{K'}}=\prod_{\sigma\in
      J}\omega_{\sigma}^{a_{\sigma,1}+1+\delta_{\sigma}}\prod_{\sigma\in
      J^c}\omega_\sigma^{a_{\sigma,2}+e-1-\delta_{\sigma}}.\] We
    define a set $B=\{b_\tau\}_{\tau\in \Hom(K',\Qlbar)}$ as follows. 
For each $\tau\in \Hom(K,\Qlbar)$, we denote the two extensions of $\tau$ to
elements of $\Hom(K',\Qlbar)$ by $\tau_1$ and $\tau_2$, where
$\taubar_1\in J$ and $\taubar_2\in J^c$. 

 For each
  $\sigma\in \Hom(k,\Flbar)$, let $S_\sigma$ be the
  subset of $\Hom(K,\Qlbar)$ consisting of
  those $\tau$ with $\taubar=\sigma$. By definition, for each
  $\sigma\in \Hom(k,\Flbar)$
  there is a distinguished element $\tilde{\sigma}$ of $S_\sigma$ with
  $\lambda_{\tilde{\sigma},i}=a_{\sigma,i}$, and for each element
  $\tau\ne\tilde{\sigma}$ of $S_\sigma$ we have
  $\lambda_{\tau,i}=0$. Choose a subset $K_\sigma$ of
  $S_\sigma\backslash\{\tilde{\sigma}\}$ of size $\delta_\sigma$. 

Then we let $b_{\tilde{\sigma}_1}=a_{\sigma,1}+1$, and
$b_{\tilde{\sigma}_2}=a_{\sigma,2}$. If $\tau\in K_\sigma$, we let
  $b_{\tau_1}=1$ and $b_{\tau_2}=0$. If $\tau\in
  S_\sigma\backslash\{\tilde\sigma\}\cup K_\sigma$, we let
  $b_{\tau_1}=0$ and $b_{\tau_2}=1$.

Then by Lemma \ref{lem:existence of crystalline chars} there is a
crystalline character $\tilde{\psi}$ of $G_{K'}$ lifting $\psi$, which
is an unramified twist of the character $\epsilon_B$. The
representation $\Ind_{G_{K'}}^{G_K}\tilde{\psi}$ gives the required lift.   
\end{proof}

\begin{cor}\label{cor: if e is big then all weights are explicit,
    niveau 2}
  Suppose that $e\ge l$, and $\rhobar:G_K\to\GL_2(\Flbar)$ is a
  continuous irreducible representation. Suppose that
  $a\in(\Z^2)_+^{\Hom(k,\Flbar)}$ is a Serre weight such that 
\[\det\rhobar|_{I_K}=\prod_{\sigma\in
    \Hom(k,\Flbar)}\omega_\sigma^{a_{\sigma,1}+a_{\sigma,2}+e}.\]Then for any weight $\lambda\in(\Z^2)_+^{\Hom(K,\Qlbar)}$ lifting $a$,
$\rhobar$ has a potentially diagonalizable crystalline lift of Hodge type
$\lambda$, which becomes diagonal upon restriction to
$G_{K'}$.
\end{cor}
\begin{proof}
We can write $\rhobar\cong\Ind_{G_{K'}}^{G_{K}}\phi$ for some character
$\phi:G_{K'}\to\Fbar_l^\times$. The condition on the determinant of 
$\rhobar$ tells us that 
\[(\phi\phi^c)|_{I_{K'}}=\prod_{\sigma\in
    \Hom(k',\Flbar)}\omega_\sigma^{a_{\sigma,1}+a_{\sigma,2}+e},\]
where $c$ denotes the nontrivial element of $\Gal(K'/K)$ and $\phi^c$
denotes $\phi$ conjugated by $c$.

On the other hand, by Lemma \ref{lem:explicit semisimple crystalline lifts in niveau 2}
and the first line of its proof, we know that if we choose $J$ and 
$(\delta_\sigma)_{\sigma\in\Hom(k,\Fbar_l)}$ as in the statement of that lemma 
and write
\[\psi=\prod_{\sigma\in
      J}\omega_{\sigma}^{a_{\sigma,1}+1+\delta_{\sigma}}\prod_{\sigma\in
      J^c}\omega_\sigma^{a_{\sigma,2}+e-1-\delta_{\sigma}},\]
then for any 
$\lambda\in(\Z^2_+)_0^{\Hom(K,\Qlbar)}$ lifting $a$ and
any representation $G_K\to\GL_2(\Flbar)$ agreeing with
$\Ind_{I_{K'}}^{I_{K}}\psi$ on $I_K$, that representation
has a potentially diagonalizable crystalline lift of Hodge type
$\lambda$ which becomes diagonal when restricted to $G_{K'}$.
Thus to prove the present corollary it suffices to show that, for an 
appropriate choice of $J$ and $(\delta_\sigma)_{\sigma\in\Hom(k,\Fbar_l)}$,
we can arrange for $\psi$ to equal $\phi|_{I_{K'}}$.

Let $f=[k:\F_l]$, and let $\{\sigma_1,\dots,\sigma_{2f}\}$ denote the embeddings 
$k'\into\Flbar$.
We will take the labels mod $2f$, and we can and do choose 
the labelling such that
\begin{itemize}
\item $\omega_{\sigma_i}=\omega_{\sigma_{i+1}}^l$, and\item $\omega_{\sigma_{i+f}|_k}=\omega_{\sigma_{i}|_k}$. (In fact, this second point will follow
from the first.)
\end{itemize}
We will write $\omega_i$ for $\omega_{\sigma_i}$ 
(thus the $i$ here is taken mod $2f$); and we
will write $\delta_i$ for $\delta_{\sigma_i|k}$ and $a_{i,j}$ for $a_{\sigma_i|k,j}$ 
(thus the $i$s here are taken mod $f$).
We will choose $J=\{\sigma_1,\dots,\sigma_f\}$ , and we see that this contains, as is required,
exactly one embedding extending each element of $\Hom(k,\Fbar_l)$.

We let
\[\phi' := \phi|_{I_{K'}} \prod_{i=1}^f \omega_i^{-a_{i,1}-1}
           \prod_{i=f+1}^{2f}\omega_i^{-a_{i,2}-e+l},
\]
and we write 
$\phi'=\prod_{\sigma\in\Hom(k',\Fbar_l)}
\omega_\sigma^{\eta_\sigma}$, 
where $0\leq \eta_\sigma\leq l-1$ for each $\eta_\sigma$. This expression is unique
except that the special case where all the $\eta$ are 0 is indistinguishable
from the case when they are all $l-1$. Let us assume for the moment that we are
not in this special case and thus the expression is genuinely unique.

We write $\eta_i$ for $\eta_{\sigma_i}$. We then calculate that
\begin{align*}
\phi' (\phi')^c  &= (\phi \phi^c)|_{I_{K'}} \prod_{i=1}^{2f} \omega_i^{-a_{i,1}-1}
           \prod_{i=1}^{2f}\omega_i^{-a_{i,2}-e+l}
           =(\phi \phi^c)|_{I_{K'}} \prod_{i=1}^{2f}\omega_i^{-a_{i,1}-a_{i,2}-e+l-1}\\
           &=\prod_{i=1}^{2f} \omega_i^{a_{i,1}+a_{i,2}+e} \prod_{i=1}^{2f}\omega_i^{-a_{i,1}-a_{i,2}-e+l-1}
           =\prod_{i=1}^{2f}\omega_i^{a_{i,1}+a_{i,2}+e-a_{i,1}-a_{i,2}-e+l-1}\\
           &=\prod_{i=1}^{2f}\omega_i^{l-1}=1,
\end{align*}
so that
\begin{align*}
\phi' &= ((\phi')^c)^{-1}
=\left(\prod_{i=1}^{2f} \omega_i^{\eta_{i+f}}\right)^{-1} 
= \prod_{i=1}^{2f} \omega_i^{-\eta_{i+f}}= \prod_{i=1}^{2f} \omega_i^{l-1-\eta_{i+f}}.
\end{align*}
It follows from the uniqueness discussed above that $\eta_{i+f}=l-1-\eta_i$. 
For $i=1,\dots,f$, we let $\delta_i = \eta_i$. Then
we see that with this choice of $J$ and
$(\delta_\sigma)_{\sigma\in\Hom(k,\Fbar_l)}$, 
\begin{align*}
\psi&=\prod_{i=1}^f \omega_i^{a_{i,1}+1+\delta_{i}}
           \prod_{i=f+1}^{2f}\omega_i^{a_{i,2}+e-1-\delta_{i}}
      =\prod_{i=1}^f \omega_i^{a_{i,1}+1+\eta_i}
           \prod_{i=f+1}^{2f}\omega_i^{a_{i,2}+e-1-(l-1-\eta_i)}
           \\
      &=\prod_{i=1}^f \omega_i^{a_{i,1}+1} 
           \prod_{i=f+1}^{2f}\omega_i^{a_{i,2}+e-l}
           \prod_{i=1}^{2f}\omega_i^{\eta_i}
      =(\phi/\phi') \prod_{i=1}^{2f}\omega_i^{\eta_i}
\end{align*}
So $\psi = (\phi|_{I_{K'}}/\phi') \phi' = \phi|_{I_{K'}}$, as we
required.
Thus we are done apart from considering the special case we deferred
earlier, where $\phi'=1$.
Assume we are in this case, and put
\[\phi'' := \phi' \omega_0^{-a_{0,1}-1+a_{0,2}}\omega_f^{a_{0,1}+1-a_{0,2}}.
\]
We claim that $\phi''$ does not equal 1. To see this, since $\phi'=1$, we 
must show that $\phi''/\phi'\neq 1$.
We recall that 
$1\leq a_{0,1}+1-a_{0,2} \leq l$; since \[\phi''/\phi'=
\omega_0^{-a_{0,1}-1+a_{0,2}}\omega_f^{a_{0,1}+1-a_{0,2}}=\omega_0^{(l^f-1)(
  a_{0,1}+1-a_{0,2})}\]and $\omega_0$ has order $l^{2f}-1$, the claim follows.
Write $\phi''=\prod_{\sigma\in\Hom(k',\Fbar_l)}
\omega_\sigma^{\eta''_\sigma}$, where $0\leq \eta''_\sigma\leq l-1$
for each $\eta''_\sigma$. This expression is unique, since
$\phi''\neq1$.  Now, we calculate that \[(\phi'')(\phi'')^c = \phi'
(\phi')^c \omega_f^{a_{f,1}+1-a_{f,2}} \omega_0^{-a_{f,1}-1+a_{f,2}}
\omega_0^{a_{f,1}+1-a_{f,2}} \omega_f^{-a_{f,1}-1+a_{f,2}} = 1.\] We
conclude that $\eta''_i=l-1-\eta''_{i+f}$ in the same way as we saw
the corresponding fact for $\eta$ above.

We now take $J=\{0,\dots,f-1\}$, $\delta_i=\eta''_i$ for
$i=1,\dots,f-1$ and $\delta_0=e-1-\eta''_f$. Then
\begin{align*}
\psi&=\prod_{i=0}^{f-1} \omega_i^{a_{i,1}+1+\delta_{i}}
           \prod_{i=f}^{2f-1}\omega_i^{a_{i,2}+e-1-\delta_{i}} \\
       &=\omega_0^{a_{0,1}+1+e-1-\eta''_f} 
            \left(\prod_{i=1}^{f-1} \omega_i^{a_{i,1}+1+\eta''_{i}}\right)
            \omega_f^{a_{f,2}+e-1-(e-1-\eta''_f)} 
            \left(\prod_{i=f+1}^{2f-1}\omega_i^{a_{i,2}+e-1-(l-1-\eta''_{i})} \right)
            \\
            \displaybreak[1]
       &=\omega_0^{a_{0,1}+1-a_{0,2}} \omega_{2f}^{a_{0,2}+e-1-(l-1-\eta''_{2f})}  
            \left(\prod_{i=1}^{f-1} \omega_i^{a_{i,1}+1+\eta''_{i}}\right)
            \\&\hspace{3cm}
            \omega_f^{a_{0,2}-1-a_{0,1}} \omega_f^{a_{0,1}+1+\eta''_f} 
            \left(\prod_{i=f+1}^{2f-1}\omega_i^{a_{i,2}+e-1-(l-1-\eta''_{i})} \right)
            \\
            \displaybreak[1]
       &=\omega_0^{a_{0,1}+1-a_{0,2}} \omega_f^{a_{0,2}-1-a_{0,1}}
            \prod_{i=1}^{f} \omega_i^{a_{i,1}+1+\eta''_{i}}
            \prod_{i=f+1}^{2f}\omega_i^{a_{i,2}+e-1-(l-1-\eta''_{i})}
            \\
            \displaybreak[1]
       &=\omega_0^{a_{0,1}+1-a_{0,2}} \omega_f^{a_{0,2}-1-a_{0,1}}
           \prod_{i=1}^f \omega_i^{a_{i,1}+1} 
           \prod_{i=f+1}^{2f}\omega_i^{a_{i,2}+e-l}
           \prod_{i=1}^{2f}\omega_i^{\eta''_i}
           \\
       &=(\phi'/\phi'') (\phi|_{I_{K'}}/\phi') \prod_{i=1}^{2f}\omega_i^{\eta''_i}.
\end{align*}
So $\psi = (\phi'/\phi'') (\phi|_{I_{K'}}/\phi')\phi'' = \phi|_{I_{K'}}$, as we required.
\end{proof}

\begin{remark} Again, if $e=l-1$, it is no longer the case that for every $\rhobar$
with determinant $\prod_{\sigma\in
    \Hom(k,\Flbar)}\omega_\sigma^{a_{\sigma,1}+a_{\sigma,2}+e}$ can 
we can apply Lemma \ref{lem:explicit semisimple crystalline lifts in niveau 2}
to find a crystalline diagonal lift.
For a counterexample,
take $l=7$, $[k:\F_l]=2$, and label the two embeddings $k\into \Fbar_l$ as $\sigma_1$
and $\sigma_2$. Then take 
\begin{align*}
a_{\sigma_1,1} &= l-1,  & a_{\sigma_2,1} &= 1, \\
a_{\sigma_1,2}&=a_{\sigma_2,2} = 0, & \rhobar &= \Ind_{G_K}^{G_{K'}}\psi
\end{align*}
where $\psi: G_{K'}\to\GL_2(\Fbar_l)$ has 
$$\psi|_{I_{K'}}=\omega_{\tilde{\sigma_2}}^{l^3(l-1)+l^2 4+l(l-1)+(l-4)}$$
for $\tilde{\sigma}_2:k'\to\Fbar_l$ an embedding extending $\sigma_2$.
\end{remark}

\begin{lem}
  \label{lem:GHS agrees with bdj agrees with Schein where both defined}If $K/\Ql$ is
  unramified and $\rhobar$ 
is semisimple, then
  $\WBDJ(\rhobar)=\WSch(\rhobar)$. Similarly, if $K/\Ql$ is
  unramified and $\rhobar$ is
  reducible, then $\WBDJ(\rhobar)=\WGHS(\rhobar)$, and if $K$ is
  arbitrary, $\rhobar$ is reducible and $\rhobar$ is
  semisimple, then $\WSch(\rhobar)=\WGHS(\rhobar)$.  
\end{lem}
\begin{proof}This follows immediately from Lemmas \ref{lem:explicit
    semisimple crystalline lifts in niveau 1} and \ref{lem:explicit
    semisimple crystalline lifts in niveau 2}, together with the
  definitions of $\WBDJ(\rhobar)$, $\WSch(\rhobar)$ and $\WGHS(\rhobar)$.\end{proof}

This motivates
  the following definition.
\begin{defn}\label{defn: Wexplicit}
  Suppose that $K/\Ql$ is a finite extension, and that
  $\rhobar:G_K\to\GL_2(\Flbar)$ is a continuous representation. Then
  we define a set $\Wexplicit(\rhobar)$ of Serre weights as follows:
  \begin{itemize}

  \item If $\rhobar$ is irreducible, we set
    $\Wexplicit(\rhobar):=\WSch(\rhobar)$.
  \item If $\rhobar$ is reducible, we set
    $\Wexplicit(\rhobar):=\WGHS(\rhobar)$.
  \end{itemize}

\end{defn}
\begin{rem}
  It is an immediate consequence of Lemma \ref{lem:GHS agrees with bdj
    agrees with Schein where both defined} that if $\rhobar$ is
  semisimple then $\Wexplicit(\rhobar)=\WSch(\rhobar)$. 
\end{rem}
\begin{prop}
  \label{prop: GHS weights are pot diag weights}We have $\Wexplicit(\rhobar)\subset \Wdiag(\rhobar)$.
\end{prop}
\begin{proof}If $\rhobar$ is irreducible, this follows from Lemma
  \ref{lem:explicit semisimple crystalline lifts in niveau 2}. If
  $\rhobar$ is reducible, then this follows from the definition of
  $\WGHS(\rhobar)$, together with point (\ref{semisimp}) of the list of
  properties of $\sim$ in section \ref{sec:A lifting theorem}.
  \end{proof}

Having obtained a lower bound on $\Wdiag(\rhobar)$, we now consider
whether there are any obvious upper bounds. Here our results are
rather less complete. Firstly, we have the following conjecture.
\begin{conj}
  (\cite{GHS}) $\Wcris(\rhobar)=\Wexplicit(\rhobar)$.
\end{conj}
By Proposition \ref{prop: GHS weights are pot diag weights} we have
$\Wexplicit(\rhobar)\subset \Wcris(\rhobar)$, so to prove this conjecture
it would be enough to show that $\Wcris(\rhobar)\subset
\Wexplicit(\rhobar)$. This is presumably accessible to the
techniques of integral $l$-adic Hodge theory, but in the absence of
any further insight we suspect that an attempt to prove the result
would result in extensive unpleasant computation. In lieu of such
calculations, we recall what is known in the case that $K/\Ql$ is
unramified or highly ramified.

\begin{lem}
  \label{lem: crystalline of given weights implies det} Suppose that $\rhobar:G_K\to\GL_2(\Flbar)$ is a
  continuous irreducible representation, and that
  $a\in \Wcris(\rhobar)$ is a Serre weight. Then 
\[\det\rhobar|_{I_K}=\prod_{\sigma\in
    \Hom(k,\Flbar)}\omega_\sigma^{a_{\sigma,1}+a_{\sigma,2}+e}.\]
\end{lem}
\begin{proof}
  This follows immediately from the definition of $\Wcris(\rhobar)$ and
  Lemma \ref{lem:existence of crystalline chars}.
\end{proof}
\begin{lem}\label{lem: all weights are explicit if e is big enough}
  Suppose that $K$ has absolute ramification index $e\ge l$, and that
  $\rhobar$ is semisimple. 
Then $\Wcris(\rhobar)\subset
  \Wexplicit(\rhobar)$, so that $\Wcris(\rhobar)=\Wdiag(\rhobar)=\Wexplicit(\rhobar)$,
  and all three sets consist of precisely the set of Serre weights $a$
  with \[\det\rhobar|_{I_K}=\prod_{\sigma\in
    \Hom(k,\Flbar)}\omega_\sigma^{a_{\sigma,1}+a_{\sigma,2}+e}.\]
\end{lem}
\begin{proof}
  This is an immediate consequence of Lemma \ref{lem: crystalline of
    given weights implies det}, and Corollaries \ref{cor: if e is big then all weights are explicit,
    niveau 1} and \ref{cor: if e is big then all weights are explicit,
    niveau 2}.\end{proof}

\begin{defn}\label{defn:regular weight for FL} 
  We say that a Serre weight $a\in(\Z^2_+)^{\Hom(k,\Flbar)}$ is
  \emph{regular} if $a_{\sigma,1}-a_{\sigma,2}\le l-3$ for all
  $\sigma\in\Hom(k,\Flbar)$.

\end{defn}

\begin{lem}\label{lem: regular weights are explicit if K is unramified}If $K$ is absolutely unramified and
  $a\in(\Z^2_+)^{\Hom(k,\Flbar)}$ is a regular Serre weight, then if
  $a\in \Wcris(\rhobar)$ then $a\in \Wexplicit(\rhobar)$.
  
\end{lem}
\begin{proof}In the reducible case, this is a special case (the case $n=2$) of Lemma
  1.4.2 
of \cite{BLGGT} and the discussion immediately preceding it. In the
  irreducible case it is an immediate consequence of Theorem E of \cite{0807.1078}.\end{proof}

\begin{rem}

It is also possible to argue globally to obtain bounds on the set of
Serre weights by considering lifts of weight 0 and nontrivial type, as
was done in \cite{geebdj} and \cite{MR2430440}. In \cite{geeliusavitt}
these methods are combined with the results of this paper to
completely determine the set of Serre weights in the totally ramified
case; see Theorem \ref{thm: the main result of GLS - iff in totally
  ramified case} below.
\begin{defn} Let $e$ be a positive integer.
  We say that a Serre weight $a\in(\Z^2_+)^{\Hom(k,\Flbar)}$ is
  $e$-\emph{regular} if $a_{\sigma,1}-a_{\sigma,2}\le l-1-e$ for all
  $\sigma\in\Hom(k,\Flbar)$.\end{defn}
The arguments of
\cite{MR2430440} can presumably be carried over to the present setting
to prove the following analogue of Theorem 3.4 of
\cite{MR2430440}.Suppose that $\rbar:G_F\to\GL_2(\Flbar)$ is irreducible and modular of some Serre
  weight $a\in(\Z^2_+)_0^{\coprod_{v|l}\Hom(k_v,\Flbar)}$. Let $v|l$ be a
  place of $F$ such that $\rbar|_{G_{F_v}}$ is irreducible and the
  corresponding weight $a_v\in(\Z^2_+)^{\Hom(k_v,\Flbar)}$ is
  $e$-regular. Then $a_v\in \Wexplicit(\rbar|_{G_{F_v}})$.
\end{rem}

\section{The main theorems}\label{sec:main theorem}\subsection{}We now combine the
results of the previous sections to prove a variety of concrete theorems.

Fix an imaginary CM field $F$ with maximal totally real subfield
$F^+$, such that \begin{itemize}
\item $F/F^+$ is unramified at all finite places.
\item Every place $v|l$ of $F^+$ splits in $F$.
\item $[F^+:\Q]$ is even.
\end{itemize}
Let $\rbar:G_F\to\GL_2(\Flbar)$ be a continuous irreducible
representation which is modular in the sense of Definition \ref{defn:
  modular of some Serre weight}. In particular, $\rbar$ has split
ramification in the sense of Definition \ref{defn: galois split
  ramification}, and $\rbar^c\cong\rbar^\vee\epsilonbar_l^{-1}$. We
define sets $\WBDJ(\rbar)$ and $\Wexplicit(\rbar)$ of Serre weights as follows
(cf. Definition \ref{defn: serre weights global}). The set $\WBDJ(\rbar)$ is
only defined if $l$ is unramified in $F$.
\begin{defn}
  \label{defn: global sets of explicit weights}$\Wexplicit(\rbar)$
  (respectively $\WBDJ(\rbar)$) is
  the set of Serre weights
  $a\in(\Z^2_+)_0^{\coprod_{w|l}\Hom(k_w,\Flbar)}$ such that for each
  place $w|l$, the corresponding Serre weight
  $a_w\in(\Z^2_+)^{\Hom(k_w,\Flbar)}$ is an element of
  $\Wexplicit(\rbar|_{G_{F_w}})$ (respectively $\WBDJ(\rbar|_{G_{F_w}})$).
\end{defn}
\begin{rem}
  In fact $\WBDJ(\rbar)=\Wexplicit(\rbar)$ when both are defined, but as
  the definition of $\WBDJ(\rbar)$ is perhaps more familiar to the
  reader, we prefer to separate them.
\end{rem}
\begin{thm}
  \label{thm: explicit local lifts implies Serre
    weight}Let $F$ be an imaginary CM field with maximal totally real subfield
  $F^+$. Assume that $\zeta_l\notin F$, that $F/F^+$ is unramified at all finite places,
  that every place of $F^+$ dividing $l$ splits completely in $F$,
  and that $[F^+:\Q]$ is even. Suppose that $l>2$, and that
  $\rbar:G_F\to\GL_2(\Flbar)$ is an irreducible modular
  representation with split ramification. Assume that $\rbar(G_{F(\zeta_l)})$ is adequate.

 Let $a\in(\Z^2_+)_0^{\coprod_{w|l}\Hom(k_w,\Flbar)}$ be a
  Serre weight. Assume that $a\in \Wexplicit(\rbar)$. Then $\rbar$ is
  modular of weight $a$. 
\end{thm}
\begin{proof}
  By Proposition \ref{prop: GHS weights are pot diag weights},
  $a\in\Wdiag(\rbar)$, so the result follows from Theorem \ref{thm: potentially diagonalizable local lifts implies Serre
    weight}.\end{proof} We can make this result particularly explicit in the cases
where $l$ is either unramified or highly ramified in $F$. We say that
a Serre weight $a\in(\Z^2_+)_0^{\coprod_{w|l}\Hom(k_w,\Flbar)}$ is
\emph{regular} if for each $w|l$ the corresponding Serre weight $a_w$
is regular in the sense of Definition \ref{defn:regular weight for FL}.
\begin{thm}
  \label{thm:explicit BDJ thm}Let $F$ be an imaginary CM field with
  maximal totally real subfield $F^+$. Assume that $\zeta_l\notin F$,
  that $F/F^+$ is unramified at all finite places, that every place of
  $F^+$ dividing $l$ splits completely in $F$, and that $[F^+:\Q]$ is
  even. Assume that $l$ is unramified in $F$. Suppose that $l>2$,
  and that $\rbar:G_F\to\GL_2(\Flbar)$ is an irreducible modular
  representation with split ramification. Assume that
  $\rbar(G_{F(\zeta_l)})$ is adequate.

 Let $a\in(\Z^2_+)_0^{\coprod_{w|l}\Hom(k_w,\Flbar)}$ be a
  Serre weight. Assume that $a\in \WBDJ(\rbar)$. Then $\rbar$ is
  modular of weight $a$. Conversely, if $a$ is regular and $\rbar$ is
  modular of weight $a$, then $a\in\WBDJ(\rbar)$. 
\end{thm}
\begin{proof}
   By Definition \ref{defn: Wexplicit} and Lemma
  \ref{lem:GHS agrees with bdj agrees with Schein where both defined},
  $\WBDJ(\rbar)=\Wexplicit(\rbar)$. The result now follows from
  Theorem \ref{thm: explicit local lifts implies Serre
    weight}, Corollary \ref{cor: modular of some weight implies
    crystalline lifts exist} and Lemma \ref{lem: regular weights are explicit if K is unramified}.
\end{proof}

\begin{thm}
  \label{thm:explicit highly ramified}Let $F$ be an imaginary CM field with maximal totally real subfield
  $F^+$. Assume that $\zeta_l\notin F$, that $F/F^+$ is unramified at all finite places,
  that every place of $F^+$ dividing $l$ splits completely in $F$,
  and that $[F^+:\Q]$ is even. Assume that for each place $w|l$ of $F$
  the absolute ramification index of $F_w$ is at least $l$, and that
  $\rbar|_{G_{F_w}}$ is semisimple. Suppose that $l>2$, and that
  $\rbar:G_F\to\GL_2(\Flbar)$ is an irreducible modular
  representation with split ramification. Assume that $\rbar(G_{F(\zeta_l)})$ is adequate.

 Let $a\in(\Z^2_+)_0^{\coprod_{w|l}\Hom(k_w,\Flbar)}$ be a
  Serre weight. Then $\rbar$ is
  modular of weight $a$ if and only if for each $w|l$, 
\[\det\rbar|_{I_{F_w}}=\prod_{\sigma\in
    \Hom(k_w,\Flbar)}\omega_\sigma^{a_{\sigma,1}+a_{\sigma,2}+e}.\]
\end{thm}
\begin{proof}
  The necessity of the given condition follows from Corollary
  \ref{cor: modular of some weight implies crystalline lifts exist}
  and Lemma \ref{lem: all weights are explicit if e is big enough}, and
  the sufficiency from Theorem \ref{thm: explicit local lifts implies Serre
    weight} and Lemma \ref{lem: all weights are explicit if e is big
    enough} again.
\end{proof}
Finally, using the results of this paper together with potential
automorphy techniques and calculations with Breuil modules, the
following theorem is proved in \cite{geeliusavitt}.\begin{thm}
  \label{thm: the main result of GLS - iff in totally ramified case}Let
  $F$ be an imaginary CM field with maximal totally real subfield
  $F^+$, and suppose that $F/F^+$ is unramified at all finite places,
  that every place of $F^+$ dividing $l$ splits completely in $F$,
  that $\zeta_l\notin F$, and that $[F^+:\Q]$ is even. Suppose that
  $l>2$, and that $\rbar:G_F\to\GL_2(\Flbar)$ is an irreducible
  modular representation with split ramification such that
  $\rbar(G_{F(\zeta_l)})$ is adequate. Assume that for each place $w|l$
  of $F$, $F_w/\Ql$ is totally ramified.

 Let  $a\in(\Z^2_+)_0^{\coprod_{w|l}\Hom(k_w,\Flbar)}$ be a Serre weight. Then
 $a\in\Wexplicit(\rbar)$ if and only if $\rbar$ is modular of
 weight $a$.
\end{thm}
\appendix
\section{Adequacy}\label{app:adequacy}
\subsection{The definition}
\begin{defn}\label{defn:adequate} We call a finite
subgroup $H \subset \GL_n(\barFF_l)$ {\em adequate} if the following
conditions are satisfied.
\begin{enumerate}
\item $H$ has no non-trivial quotient of $l$-power order (i.e. $H^1(H,\barFF_l)=(0)$).
\item $l \ndiv n$.
\item The elements of $H$ with order coprime to $l$ span $M_{n \times
    n}(\barFF_l)$ over $\barFF_l$. (This implies that $\barFF_l^n$ is an irreducible representation of $H$.)
\item $H^1(H,\gothgl_n(\barFF_l))=(0)$.
\end{enumerate}
\end{defn}
(The notion of adequacy was introduced in \cite{jack}. The formulation above is as in \cite{BLGGT},
and while it is not identical to that in \cite{jack}, it is equivalent
to it by the discussion following the definition of adequacy in
Section 2.1 of \cite{BLGGT}.) 

\begin{remark} \label{coprime order and adequacy}
Note that if $l \ndiv \# H$ and $H$ acts irreducibly, then $H$ will be adequate,
as we now explain. The 
first statement in the definition of adequacy is trivial. 
For the second, observe that because $l \ndiv \# H$,
the tautological representation $H\to\GL_n(\barFF_l)$ will lift to characteristic zero,
and hence $n$ is the dimension of an irreducible characteristic zero representation of $H$ and
so divides $\#H$. It follows $l\ndiv n$. For the third, we see that elements of $H$ with 
order coprime to $l$ will just be all the elements of $H$ and will span  
$M_{n \times n}(\barFF_l)$ over $\barFF_l$ since $H$ acts irreducibly. For the fourth, we 
use Corollary 1 of section VIII.2 of \cite{MR554237}.
\end{remark}

A small point of notation: throughout this section, we will be considering 
subgroups of $\GL_n(\barFF_l)$ for some $n$, and we will often find it useful
to write $V$ for the vector space $\barFF_l^n$, especially considered as 
a representation of some subgroup of $\GL_n(\barFF_l)$ which should be clear 
from context.

The following lemmas will be useful. They were proved in the related context of bigness 
by Snowden and Wiles (see Propositions 2.1 and 2.2 of \cite{snw}), and the proofs 
generalize very straightforwardly. 
\begin{lem}\label{normal subgroups and adequacy}
Suppose $H \subset \GL_n(\barFF_l)$ is a finite subgroup, and $N \triangleleft H$
is a normal subgroup which is adequate and has $[H:N]$ prime to $l$. 
Then $H$ is adequate.
\end{lem}
\begin{proof} Points (1), (2) and (3) are trivial. There is an exact sequence
\[
 H^1(H/N,\gothgl_n(\barFF_l)^N) \lra H^1(H,\gothgl_n(\barFF_l)) \lra H^1(N, \gothgl_n(\barFF_l))^{G/H}
\]
Since $N$ is adequate, $H^1(N, \gothgl_n(\barFF_l))$ is trivial and so the right term vanishes.

Since $N$ is adequate, the standard representation of $N$ is irreducible (by condition (3)),
and thus $\gothgl_n(\barFF_l)^N=\barFF_l \mathbf{1}$ (this uses $l\ndiv n$). 
Then the left term in the exact sequence is just $H^1(H/N,\barFF_l)$ and vanishes since
$H/N$ has order prime to $l$ and hence no $l$ power quotients. Thus the middle term 
in the exact sequence vanishes, establishing (4).
\end{proof}

\begin{lem}\label{scalars and adequacy}
Suppose $H \subset \GL_n(\barFF_l)$ is a finite subgroup, and $k$ is a finite extension of $\F_l$. 
Then $H$ is adequate if and only if $k^\times H$ is adequate.
\end{lem}
\begin{proof} Since $H$ is a normal subgroup of $k^\times H$ of prime-to-$l$ index, the `only if'
part follows from the previous lemma. We now prove the other direction, assuming 
$k^\times H$ is adequate, and showing $H$ is adequate. Point (2) is trivial. For point (1),
let $K$ be a $l$-power order quotient of $H$. Since $k^\times \cap H$ has order prime to $l$,
it has trivial image in $K$. Thus $K$ is a quotient of the group $H/(H\cap k^\times) = k^\times H/ k^\times$. By assumption, $k^\times H$ has no nontrivial $l$-power quotient so $K$ is trivial and we have point (1). For point (3) note that the elements of $k^\times H$  
of prime-to-$l$ order will have the same $\Fbar_l$ span in $M_{n \times n}(\barFF_l)$ 
as those of $H$. 

For point (4), note that it will be enough to establish $H^1(H,\gothsl_n(\Fbar_l))=(0)$
(see the discussion immediately after the definition), and we may similarly assume 
$H^1(k^\times H,\gothsl_n(\Fbar_l))=(0)$. We have an exact sequence
\[ 1 \lra H \lra k^\times H \lra G \lra 1\]
for some quotient $G$ of $k^\times$. We therefore have an exact sequence
\[
H^1(k^\times H,\gothsl_n(\Fbar_l)) \lra H^1(H,\gothsl_n(\Fbar_l))^G \lra H^2(G, (\gothsl_n(\Fbar_l))^H).
\]
The left-hand group vanishes by our assumption that $k^\times H$ is adequate. 
Since $k^\times H$ is adequate,it acts irreducibly (by condition 3), and so (since $l \ndiv n$) 
we have that $(\gothsl_n(\Fbar_l))^H$ is trivial, thus the right hand group in the exact sequence
vanishes. It follows that the middle term vanishes (alternatively, it
vanishes because $G$ has order prime to $p$). One easily checks
that $G$ acts trivially on $ H^1(H,\gothsl_n(\Fbar_l))$, so we are done.
\end{proof}

\subsection{Adequacy for $\GL_2$} 

The aim of this subsection is to explicate the notion of adequacy for subgroups of $\GL_2$.
Theorem 9 of \cite{jackapp} already tells us that in characteristic greater than 5, `adequate' 
simply means `acts irreducibly', but we would like to have results for
characteristics $3$ and $5$. 
We prove that subgroups acting irreducibly are adequate apart from some explicit exceptions. 
More precisely, we prove the following proposition:

\begin{prop} 
\label{prop:adequacy for n=2} Suppose that $l>2$ is a prime, and 
that $G\leq \GL_2(\Fbar_l)$ is a finite subgroup which acts irreducibly on $\Fbar_l^2$. 
Then precisely one of the following is true:
\begin{itemize}
\item We have $l=3$, and the image of $G$ in $\PGL_2(\Fbar_3)$ is  conjugate to
$\PSL_2(\F_3)$.
\item We have $l=5$, and the image of $G$ in $\PGL_2(\Fbar_5)$ is  conjugate to
 $\PGL_2(\F_5)$ or $\PSL_2(\F_5)$.
\item $G$ is adequate.
\end{itemize}
\end{prop}

\begin{rem} For any $G$ as in the theorem,
its image in $\PGL_2(\Fbar_l)$, which we will call $\bar{G}$,
either must be isomorphic to one of $A_5$, $S_4$, 
$A_4$, or a dihedral group of order coprime to $l$, or must be 
conjugate to $\PSL_2(k)$ or $\PGL_2(k)$ for some finite 
extension $k$ of $\Fbar_l$ (see Theorem 2.47 (b) of \cite{ddt}).
  We show in the course of the proof that if $l=3$ (resp.\ $l=5$) and
  if $\bar G$ is isomorphic to $A_4$ (resp.\ $A_5$) then
  in fact, $\bar G$ is conjugate to $\PSL_2(\F_3)$ (resp.\
  $\PSL_2(\F_5)$). 
\end{rem}

\begin{proof}
The proof will be a very straightforward case analysis. On the one
hand, we have the list of possibilities for $\bar G$ recalled in the
previous remark.
We divide into cases according to which of these is 
true, further subdividing the $\PSL_2(k)$ and $\PGL_2(k)$ cases into the subcase 
where $|k|=l$ and
the subcase where $|k|>l$. On the other hand, we divide into cases according to
the value of $l$, considering the cases $l=3$, $l=5$ and $l\geq 7$. The resulting
`two dimensional' collection of cases is depicted in Figure \ref{tbl: adequacy cases}.
We will often give arguments which treat several cases in this collection 
at once, and the reader may find it useful to refer to Figure \ref{tbl: adequacy cases}
which summarizes which argument is used in which case. 
We will number the various points of the argument to make them easier to refer to.

But before we move into the detailed consideration of the cases, it will
be useful to discuss in a little more detail the cases where $\bar{G}$ 
is isomorphic to $A_4$ and $A_5$. Specifically, it will be important to us to establish
\begin{sublem} Let us write $2.A_4$ (resp $2.A_5$) for the binary
tetrahedral group (resp binary icosahedral group).
(Thus if we consider $A_5$ as the group of symmetries of an 
icosahedron, a subgroup of $\SO(3)$, then $2.A_5$ is the inverse image of $A_5$ under
the natural 2-to-1 map $\SU(2)\to\SO(3)$; and similarly for $A_4$ and the group of
symmetries of the tetrahedron.) 

Now suppose that $\bar{G}$ is isomorphic to $A_k$ for $k\in\{4,5\}$. 
Then we can find some representation $\widetilde{\phi}:2.A_k\to\FL_2(\Flbar)$
such that 
$\Fbar_l^\times \widetilde{\phi}(2.A_k) = \Fbar_l^\times G$.
\end{sublem}
\begin{proof}
Before we can begin the proof proper, we must 
recall some general facts from the theory of projective
modular representations of finite groups. Given any finite group $H$ and prime $l$, 
we call a group $\widetilde{H}$ an $l$-\emph{representation group}
of $H$, if \emph{(a)} $\widetilde{H}$ has a central subgroup $A$ contained in the 
commutator subgroup $\widetilde{H}'$ of $\widetilde{H}$, 
\emph{(b)} $\widetilde{H}/A\cong H$ and \emph{(c)}
$A\cong H^2(H,\Fbar_l^\times)$. 
We have the following facts.
\emph{(1.)} There always exists such a group (not necessarily unique).
\emph{(2.)} Given any such group $\widetilde{H}$, 
and given any homomorphism $\phi: H\to\PGL_n(\Fbar_l)$, there is a homomorphism $\widetilde{\phi}:\widetilde{H}\to\GL_n(\Fbar_l)$ such that the maps
$\widetilde{H} \onto H \to \PGL_n(\Fbar_l)$ and $\widetilde{H} \to \GL_n(\Fbar_l) \onto \PGL_n(\Fbar_l)$
agree. \emph{(3.)} Finally, the group $H^2(H,\Fbar_l^\times)$ is just
the prime-to-$l$ part of 
$H^2(H,\Qbar^\times)$, the \emph{Schur multiplier} of $H$.
[The original reference for these three facts is \cite{asano-osima-takahasi}, although the first
two have older proofs in characteristic 0 which essentially go over unchanged to 
characteristic $l$. The authors found a more accessible `reference' for the first (resp second) 
of these facts was to read the proof of Theorem 1.2 (resp 1.3) of \cite{hoffman1992projective},
which proves these results in characteristic 0, and observe that the proof goes through
in characteristic $l$. The third fact is \cite[Satz 1]{asano-osima-takahasi}.]

We wish to apply these facts in the case where $H$ is isomorphic to $A_n$ for $n\geq 4$. 
By the last sentence
of chapter 2 of \cite{hoffman1992projective} (on p23, just after
the unnumbered remark after Theorem 2.12) we see the construction of a group,
called there $\widetilde{A}_n$, which is a `representation group' for $A_n$. 
[This means---see the definition at the bottom of \cite[p6]{hoffman1992projective}---a 
group satisfying the properties \emph{(a--c)} of the previous paragraph,
except with $H^2(H,\Qbar^\times)$ replacing $H^2(H,\Fbar_l^\times)$.] 
Given the construction there\footnote{Specifically, a group $\tilde{S}_n$ is 
constructed---see Theorem 2.8 of \cite{hoffman1992projective}---
which is a double cover of $S_n$; $\widetilde{A}_n$ is defined as 
the inverse image of $S_n$ under this map.}, $\widetilde{A}_n$ is a double cover of $A_n$, and 
we conclude that $H^2(A_n,\Qbar^\times)=\Z/2\Z$. But then
$H^2(A_n,\Qbar^\times)\cong H^2(A_n,\Fbar_l^\times)$ (because 
$H^2(A_n,\Qbar^\times)=\Z/2\Z$, $l>2$ and using the fact (3) above) so 
$\widetilde{A}_n$ satisfies properties \emph{(a--c)} of the previous paragraph.
Thus $\widetilde{A}_n$ is in fact also an $l$-representation group of $A_n$, for $l>2$.

Now we begin to the proof proper, and imagine that $\bar{G}$ is, as in the statement
of the sublemma, isomorphic to $A_k$ for $k\in\{4,5\}$. 
By the discussion of the previous paragraph $\widetilde{A}_k$ is an $l$-representation 
group of $\bar G$, and so by fact (2) above applied with $\phi$ the natural
inclusion $\bar G\into\PGL_2(\Fbar_l)$, there is a map 
$\widetilde{\phi}:\widetilde{A}_k\to\GL_2(\Fbar_l)$
such that $\widetilde{A}_k \onto A_k \isoto \bar{G} \into \PGL_2(\Fbar_l)$ and 
$\widetilde{A}_k \overset{\widetilde{\phi}}{\to} \GL_2(\Fbar_l) \onto \PGL_2(\Fbar_l)$
agree, which means that $\Fbar_l^\times \widetilde{\phi}(\widetilde{A}_k) = \Fbar_l^\times G$.

This gives us everything we need, apart from checking this group $\widetilde{A}_k$ 
defined in \cite{hoffman1992projective},
is isomorphic to the group $2.A_k$ as defined in the statement of the sublemma.
To check this, observe $\widetilde{A}_n$ is defined in \cite{hoffman1992projective} 
as a certain subgroup of a certain group $\tilde{S}_n$,
which is given a presentation just before Theorem 2.8 of \emph{loc.~cit.}, on p18.
Comparing this presentation to the discussion in \S2.7.2 of \cite{wilsonfinitesimple},
we see that $\widetilde{A}_n$ is the same group as the group called $2.A_n$ in \cite{wilsonfinitesimple}. Examining the discussion in \S5.6.8 and \S5.6.2
of \cite{wilsonfinitesimple}, we see that the groups that book calls 
$2.A_5$ and $2.A_4$ are indeed respectively the 
binary icosahedral and tetrahedral groups.
\end{proof}

We are now ready to move on to the case analysis that is the proof proper.

\begin{figure}
\begin{tabular}{c|ccc}
		&$l=3$				&$l=5$				&$l\geq 7$\\
\hline
dihedral 	&\checkmark 1			& \checkmark 1		&\checkmark 0 \\
$S_4$	&$-$ 7			& \checkmark 1 	&\checkmark 0 \\
$A_4$	&$-$ 7			& \checkmark 1 	&\checkmark 0 \\
$A_5$	&\checkmark 4			&$-$ 7 		&\checkmark 0 \\
$\PSL_2(k)$, $|k|=l$&$\times$ 6	&$\times$ 5		&\checkmark 0 \\
$\PGL_2(k)$, $|k|=l$&\checkmark 3	&$\times$ 5		&\checkmark 0 \\
$\PSL_2(k)$, $|k|\geq l^2$&\checkmark 2&\checkmark 2	&\checkmark 0 \\
$\PGL_2(k)$, $|k|\geq l^2$&\checkmark 2&\checkmark 2	&\checkmark 0 \\
\end{tabular}
\bigskip

\begin{tabular}{r|p{8cm}}
\checkmark 0 & Always adequate by appeal to Theorem 9 of \cite{jackapp} (see point 0)\\
\checkmark $n$ & Always adequate; see point $n$.\\
$\times$ $n$ & Never adequate;  see point $n$. \\
$-$ $7$ & This case is already included in other cases, and hence needs not 
be considered in its own right. See point 7.
\end{tabular}
\caption{The various cases for the proof of Proposition \ref{prop:adequacy for n=2}
\label{tbl: adequacy cases}}
\end{figure} 

\emph{Point 0.} The majority of cases are handled by an appeal to 
Theorem 9 of \cite{jackapp}. In our present
notation, this asserts \emph{inter alia} 
that if we write $G^0$ for the subgroup of $G$ generated by elements
of $l$-power order and $d$ for the maximal dimension of an irreducible $G^0$-submodule of $\Flbar^2$,
then $G$ is adequate so long as $l\geq
2(d+1)$. 
Since clearly $d\leq 2$, we immediately see that $G$ is automatically adequate 
in any case with $l \geq 7$.   

\emph{Point 1.} Now we consider the case where either
\begin{itemize}
\item $l=5$ and $\bar{G}$ is isomorphic to $S_4$ or $A_4$. 
\item $l=3$ or $l=5$ and $\bar{G}$ is a dihedral group of prime-to-$l$ order
\end{itemize} 
In either of these cases, the projective image of $G$ has order coprime to $l$,
whence $G$ has order coprime to $l$, which is enough by Remark 
\ref{coprime order and adequacy}.

\emph{Point 2.} Next we consider the case where $l=3$ or 5 and the projective image of $G$
is $\PSL_2(k)$ or $\PGL_2(k)$ for some $k$ with $|k|\geq l^2$. 
We claim that $G$ is adequate in this case.

If the projective image of $G$ is $\PSL_2(k)$, then
by applying Lemma \ref{scalars and adequacy} we can replace $G$ with 
$(k)^\times G=k^\times \SL_2(k)$, and by applying Lemma  \ref{scalars and adequacy} again
we can replace $G$ with $\SL_2(k)$. If the projective image of $G$ is $\PGL_2(k)$, then
by a similar argument we can replace $G$ with $\GL_2(k)$ and then by applying 
Lemma \ref{normal subgroups and adequacy} we can again replace $G$ with $\SL_2(k)$.
Thus in either case we may assume that $G=\SL_2(k)$.

Let us verify the conditions for adequacy in turn:
\begin{itemize}
\item We see that $G$ has no non-trivial quotient of $l$ power order since the simplicity of 
$\PSL_2(\F_{3^n})$ and $\PSL_2(\F_{5^n})$ for $n\geq 2$ tells us $G$ in fact has no Jordan H\"older
constituent of $l$-power order.
\item The fact that $l\notdiv n=2$ is trivial.
\item Certainly the elements
  of $\SL_2(k)$ of order prime to $l$ span $M_{2\times 2}(\Flbar)$ as an
  $\Flbar$-vector space 
(one may use the matrices   $\begin{pmatrix}
    1&0\\0&1
  \end{pmatrix}$, $
  \begin{pmatrix}
    \alpha & 0\\ 0 &\alpha^{-1}
  \end{pmatrix}$, $
  \begin{pmatrix}
    0&1\\-1&0
  \end{pmatrix}$, $
  \begin{pmatrix}
    0&\alpha\\-\alpha^{-1}&0
  \end{pmatrix}$ for any $\alpha\in k^\times$, $\alpha\neq \pm 1$).

\item To verify the fourth condition it will suffice to check $H^1(G,\gothsl_n(\Fbar_l))=(0)$.
Since $G=\SL_2(k)$, this is just $H^1(\SL_2(k),\gothsl_n(\Fbar_l))=(0)$, which
follows, under our present assumptions, from Lemma 2.48 of
\cite{ddt}. 
\end{itemize}

\emph{Point 3.} 
We now turn to the case where $l=3$ and $\bar G$ is conjugate to
$\PGL_2(\F_3)$. We claim that $G$ is adequate in this case. Applying
Lemma \ref{scalars and adequacy} twice, we may assume that $G=\GL_2(\F_3)$. Since
$\PGL_2(\F_3)\cong S_4$, we see that $G$ has no quotients of $3$-power
order. Indeed, $S_4$ has 3 subgroups of index 3 and they are all
conjugate, being 2-Sylow subgroups. Thus the first condition for
adequacy holds. The second condition holds trivially. For the third
condition, we note that the elements 
  $\begin{pmatrix}
    1&0\\0&1
  \end{pmatrix}$, $
  \begin{pmatrix}
    0&-1\\1&0
  \end{pmatrix}$, $
  \begin{pmatrix}
    1&1\\1&-1
  \end{pmatrix}$, and $
  \begin{pmatrix}
    -1&1\\1&1
  \end{pmatrix}$ of $\SL_2(\F_3)$ are semi-simple and span $M_{2\times
    2}(\Fbar_3)$ as an $\Fbar_3$-vector space.
To verify the fourth condition, we
  think of $\SL_2(\F_3)$ as a normal subgroup of $\GL_2(\F_3)$ with quotient $Q$
  of order 2, giving us an exact sequence
\[
 H^1(Q,\gothgl_2(\Fbar_3)^{\SL_2(\F_3)}) \lra H^1(\GL_2(\F_3),\gothgl_2(\Fbar_3)) \lra H^1(\SL_2(\F_3), \gothgl_2(\Fbar_3))^{Q}.
\]
The right term vanishes by appeal to Lemma 2.48 of \cite{ddt},
  which tells us that $H^1(\SL_2(\F_3),\gothsl_2(\Fbar_3))$ is trivial. On the other hand
  $\gothgl_2(\Fbar_3)^{\SL_2(\F_3)} =\gothsl_2(\Fbar_3)^{\SL_2(\F_3)}\oplus (\mathbf{1}\Fbar_3)^{\SL_2(\F_3)}=\mathbf{1}\Fbar_3$ (since $\gothsl_2(\Fbar_3)$
  is irreducible and nontrivial under the action of $\SL_2(\F_3)$); and 
  $H^1(Q,\mathbf{1}\Fbar_3)=(0)$. So the left term vanishes too.
    Thus $H^1(\GL_2(\F_3),\gothgl_2(\Fbar_3))=(0)$; that is,  
  $H^1(G,\gothgl_2(\Fbar_3))=(0)$, as required.

\emph{Point 4.} 
We now treat the case where $l=3$ and $\bar G\cong A_5$. We claim $G$ is
adequate in this case. Applying the sublemma 
we can find some irreducible two-dimensional $\Flbar$-representation $\widetilde{\phi}$ of $2.A_5$ such that 
$\Fbar_l^\times \widetilde{\phi}(2.A_5) = \Fbar_l^\times G$.
Having done this, by applying Lemma \ref{scalars and adequacy} twice, we see
that to show $G$ adequate it suffices to show $\widetilde{\phi}(2.A_5)$ adequate.
By consulting \cite[p2]{jansen1995atlas}, we see that $2.A_5$ has only
two $2$-dimensional irreducible
mod 3 representations, corresponding to the Brauer characters $\phi_5$ and $\phi_6$ 
there.
By comparing with \cite[p2]{conway-atlas} we see
that these Brauer characters each come from characteristic 0 characters, viz
the characters called $\chi_6$ and $\chi_7$ in \cite[p2]{conway-atlas},
the first of which corresponds to $\rho_{\text{nat},2.A_5}$, the natural representation we get
by thinking of $2.A_5$ as the binary icosahedral group, and the second to 
$\rho^{(12)}_{\text{nat},2.A_5}$. This means that $\widetilde{\phi}$ is either the reduction mod 3
of $\rho_{\text{nat},2.A_5}$ or of $\rho^{(12)}_{\text{nat},2.A_5}$. We will write
$\bar\rho_{\text{nat},2.A_5}$ and $\bar\rho^{(12)}_{\text{nat},2.A_5}$ for these reductions.
 
 We shall now verify that $\widetilde{\phi}(2.A_5)$ is adequate, verifying the conditions in turn.
 \begin{itemize}
 \item The first condition (no $l$-power order quotients) follows immediately from the
 simplicity of $A_5$, which shows $\widetilde{\phi}(2.A_5)$ can have no $l$-power order 
 Jordan H\"older constituents.
 \item The second condition, $l\ndiv n$, is trivial.
 \item Examining \cite[p2]{jansen1995atlas}, we see that 
 the character $\phi_5$ is real, so the dual representation of 
$\bar\rho_{\text{nat},2.A_5}$ has the same character
and $\ad^0\bar\rho_{\text{nat},2.A_5}$ has character $\phi_5^2-1$.
We recognize this character as $\phi_2$ from the table. Thus $\ad^0\bar\rho_{\text{nat}}$
is irreducible. Similarly $\ad^0\bar\rho_{\text{nat},2.A_5}^{(12)}$ has character $\phi_3$, which
is irreducible. It follows that $\ad^0V = \gothsl_2 \Flbar$ is irreducible. 
Choose $g\in \widetilde{\phi}(2.A_5)$ to be the image under $\widetilde{\phi}$ of some
non-central element of $2.A_5$ of order prime to 3. 
Then $g$ is not a scalar and acts semisimply, so 
is conjugate to $\diag(\alpha,\beta)$ where $\alpha\neq\beta$. Then it is easy to check that
$\pi_{g,\alpha} \gothsl_2 (\Flbar) \iota_{h,\alpha}\neq (0)$. Thus we see
that condition (C) of \cite{jackapp} holds, which is equivalent to the third condition for
adequacy by Lemma 1 of \cite{jackapp}. 
\item To verify the fourth condition it will suffice to check
  $H^1(\widetilde{\phi}(2.A_5),\gothsl_2(\Fbar_l))=(0)$.  Recall that
  $\widetilde{\phi}$ is $\bar\rho_{\text{nat},2.A_5}$ or
  $\bar\rho^{(12)}_{\text{nat},2.A_5}$, both of which are easily seen
  to be injective. Thus we must show $H^1(2.A_5,\ad^0
  \widetilde{\phi})=(0)$ for
  $\widetilde{\phi}=\bar{\rho}_{\text{nat},2.A_5}$ and
  $\widetilde{\phi}=\bar{\rho}^{(12)}_{\text{nat},2.A_5}$. We give the
  argument for $\widetilde{\phi}=\bar{\rho}_{\text{nat},2.A_5}$, the
  other case being entirely analogous. It is easy to see that
  $\ad^0\bar{\rho}_{\text{nat},2.A_5}$ is the natural 3D
  representation $\bar{\rho}_{3,2.A_5}$ we get by mapping to $A_5$,
  realizing $A_5$ as the symmetries of a icosahedron, then reducing
  mod 3. By Proposition 46 in Section 16.4 of \cite{MR0450380}, we see
  that   $\ad^0\bar{\rho}_{\text{nat},2.A_5}$ is a projective
  $\Fbar_3[2.A_5]$-module, so it is the only simple module in its
  block, and in particular any extension of the trivial representation
  by  $\ad^0\bar{\rho}_{\text{nat},2.A_5}$ splits, as required. (We
  thank Florian Herzig for supplying us with this argument.)

\end{itemize}

\emph{Point 5.} Next we consider the case where $l=5$ and $\bar G$
is $\PSL_2(\F_5)$ or $\PGL_2(\F_5)$. $G$ is adequate in neither case.
In the case where $\bar{G}$ is $\PSL_2(\F_5)$,
Table 4.5 of \cite{cps} tells us that $H^1(G,\gothgl_2(\barFF_l))$
is
one dimensional, violating the fourth condition in the definition of adequacy. Thus in this case $G$ 
will fail to be adequate.  The case where $\bar{G}$ is $\PGL_2(\F_5)$ 
will then also have $H^1(G,\gothgl_n(\barFF_l))\neq(0)$ by \cite[2.3 (g)]{cps}, and again
$G$ will fail to be adequate.

\emph{Point 6.} Next we consider the case where $l=3$ and $\bar G$
is conjugate to $\PSL_2(\F_3)$.We claim that $G$ is not adequate in this case.
Since
$\PSL_2(\F_3)\cong A_4$, it suffices to
note that $A_4$ has a quotient of order $3$, so that
$G$ must also have a quotient of order $3$.
This violates the first condition for adequacy.

\emph{Point 7.} We now treat the remaining cases. We start with the case where
where $l=5$ and
$\bar G\cong A_5$.
It is obvious that this case \emph{includes} the case already 
considered where we have (up to conjugation) an equality 
$\bar G = \PSL_2(\F_5)$ 
(rather than a mere isomorphism), since $A_5\cong \PSL_2(\F_5)$.
But we will show that in fact whenever 
$\bar G\cong A_5$ we must indeed have $\bar G = \PSL_2(\F_5)$ up to conjugation, 
thus reducing this case
to a case we have already considered.

Applying the sublemma 
we can find some irreducible mod 5 representation $\widetilde{\phi}$ of $2.A_5$ such that 
$\Fbar_l^\times \widetilde{\phi}(2.A_5) = \Fbar_l^\times G$.
Having done this, by applying Lemma \ref{scalars and adequacy} twice, we see
that to show $G$ inadequate it suffices to show $\widetilde{\phi}(2.A_5)$ inadequate.
By consulting \cite[p2]{jansen1995atlas}, we see that $2.A_5$ has only one mod 5
Brauer character of dimension 2. But $2.A_5\isoto \SL_2(\F_5)\into\GL_2(\F_5)$
(see \cite[p2]{conway-atlas}) is clearly an irreducible representation mod 5 of
dimension 2, so we deduce that $\widetilde{\phi}$ must be exactly this map. This reduces us to the
case $\bar G=\PSL_2(\F_5)$.

Similar arguments allow us to see that 
the (apparently more general) 
case where $l=3$ and $\bar G\cong A_4$ is actually included in 
the case that
$\bar G$ is conjugate to $\PSL_2(\F_3)$. 

Finally, again using similar arguments, we can 
reduce the case where $l=3$ and $\bar G\cong S_4$ to the case where 
$\bar G$ is conjugate to $\PGL_2(\F_3)$.

\end{proof}

\subsection{Adequacy for tensor products}We would like to thank
Richard Taylor for allowing us to include the following lemma here; it
was originally proved by him during the writing of \cite{BLGGT}.

\begin{lem}\label{lem:adequacy for tensor products} Suppose that $\Gamma$ is a group and that $r_i:\Gamma \ra \GL_{n_i}(\barFF_l)$ is a representation of $\Gamma$ for $i=1,2$. Suppose moreover that $r_1(\Gamma)$ is adequate, that $r_2|_{\ker r_1}$ is irreducible and that $r_2(\Gamma)$ has order prime to $l$. Then $(r_1 \otimes r_2)(\Gamma)$ is adequate. \end{lem}

\begin{proof} Write $H_i$ for the image of $r_i$ and $H$ for the image
  of $r_1 \otimes r_2$. Write $K_i$ for $r_i(\ker r_{3-i})$. Write $Z$
  for the set of $z \in \barFF_l^\times$ for which there exists
  $\gamma \in \Gamma$ with $r_1(\gamma)=z$ and
  $r_2(\gamma)=z^{-1}$. Then there is a natural identification
\[ H_1/K_1=\Gamma/(\ker r_1).(\ker r_2)=H_2/K_2 \]
and an exact sequence
\[ \{ 1\} \lra Z \lra \{ (h_1,h_2) \in H_1 \times H_2: \,\, h_1 \bmod K_1 = h_2 \bmod K_2\} \lra H  \lra \{ 1\}. \]
In particular there is an exact sequence
\[ \{ 1\} \lra Z \lra H_1 \lra H/K_2  \lra \{ 1\}. \]

It is easy to check the first two conditions for $H$ to be
adequate. (Note that $\dim r_2 | \# H_2$, so that $l\nmid\dim r_2$,
and that any $l$-power order quotient of $H$ would yield an $l$-power
order quotient of $H/K_2\cong H_1/Z$ and thus of $H_1$, a
contradiction.) 
To check the third
condition, suppose that $A_i \in M_{n_i \times n_i}(\barFF_l)$. We can
write
\[ A_1 = \sum_i a_i r_1(\gamma_i) \]
for some $a_i \in \barFF_l$ and $\gamma_i \in \Gamma$ with $r_1(\gamma_i)$ semi-simple. We can also write
\[ r_2(\gamma_i^{-1})A_2 = \sum_j b_{ij} r_2(\delta_{ij}) \]
for some $b_{ij} \in \barFF_l$ and some $\delta_{ij} \in \ker r_1$. Then
\[ \begin{array}{rl} &\sum_{i,j} a_i b_{ij} (r_1\otimes r_2)(\gamma_i\delta_{ij}) \\ = & \sum_i a_i r_1(\gamma_i) \otimes (r_2(\gamma_i) \sum_j b_{ij} r_2(\delta_{ij})) \\ = & \sum_i a_i r_1(\gamma_i) \otimes A_2 \\ = & A_1 \otimes A_2. \end{array} \]
Moreover each $r_1(\gamma_i\delta_{ij})=r_1(\gamma_i)$ is semi-simple by assumption and each $r_2(\gamma_i\delta_{ij})$ is semi-simple as $H_2$ has order prime to $l$. Thus $H$ satisfies the third condition to be adequate.

To check the fourth condition it suffices by the Hochschild-Serre
spectral sequence to check that $H^1(H/K_2, \ad (r_1 \otimes r_2)^{K_2})=(0)$ and
$H^1(K_2,\ad(r_1 \otimes r_2))^H=(0)$. However
\[ H^1(H/K_2, \ad (r_1 \otimes r_2)^{K_2})=H^1(H/K_2, \ad
r_1)=H^1(H_1/Z,\ad r_1)=H^1(H_1, \ad r_1)=(0) \]
and
\[ H^1(K_2,\ad(r_1 \otimes r_2))^H= ((\ad r_1) \otimes H^1(K_2, \ad
r_2))^H=(0) \](since $K_2$ has order prime to $l$).
The lemma follows.
\end{proof}

\subsection{An improvement to a lifting result of \cite{BLGGT}}We now
prove a slight variant of Theorem 4.3.1 of \cite{BLGGT}. At the
expense of assuming that the representation $\rbar$ admits a
potentially automorphic lift, we are able to weaken the assumption on
the prime $l$.
We will follow the proof of Theorem 4.3.1 of
\cite{BLGGT}, and in particular we refer to \cite{BLGGT}
for any notation not already defined in the present paper.

\begin{thm}\label{diaglift} Let $n$ be a positive integer and $l$ an odd  prime. Suppose that $F$ is a CM field not containing $\zeta_l$ and with maximal totally real subfield $F^+$.
  Let $S$ be a finite set of finite places of $F^+$ which split in $F$
  and suppose that $S$ includes all places above $l$. For each $v \in
  S$ choose a prime $\tv$ of $F$ above $v$.

Let $\mu:G_{F^+} \ra \barQQ_l^\times$ be a continuous, totally odd, de
Rham character unramified outside $S$. Also let
\[ \barr: G_{F^+} \lra \CG_n(\barFF_l) \] be a continuous
representation unramified outside $S$ with $\nu \circ \barr = \barmu$
and $\barr^{-1}\CG^0_n(\barFF_l)=G_F$.  Suppose that
$\breve{\barr}|_{G_{F(\zeta_l)}}$ is irreducible, and that
$\breve{\rbar}(G_{F(\zeta_l)})$ is adequate.

For $v \in S$, let $\rho_v:G_{F_\tv} \ra \GL_n(\CO_{\barQQ_l})$ denote
a lift of $\breve{\barr}|_{G_{F_\tv}}$. If $v|l$ we assume that
$\rho_v$ is potentially diagonalizable and that, for all
$\tau:F_\tv\into\Qlbar$, the multiset $\HT_\tau(\rho_v)$ consists of
$n$ distinct integers.

Assume further that there is a finite extension of CM fields $F'/F$ and
a RAECSDC automorphic representation $(\pi',\chi')$ of $\GL_n(\A_{F'})$ such
that
\begin{itemize}
\item $F'$ does not contain $\zeta_l$,
\item $(\pi',\chi')$ is unramified outside the set of primes above $S$,
\item $(\rbar_{l,\imath}(\pi'),\rbar_{l,\imath}(\chibar'))\cong
  (\rbar|_{G_{F'}},\mubar|_{G_{F'}})$, 
\item for all places $w|l$ of $F'$, $r_{l,\imath}(\pi')|_{G_{F'_w}}$
  is potentially diagonalizable, and
\item $\breve{\rbar}(G_{F'(\zeta_l)})$ is adequate.
\end{itemize}

Then there
is a lift
\[ r:G_{F^+} \lra \CG_n(\CO_{\barQQ_l}) \]
of $\barr$ such that
\begin{enumerate}
\item $\nu \circ r = \mu$;
\item if $v \in S$ then $\breve{r}|_{G_{F_\tv}} \sim \rho_v$;
\item $r$ is unramified outside $S$;
\item $r|_{G_{F'}}$ is automorphic of level potentially prime to $l$.
\end{enumerate} \end{thm}

\begin{proof}
We begin the proof with some brief remarks that may help to orient the
reader. In comparison to Theorem 4.3.1 of \cite{BLGGT}, we have
weakened the hypothesis that $l\ge 2(d+1)$, where $d$ is the maximal
dimension of an irreducible subrepresentation for the subgroup of
$\rbar(G_{F(\zeta_l)})$ generated by elements of order $l$, to
the hypothesis that $\rbar(G_{F(\zeta_l)})$ is adequate (this condition is implied
by the assumption that $l\ge 2(d+1)$ by Theorem 9 of
\cite{jackapp}). On the other hand, we have had to add the hypothesis
that $\rbar|_{G_{F'}}$ is automorphic. In the proof of Theorem 4.3.1
of \cite{BLGGT}, an appeal is made to Proposition 3.3.1 of
\emph{op. cit.}, which proves that $\rbar$ is potentially
automorphic. We do not know whether Proposition 3.3.1 can be proved
using only the condition that $\rbar(G_{F(\zeta_l)})$ is adequate,
rather than the condition that $l\ge 2(d+1)$; the difficulty lies in
establishing when the induction of an adequate representation is
adequate. 

The proof below is essentially a combination of the proofs of Theorem
4.3.1 and Proposition 3.3.1 of \cite{BLGGT}. The reason that we need
to incorporate details of the proof of Proposition 3.3.1 of
\cite{BLGGT} is that in addition to proving the potential automorphy
of $\rbar$, the Proposition also shows that $\rbar$ potentially admits
an ordinary automorphic lift with prescribed behaviour at places not
dividing $l$. In order to carry out the rest of the proof of Theorem
4.3.1 of \cite{BLGGT} in our setting, we need to produce such a lift
of $\rbar|_{G_{F'}}$, possibly after making a further solvable base
change. We can do this using the techniques of \cite{BLGGT}.

In outline, we do the following: we choose a solvable CM extension
$F_1/F'$ with various helpful local properties. We then use the
methods of \cite{BLGGT} to produce an ordinary automorphic lift $r_1$ of
$\rbar|_{G_{F_1}}$. The arguments of \cite{gg}, as refined in
\cite{jack} and \cite{BLGGT}, allow us to replace this with an
ordinary automorphic lift $r_{l,\imath}(\pi'_1)$ which has the
behaviour prescribed for $r$ at places not dividing $l$. The
techniques of \cite{BLGGT} then allow us to produce the representation
$r$, and the automorphicity of $r|_{G_{F'}}$ follows as a byproduct of
the construction.

We now begin the proof proper. We may suppose that for $v \in S$ with $v \ndiv l$ the
  representation $\rho_v$ is robustly smooth (see Lemma 1.3.2 of \cite{BLGGT}) and
  hence lies on a unique component $\CC_v$ of
  $R^\Box_{\breve{\barr}|_{G_{F_\tv}}} \otimes \barQQ_l$.  If $v|l$ is
  a place of $F^+$ then choose a finite extension $K_v/F_\tv$ over
  which $\rho_v$ becomes crystalline, and let $\CC_v$ denote the
  unique component of $R^\Box_{\breve{\barr}|_{G_{F_\tv}}, \{ \HT_\tau(\rho_v)
    \}, K_v-\cris} \otimes \barQQ_l$ on which $\rho_v$ lies.

Let $\tmu$ denote the Teichmuller lift of $\barmu$. Choose a
  positive integer $m$ which is greater than one plus the difference
  of every two Hodge-Tate numbers of $\rho_v$ and of
  $r_{l,\imath}(\pi')|_{G_{F'_w}}$ for every place $v|l$ of $F$ and
  every place $w|l$ of $F'$.
 
  Choose a finite, soluble, Galois, CM extension $F_1/F'$ which is
  linearly disjoint from $\barF^{\ker \barr|_{G_{F'}}}(\zeta_l)$ over
  $F'$ such that
\begin{itemize}
\item for all $u$ lying above $S$ we have $\barr(G_{F_{1,u}})=\{ 1\}$;
\item for all $u|l$ we have $\zeta_l \in F_{1,u}$;
\item $\mu|_{G_{F_1^+}}$ is crystalline above $l$;
\item $r_{l,\imath}(\chi')|_{G_{F_1^+}}$ is crystalline above $l$;
\item if $u|\tv|l$ with $v \in S$ then $\rho_v|_{G_{F_{1,u}}}$ is
  crystalline and $\rho_v|_{G_{F_{1,u}}} \sim \psi_{1}^{(u)} \oplus
  \dots \oplus \psi_{n}^{(u)}$ with each $\psi_{i}^{(u)}$ a
  crystalline character;
\item if $u|l$  then $r_{l,\imath}(\pi')|_{G_{F_{1,u}}}$ is
  crystalline and $r_{l,\imath}(\pi')|_{G_{F_{1,u}}} \sim \phi_{1}^{(u)} \oplus
  \dots \oplus \phi_{n}^{(u)}$ with each $\phi_{i}^{(u)}$ a
  crystalline character.
\end{itemize}
We can and do assume that
$(\phi_{i}^{(cu)})^c\phi_{i}^{(u)}=r_{l,\imath}(\chi')\epsilon^{1-n}|_{G_{F_1,u}}$. If
$u|\tv|l$ with $v \in S$, then for $i=1,\ldots,n$, we define
$\psi_{i}^{(cu)} : G_{F_{1,cu}}\ra \Qlbar^\times$ by
$(\psi_{i}^{(cu)})^c\psi_{i}^{(u)}=\mu|_{G_{F_1,u}}$.

Choose a CM extension $M/F_1$ such that
\begin{itemize}
\item $M/F_1$ is cyclic of  degree $n$;
\item $M$ is linearly disjoint from $\barF^{\ker
    \barr|_{G_{F'}}}(\zeta_l)$ over $F'$;
\item and all primes of $F_1$ above $l$ split completely in $M$.
\end{itemize}
Choose a prime $u_q$ of $F_1$ above a rational prime $q$ such that
\begin{itemize}
\item $q \neq l$ and $q$ splits completely in $M$;
\item $\barr$ is unramified above $q$.
\end{itemize}
If $v|ql$ is a prime of $F_1$ we label the primes of $M$ above $v$ as
$v_{M,1},\dots,v_{M,n}$ so that $(cv)_{M,i}=c(v_{M,i})$. Choose continuous characters
\[ \theta,\theta',\theta'':G_M \lra \barQQ_l^\times \]
such that
\begin{itemize}
\item the reductions $\bartheta$, $\thetabar'$ and $\thetabar''$ are equal;
\item $\theta\theta^c=\tmu \omega_l^{(n-1)m}\epsilon_l^{(1-n)m}$, $\theta'(\theta')^c=\mu$, and
  $\theta''(\theta'')^c=r_{l,\imath}(\chi')\epsilon_l^{1-n}$;
\item $\theta$, $\theta'$ and $\theta''$ are de Rham;
\item if $\tau:M \into \barQQ_l$ lies above a place $v_{M,i}|l$ of $M$
  then $\HT_\tau(\theta)=\{(i-1)m\}$ ,
  $\HT_{\tau}(\theta')=\HT_{\tau|_{F_1}}(\psi_{i}^{(v_{M,i}|_{F_1})})$
  and $\HT_\tau(\theta'')=\HT_{\tau|_{F_1}}(\phi_{i}^{(v_{M,i}|_{F_1})})$;
\item $\theta$, $\theta'$ and $\theta''$ are unramified at $u_{q,M,i}$
  for $i>1$, but $q$ divides $\#\theta(I_{M_{u_{q,M,1}}})$,
  $\#\theta'(I_{M_{u_{q,M,1}}})$ and
$\#\theta''(I_{M_{u_{q,M,1}}})$.
\end{itemize}
(Use Lemma 4.1.6 of \cite{cht}.)

Note the following:
\begin{itemize}
\item If $u|l$ is a place of $F_1$ and if $K/F_{1,u}$ is a finite extension over which $\theta$, $\theta'$ and $\theta''$ become crystalline and $\bartheta=\bartheta'=\thetabar''$ become trivial, then 
\[ (\Ind_{G_M}^{G_{F_1}} \theta)|_{G_{K}} \sim 1 \oplus \epsilon_l^{-m} \oplus \dots\oplus \epsilon_l^{(1-n)m}, \]
\[ (\Ind_{G_M}^{G_{F_1}} \theta')|_{G_{K}} \sim
\psi_{1}^{(u|_{F_1})}|_{G_K} \oplus \dots \oplus
\psi_{n}^{(u|_{F_1})}|_{G_K}, \] and
\[ (\Ind_{G_M}^{G_{F_1}} \theta'')|_{G_{K}} \sim \phi_{1}^{(u|_{F_1})}|_{G_K} \oplus \dots \oplus \phi_{n}^{(u|_{F_1})}|_{G_K}. \]

\item $(\Ind_{G_M}^{G_{F_1}} \theta)^c \cong (\Ind_{G_M}^{G_{F_1}}
  \theta)^\vee \otimes \tmu \omega_l^{(n-1)m}\epsilon_l^{(1-n)m}$,
  $(\Ind_{G_M}^{G_{F_1}} \theta')^c \cong (\Ind_{G_M}^{G_{F_1}}
  \theta')^\vee \otimes \mu$ and $(\Ind_{G_M}^{G_{F_1}} \theta'')^c \cong (\Ind_{G_M}^{G_{F_1}} \theta'')^\vee \otimes r_{l,\imath}(\chi') \epsilon_l^{1-n}$.

\item The representation 
\[ (\Ind_{G_M}^{G_{F_1}} \bartheta )|_{\ker\breve{\rbar}|_{G_{F_1(\zeta_l)}}}\]
is irreducible, and hence by Lemma \ref{lem:adequacy for tensor products}
\[  (\breve{\barr}|_{G_{F_1}} \otimes (\Ind_{G_M}^{G_{F_1}} \bartheta ))(G_{F_1(\zeta_l)}) \]
is adequate.

 [That $(\Ind_{G_M}^{G_{F_1}} \bartheta
 )|_{\ker\breve{\rbar}|_{G_{F_1(\zeta_l)}}}$ is irreducible follows from looking at
 ramification above $u_q$, and noting that $\rbar$ is unramified at
 $q$, so that $u_q$ is unramified in $\overline{F}^{\ker\breve{\rbar}|_{G_{F_1(\zeta_l)}}}$.]

\end{itemize}

Let $F_2/F_1$ be a finite, soluble, Galois, CM extension linearly
disjoint from $\barF_1^{\ker \Ind_{G_M}^{G_{F_1}} \bartheta} \barF^{\ker \breve{\barr}|_{G_{F_1}}}(\zeta_l)$
over $F_1$ such that
\begin{itemize}
\item $\theta|_{G_{F_2M}}$, $\theta'|_{G_{F_2M}}$ and $\theta''|_{G_{F_2M}}$ are crystalline above $l$ and unramified away from $l$;
\item $MF_2/F_2$ is unramified everywhere.\end{itemize}

Then there is a RAECSDC automorphic representation $(\pi_2,\chi_2)$ of $GL_{n^2}(\A_{F_2})$ such
that
\begin{itemize}
\item $r_{l,\imath}(\pi_2) \cong (r_{l,\imath}(\pi')|_{G_{F_1}} \otimes \Ind_{G_{M}}^{G_{F_1}} \theta)|_{G_{F_2}}$;
\item $r_{l,\imath}(\chi_2)= \tmu \omega_l^{(n-1)m}\epsilon_l^{(n-1)(n-m)}r_{l,\imath}(\chi')\delta_{F_2/F_2^+}$;
\item $\pi_2$ is unramified above $l$ and outside $S$.
\end{itemize}
[The representation $\pi_2$ is the automorphic induction of
$(\pi')_{MF_2}\otimes (\phi |\,\,|^{n(n-1)/2}\circ \det)$
to $F_2$, where $r_{l,\imath}(\phi)= \theta|_{G_{F_2M}}$. The first
two properties are clear.
The third property follows by the choice of $F_2$ and local-global
compatibility (\cite{ana}, \cite{blggtlocalglobalII}).]

Let $\tS$ denote the set of $\tv$ as $v$ runs over $S$, let $S_1$
(resp. $S_2$)
denote the primes of $F_1^+$ (resp. $F_2^+$) above $S$ and
$\tS_1$ (resp. $\tS_2$) the primes of $F_1$ (resp. $F_2$)
above $\tS$. If $v\in S_1$ (resp. $S_2$), let $\tv$ denote the
element of $\tS_1$ (resp. $\tS_2$)
lying above it. For $v\in S_1$ with $v \ndiv l$ (resp. $v|l$) let
$\CC_{1,v}$ denote the unique component of
$R^\Box_{\breve{\barr}|_{G_{F_{1,\tv}}}} \otimes \barQQ_l$
(resp. $R^\Box_{\breve{\barr}|_{G_{F_{1,\tv}}},
  \{0,m,2m,\dots,(n-1)m)\}, \cris} \otimes
\barQQ_l$) containing $r_{l,\imath}(\pi')|_{G_{F_{1,\tv}}}$
(resp. $1\oplus \epsilon_l^{-m}\oplus\cdots\oplus\epsilon_l^{(1-n)m}$).  For $v \in S_2$ with $v \ndiv l$ (resp. $v|l$) let
$\CC_{2,v}$ denote the unique component of
$R^\Box_{\barr_{l,\imath}(\pi_2)|_{G_{F_{2,\tv}}}} \otimes \barQQ_l$
(resp. $R^\Box_{\barr_{l,\imath}(\pi_2)|_{G_{F_{2,\tv}}},
  \{\HT_\tau(r_{l,\imath}(\pi_2)|_{G_{F_{2,\tv}}})\}, \cris} \otimes
\barQQ_l$) containing $r_{l,\imath}(\pi_2)|_{G_{F_{2,\tv}}}$. Choose a
finite extension $L/\Q_l$ in $\barQQ_l$ such that
\begin{itemize}
\item $L$ contains the image of each embedding $F_2 \into \barQQ_l$;
\item $L$ contains the image of $\theta$;
\item $r_{l,\imath}(\pi_2)$ is defined over $L$;
\item each of the components $\CC_{1,v}$ for $v \in S_1$ and $\CC_{2,v}$ for $v \in S_2$ is defined over $L$.
\end{itemize}
Set
\[ s = \Ind_{G_{M^+},G_{M}}^{G_{F_1^+},G_{F_1}, r_{l,\imath}(\chi')\epsilon_l^{1-n}} (\theta'', r_{l,\imath}(\chi') \epsilon_l^{1-n}): G_{F_1^+} \lra \CG_n(\CO_L) \]
in the notation of section 1.1 of \cite{BLGGT} and section 2.1 of \cite{cht}. Thus $\nu \circ s = r_{l,\imath}(\chi')\epsilon_l^{1-n}$.
For $v \in S_1$ (resp. $v \in S_2$) let $\calD_{1,v}$ (resp. $\calD_{2,v}$) denote the deformation problem for $\breve{\barr}|_{G_{F_{1,\tv}}}$ (resp. $\barr_{l,\imath}(\pi_2)|_{G_{F_{2,\tv}}}$) over $\CO_L$ 
corresponding to $\CC_{1,v}$ (resp. $\CC_{2,v}$). Also let 
\[ \CS_1=(F_1/F_1^+,S_1,\tS_1, \CO_L,\barr|_{G_{F_1^+}},\tmu|_{G_{F_2^+}} \omega_l^{(n-1)m}\epsilon_l^{(1-n)m}, \{ \calD_{1,v}\}) \]
and
\[ \CS_2=(F_2/F_2^+,S_2,\tS_2, \CO_L, \tilde{\barr}_{l,\imath}(\pi_2),\tmu|_{G_{F_2^+}} \omega_l^{(n-1)m}\epsilon_l^{(1-n)(m+1)}r_{l,\imath}(\chi')\delta_{F_2/F_2^+}, \{ \calD_{2,v}\}). \]

There is a natural map
\[ R_{\CS_2}^\univ \lra R_{\CS_1}^\univ \]
induced by $r_{\CS_1}^\univ|_{G_{F_2^+}} \otimes s|_{G_{F_2^+}}$. [We must check that if $u \in S_2$ 
then $\breve{r}_\CS^\univ|_{G_{F_{2,\tu}}} \otimes
(\Ind_{G_M}^{G_{F_1}} \theta'')|_{G_{F_{2,\tu}}} \in \calD_{2,u}$. Let
$v=u|_{F^+}$ and let $\rho^\Box_{v,\CC_{1,v}}$ denote the universal lift
of $\barr|_{G_{F_{1,\tv}}}$ to $R_{\CO_L,\breve{\barr}|_{G_{F_{1,\tv}}},\CC_{1,v}}^\Box$.
It suffices to show that $\rho^\Box_{v,\CC_{1,v}}|_{G_{F_{2,\tu}}} \otimes
(\Ind_{G_M}^{G_{F_1}} \theta'')|_{G_{F_{2,\tu}}} \in \calD_{2,u}$. For
this, it suffices to show if $\rho : G_{F_{1,\tv}} \ra
\GL_n(\CO_{\Qlbar})$ is a lift of $\breve{\barr}|_{G_{F_{1,\tv}}}$ lying on
$\CC_{1,v}$, then $\rho|_{G_{F_{2,\tu}}} \otimes (\Ind_{G_M}^{G_{F_1}}
\theta'')|_{G_{F_{2,\tu}}}$ lies on $\CC_{2,u}$. 
If $u|l$, then $\rho|_{G_{F_{2,\tu}}} \sim
(\Ind_{G_M}^{G_{F_1}}\theta)|_{G_{F_{2,\tu}}}$ and $ (\Ind_{G_M}^{G_{F_1}}
\theta'')|_{G_{F_{2,\tu}}} \sim r_{l,\imath}(\pi')|_{G_{F_{2,\tu}}}$   and hence
\[  \rho|_{G_{F_{2,\tu}}} \otimes (\Ind_{G_M}^{G_{F_1}}
\theta'')|_{G_{F_{2,\tu}}} \sim (r_{l,\imath}(\pi'_{F_1})\otimes \Ind_{G_M}^{G_{F_1}}\theta)|_{G_{F_{2,\tu}}}\cong
r_{l,\imath}(\pi_2)|_{G_{F_{2,\tu}}}.\]
If $u\ndiv l$, then by definition $\rho|_{G_{F_{2,\tu}}} \sim
r_{l,\imath}(\pi')|_{G_{F_{2,\tu}}}$.
By the choice of $F_2$ we have
$(\Ind_{G_M}^{G_{F_1}}\theta)|_{G_{F_{2,\tu}}}\sim
(\Ind_{G_M}^{G_{F_1}}\theta'')|_{G_{F_{2,\tu}}}$. Hence
\[  \rho|_{G_{F_{2,\tu}}} \otimes (\Ind_{G_M}^{G_{F_1}}
\theta'')|_{G_{F_{2,\tu}}} \sim (r_{l,\imath}(\pi')\otimes \Ind_{G_M}^{G_{F_1}}\theta)|_{G_{F_{2,\tu}}}\cong
r_{l,\imath}(\pi_2)|_{G_{F_{2,\tu}}}\]
and we are done.]
It follows from
Lemma 1.2.2 of \cite{BLGGT} that this map makes $R_{\CS_1}^\univ$ a
finite $R_{\CS_2}^\univ$-module. By Theorem 2.2.2 of \cite{BLGGT}, $R_{\CS_2}^\univ$ is a finite $\CO_L$-module, and hence $R_{\CS_1}^\univ$ is a finite $\CO_L$-module. On the other hand by Proposition
1.5.1 of \cite{BLGGT}, $R_{\CS_1}^\univ$ has Krull dimension at least
$1$. Hence $\Spec R_{\CS_1}^\univ$ has a $\barQQ_l$-point. This point
gives rise to a lifting $r_1:G_{F_1}\to\GL_n(\Qlbar)$ of
$\rbar|_{G_{F_1}}$ with the following properties:
\begin{itemize}
\item $\nu\circ r_1=\tmu \omega_l^{(n-1)m}\epsilon_l^{(1-n)m}$,
\item $r_1$ is unramified outside $S$,
\item if $u|l$ then $r_1|_{G_{F_{1,u}}}\sim 1\oplus\epsilon_l^{-m}\oplus\cdots\oplus\epsilon_l^{(1-n)m}$.
\end{itemize}
By Theorem 2.2.1 of \cite{BLGGT}, Lemma 1.4 of \cite{blght} and the construction of $r_1$, we also
have that
\begin{itemize}
\item $r_1\otimes(\Ind_{G_M}^{G_{F_1}}\theta'')$ is automorphic of
  level prime to $l$.
\end{itemize}
It follows from Lemma 2.1.1 of \cite{BLGGT} that $r_1$ itself is
automorphic of level prime to $l$, say $r_1\cong r_{l,\imath}(\pi_1)$. By the main result of
\cite{ana}, $\pi_1$ is unramified outside of places lying over $S$,
and by the main result of \cite{blggtlocalglobalII} and Lemma 5.2.1 of
\cite{ger}, we see that $\pi_1$ is
$\imath$-ordinary of level prime to $l$. It then follows from Theorems 2.3.1 and 2.3.2 of
\cite{BLGGT}, which together strengthen Theorem 5.1.1 of \cite{gg},
that we may find a RAECSDC automorphic representation
$(\pi_1',\chi_1')$ of $\GL_n(\A_{F_1})$ such that
\begin{itemize}
\item $(\rbar_{l,\imath}(\pi_1'),r_{l,\imath}(\chibar_1'))\cong (\rbar|_{G_{F_1}},\mubar|_{G_{F_1}})$,
\item $\pi_1'$ is $\imath$-ordinary, unramified at places dividing
  $l$, and unramified outside $S$,
\item  if $u|l$ then $r_{l,\imath}(\pi'_1)|_{G_{F_{1,u}}}\sim 1\oplus\epsilon_l^{-m}\oplus\cdots\oplus\epsilon_l^{(1-n)m}$,
\item $r_{l,\imath}(\chi_1')=\tmu|_{G_{F_1^+}} \omega_l^{(n-1)m}\epsilon_l^{(1-n)(m-1)}$,
\item if $\tu\nmid l$ is a place in $\tS_1$ lying over $v\in S$, then
  $r_{l,\imath}(\pi'_1)|_{G_{F_{1,\tu}}}\sim \rho_v|_{G_{F_{1,\tu}}}$.
\end{itemize}

We now argue in a similar fashion to the above to construct the
sought-after representation $r$.

There is a RAECSDC automorphic representation $(\pi'_2,\chi'_2)$ of $\GL_{n^2}(\A_{F_2})$ such
that
\begin{itemize}
\item $r_{l,\imath}(\pi'_2) \cong (r_{l,\imath}(\pi'_1) \otimes \Ind_{G_{M}}^{G_{F_1}} \theta')|_{G_{F_2}}$;
\item $r_{l,\imath}(\chi'_2)=\mu\tmu \omega_l^{(n-1)m}\epsilon_l^{(1-n)(m+n+1)}\delta_{F_2/F_2^+}$;
\item $\pi'_2$ is unramified above $l$ and outside $S$.
\end{itemize}
[The representation $\pi'_2$ is the automorphic induction of
$(\pi'_1)_{MF_2}\otimes (\phi' |\,\,|^{n(n-1)/2}\circ \det)$
to $F_2$, where $r_{l,\imath}(\phi')= \theta'|_{G_{F_2M}}$. The first
two properties are clear.
The third property follows by the choice of $F_2$ and the fact that
$\pi'_1$ is unramified above $l$ and outside $S$.]

For $v \in S_2$ with $v \ndiv l$ (resp. $v|l$) let $\CC'_{2,v}$ denote
the unique component of
$R^\Box_{\barr_{l,\imath}(\pi'_2)|_{G_{F_{2,\tv}}}} \otimes \barQQ_l$
(resp. $R^\Box_{\barr_{l,\imath}(\pi'_2)|_{G_{F_{2,\tv}}},
  \{\HT_\tau(r_{l,\imath}(\pi'_2)|_{G_{F_{2,\tv}}})\}, \cris} \otimes
\barQQ_l$) containing $r_{l,\imath}(\pi'_2)|_{G_{F_{2,\tv}}}$. Extending $L$ if necessary we may further assume that
\begin{itemize}
\item $L$ contains the image of $\mu$;\item $r_{l,\imath}(\pi'_2)$ is defined over $L$;
\item each of the components $\CC_v$ for $v \in S$ and $\CC'_{2,v}$ for $v \in S_2$ is defined over $L$.
\end{itemize}
Set
\[ s' = \Ind_{G_{M^+},G_{M}}^{G_{F_1^+},G_{F_1},\tmu
  \omega_l^{(n-1)m}\epsilon_l^{(1-n)m}} (\theta,\tmu
\omega_l^{(n-1)m}\epsilon_l^{(1-n)m}): G_{F_1^+} \lra \CG_n(\CO_L) \]
in the notation of section 1.1 of this paper and section 2.1 of
\cite{cht}. Thus $\nu \circ s' =\tmu
\omega_l^{(n-1)m}\epsilon_l^{(1-n)m}$. For $v \in S$ (resp. $v \in
S_2$) let $\calD_v$ (resp. $\calD'_{2,v}$) denote the deformation
problem for $\breve{\barr}|_{G_{F_\tv}}$
(resp. $\barr_{l,\imath}(\pi_2)|_{G_{F_{2,\tv}}}$) over $\CO_L$
corresponding to $\CC_v$ (resp. $\CC'_{2,v}$). Also let
\[ \CS=(F/F^+,S,\tS, \CO_L,\barr, \mu, \{ \calD_v\}) \]
and
\[ \CS'_2=(F_2/F_2^+,S_2,\tS_2, \CO_L, \tilde{\barr}_{l,\imath}(\pi'_2), \mu\tmu|_{G_{F_1^+}} \omega_l^{(n-1)m}\epsilon_l^{(1-n)m} \delta_{F_2^+/F_2}, \{
\calD'_{2,v}\}). \]

As above, there is a natural map
\[ R_{\CS'_2}^\univ \lra R_\CS^\univ \] induced by
$r_\CS^\univ|_{G_{F_2^+}} \otimes s'|_{G_{F_2^+}}$. It follows from
Lemma 1.2.2 of \cite{BLGGT} that this map makes $R_{\CS}^\univ$ a
finite $R_{\CS'_2}^\univ$-module. By Theorem 2.2.2 of \cite{BLGGT},
$R_{\CS'_2}^\univ$ is a finite $\CO_L$-module, and hence
$R_{\CS}^\univ$ is a finite $\CO_L$-module. On the other hand by
Proposition 1.5.1 of \cite{BLGGT}, $R_\CS^\univ$ has Krull dimension
at least $1$. Hence $\Spec R_{\CS}^\univ$ has a
$\barQQ_l$-point. This point gives rise to the desired lifting $r$ of
$\barr$. [To see that $r|_{G_{F'}}$ is automorphic, note that by
Theorem 2.2.1 of \cite{BLGGT},
$(r|_{G_{F_1}}\otimes(\Ind_{G_M}^{G_{F_1}}\theta))|_{G_{F_2}}$ is
automorphic, so by Lemma 1.4 of \cite{blght}
$r|_{G_{F_1}}\otimes(\Ind_{G_M}^{G_{F_1}}\theta)$ is automorphic. It
follows from Lemma 2.1.1 of \cite{BLGGT} that $r|_{G_{F_1}}$ is automorphic, and a
further application of Lemma 1.4 of \cite{blght} shows that
$r|_{G_{F'}}$ is automorphic, as required.]
\end{proof}

\bibliographystyle{amsalpha}
\bibliography{barnetlambgeegeraghty}

\newcommand{\etalchar}[1]{$^{#1}$}
\providecommand{\bysame}{\leavevmode\hbox to3em{\hrulefill}\thinspace}
\providecommand{\MR}{\relax\ifhmode\unskip\space\fi MR }
\providecommand{\MRhref}[2]{%
  \href{http://www.ams.org/mathscinet-getitem?mr=#1}{#2}
}
\providecommand{\href}[2]{#2}
\begin{thebibliography}{BLGGT11}

\bibitem[AOT37]{asano-osima-takahasi}
K.~Asano, M.~Osima, and M.~Takahasi, \emph{{\"Uber die Darstellung von Gruppen
  durch Kollineationen im K\"orper der Charakteristik $p$}}, Proc. Phys. Math.
  Soc. Japan \textbf{19} (1937), no.~1, 199--209.

\bibitem[BC09]{belchen}
Joel Bellaiche and Ga{\"e}tan Chenevier, \emph{The sign of the {G}alois
  representations attached to automorphic forms for unitary groups}, preprint,
  2009.

\bibitem[BDJ10]{bdj}
Kevin Buzzard, Fred Diamond, and Frazer Jarvis, \emph{On {S}erre's conjecture
  for mod {$\ell$} {G}alois representations over totally real fields}, Duke
  Math. J. \textbf{155} (2010), no.~1, 105--161. \MR{2730374}

\bibitem[BLGG10]{blggord}
Tom Barnet-Lamb, Toby Gee, and David Geraghty, \emph{Congruences between
  {H}ilbert modular forms: constructing ordinary lifts in parallel weight two},
  Preprint, 2010.

\bibitem[BLGG11]{blgg}
\bysame, \emph{The {S}ato-{T}ate {C}onjecture for {H}ilbert {M}odular {F}orms},
  J. Amer. Math. Soc. \textbf{24} (2011), no.~2, 411--469.

\bibitem[BLGGT10]{BLGGT}
Tom Barnet-Lamb, Toby Gee, David Geraghty, and Richard Taylor, \emph{Potential
  automorphy and change of weight}, Preprint, 2010.

\bibitem[BLGGT11]{blggtlocalglobalII}
Thomas Barnet-Lamb, Toby Gee, David Geraghty, and Richard Taylor,
  \emph{Local-global compatibility for $l=p$, {II}.}, 2011.

\bibitem[BLGHT09]{blght}
Tom Barnet-Lamb, David Geraghty, Michael Harris, and Richard Taylor, \emph{A
  family of {C}alabi-{Y}au varieties and potential automorphy {II}}, Preprint,
  2009.

\bibitem[Car10]{ana}
Ana Caraiani, \emph{Local-global compatibility and the action of monodromy on
  nearby cycles}, preprint arXiv:1010.2188, 2010.

\bibitem[CCN{\etalchar{+}}]{conway-atlas}
JH~Conway, RT~Curtis, SP~Norton, RA~Parker, and RA~Wilson, \emph{{Atlas of
  finite groups}}, Bulletin of the London Mathematical Society.

\bibitem[CH09]{chenevierharris}
Ga{\"e}tan Chenevier and Michael Harris, \emph{Construction of automorphic
  {G}alois representations, {II}}, preprint, 2009.

\bibitem[CHT08]{cht}
Laurent Clozel, Michael Harris, and Richard Taylor, \emph{Automorphy for some
  $l$-adic lifts of automorphic mod $l$ {G}alois representations}, Pub. Math.
  IHES \textbf{108} (2008), 1--181.

\bibitem[Clo90]{MR1044819}
Laurent Clozel, \emph{Motifs et formes automorphes: applications du principe de
  fonctorialit\'e}, Automorphic forms, Shimura varieties, and $L$-functions,
  Vol.\ I (Ann Arbor, MI, 1988), Perspect. Math., vol.~10, Academic Press,
  Boston, MA, 1990, pp.~77--159. \MR{MR1044819 (91k:11042)}

\bibitem[CPS75]{cps}
E.~Cline, B.~Parshall, and L.~Scott, \emph{{Cohomology of finite groups of Lie
  type, I}}, Publications Math{\'e}matiques de l'IH{\'E}S \textbf{45} (1975),
  no.~1, 169--191.

\bibitem[DDT95]{ddt}
H.~Darmon, F.~Diamond, and R.~Taylor, \emph{{FermatÕs last theorem}}, Current
  developments in mathematics \textbf{1} (1995), 1--157.

\bibitem[Gee06]{MR2280776}
Toby Gee, \emph{A modularity lifting theorem for weight two {H}ilbert modular
  forms}, Math. Res. Lett. \textbf{13} (2006), no.~5-6, 805--811. \MR{MR2280776
  (2007m:11065)}

\bibitem[Gee10a]{gee061}
\bysame, \emph{Automorphic lifts of prescribed types}, Math. Annalen (to
  appear) (2010).

\bibitem[Gee10b]{geebdj}
\bysame, \emph{On the weights of mod $p$ {H}ilbert modular forms}, Invent.
  Math. (to appear) (2010).

\bibitem[Ger09]{ger}
David Geraghty, \emph{Modularity lifting theorems for ordinary {G}alois
  representations}, preprint, 2009.

\bibitem[GG09]{gg}
Toby Gee and David Geraghty, \emph{Companion forms for unitary and symplectic
  groups}, Preprint, 2009.

\bibitem[GHS11]{GHS}
Toby Gee, Florian Herzig, and David Savitt, \emph{Explicit {S}erre weight
  conjectures}, in preparation, 2011.

\bibitem[GHTT10]{jackapp}
Robert Guralnick, Florian Herzig, Richard Taylor, and Jack Thorne,
  \emph{Adequate subgroups}, Appendix to \cite{jack}, 2010.

\bibitem[GLS11]{geeliusavitt}
Toby Gee, Tong Liu, and David Savitt, \emph{Crystalline extensions and the
  weight part of {S}erre's conjecture}, preprint, 2011.

\bibitem[GS10]{geesavitttotallyramified}
Toby Gee and David Savitt, \emph{Serre weights for mod $p$ {H}ilbert modular
  forms: the totally ramified case}, J. Reine Angew. Math.(to appear) (2010).

\bibitem[HH92]{hoffman1992projective}
P.N. Hoffman and JF~Humphreys, \emph{{Projective representations of the
  symmetric groups: Q-functions and shifted tableaux}}, Oxford University
  Press, USA, 1992.

\bibitem[HT01]{ht}
Michael Harris and Richard Taylor, \emph{The geometry and cohomology of some
  simple {S}himura varieties}, Annals of Mathematics Studies, vol. 151,
  Princeton University Press, Princeton, NJ, 2001, With an appendix by Vladimir
  G. Berkovich. \MR{MR1876802 (2002m:11050)}

\bibitem[JLPW95]{jansen1995atlas}
C.~Jansen, K.~Lux, R.~Parker, and R.~Wilson, \emph{{An atlas of Brauer
  characters}}, vol.~19, Clarendon Press, 1995.

\bibitem[Kis07]{kis04}
Mark Kisin, \emph{Moduli of finite flat group schemes, and modularity}, to
  appear in Annals of Mathematics (2007).

\bibitem[Kis08]{kisindefrings}
\bysame, \emph{Potentially semi-stable deformation rings}, J. Amer. Math. Soc.
  \textbf{21} (2008), no.~2, 513--546. \MR{MR2373358 (2009c:11194)}

\bibitem[Kis10]{kisinICM}
\bysame, \emph{The structure of potentially semi-stable deformation rings},
  preprint, 2010.

\bibitem[Lab09]{labesse}
Jean-Pierre Labesse, \emph{Changement de base {C}{M} et s\'eries discr\`etes},
  preprint, 2009.

\bibitem[Sch08]{MR2430440}
Michael~M. Schein, \emph{Weights in {S}erre's conjecture for {H}ilbert modular
  forms: the ramified case}, Israel J. Math. \textbf{166} (2008), 369--391.
  \MR{MR2430440 (2009e:11090)}

\bibitem[Ser77]{MR0450380}
Jean-Pierre Serre, \emph{Linear representations of finite groups},
  Springer-Verlag, New York, 1977, Translated from the second French edition by
  Leonard L. Scott, Graduate Texts in Mathematics, Vol. 42. \MR{0450380 (56
  \#8675)}

\bibitem[Ser79]{MR554237}
\bysame, \emph{Local fields}, Graduate Texts in Mathematics, vol.~67,
  Springer-Verlag, New York, 1979, Translated from the French by Marvin Jay
  Greenberg. \MR{MR554237 (82e:12016)}

\bibitem[SW10]{snw}
A.~Snowden and A.~Wiles, \emph{{Bigness in compatible systems}}, Arxiv preprint
  arXiv:0908.1991 (2010).

\bibitem[Tho10]{jack}
Jack Thorne, \emph{On the automorphy of $l$-adic {G}alois representations with
  small residual image}, preprint available at
  \verb+http://www.math.harvard.edu/~thorne+, 2010.

\bibitem[TY07]{ty}
Richard Taylor and Teruyoshi Yoshida, \emph{Compatibility of local and global
  {L}anglands correspondences}, J. Amer. Math. Soc. \textbf{20} (2007), no.~2,
  467--493 (electronic). \MR{MR2276777 (2007k:11193)}

\bibitem[Wil09]{wilsonfinitesimple}
R.~Wilson, \emph{{The {F}inite {S}imple {G}roups: {A}n {I}ntroduction}},
  Springer, 2009.

\bibitem[Zhu08]{0807.1078}
Hui~June Zhu, \emph{Crystalline representations of ${G}_{\Q_{p^a}}$ with
  coefficients}, 2008.

\end{thebibliography}

\end{document}